\documentclass[a4paper,12pt,reqno]{amsart}

\usepackage[T1]{fontenc}
\usepackage[utf8]{inputenc}
\usepackage{amsmath}
\usepackage{amssymb}
\usepackage{amsthm}
\usepackage[style=alphabetic]{biblatex}
\usepackage{comment}
\usepackage{enumitem}
\usepackage{geometry}
\usepackage{graphicx} 
\usepackage{hyperref}
\usepackage{libertine}
\usepackage{mathtools}
\usepackage{ulem}
\usepackage[dvipsnames]{xcolor}
\usepackage{cleveref}

\DeclareSourcemap{
  \maps[datatype=bibtex]{
    \map{
      \step[fieldset=issn, null]
      }
    }
  }

\numberwithin{equation}{section}

\theoremstyle{plain}
\newtheorem{theorem}{Theorem}[section]
\newtheorem{lemma}[theorem]{Lemma}
\newtheorem{proposition}[theorem]{Proposition}
\newtheorem{corollary}[theorem]{Corollary}
\newtheorem{conjecture}[theorem]{Conjecture}
\newtheorem{assumption}{Assumption}

\theoremstyle{remark}

\newcommand{\E}{{\mathbb{E}}}

\renewcommand{\P}{{\mathbb{P}}}
\newcommand{\Q}{{\mathbb{Q}}}
\newcommand{\R}{{\mathbb{R}}}
\newcommand{\Z}{{\mathbb{Z}}}

\newcommand{\cF}{{\mathcal F}}

\newcommand*{\ges}{\geqslant}
\newcommand*{\les}{\leqslant}
\let\leq\leqslant
\let\geq\geqslant
\let\le\leqslant
\let\ge\geqslant

\newcommand*{\indf}{1\negmedspace\mathrm{I}}
\newcommand*{\lm}{\mu_{\rm min}} 
\newcommand*{\lM}{\mu_{\rm max}}
\newcommand*{\On}{\mathrm{On}}
\newcommand*{\TAP}{\mathrm{TAP}}
\newcommand*{\Es}{E} 

\newcommand*{\Var}{\mathrm{Var}}

\newcommand*{\e}{\epsilon}
\newcommand*{\vf}{\varphi}
\newcommand*{\mbP}{\mathbb{P}}
\newcommand*{\mbR}{\mathbb{R}}

\newcommand{\supp}{\textup{supp}}

\newcommand{\Flim}[1]{F_{#1}}
\newcommand{\Ffin}[1]{\hat{F}_{#1}}
\newcommand{\logpotential}[1]{L_{#1}}
\newcommand{\Lstar}[1]{L^\ast_{#1}}

\DeclarePairedDelimiter{\abs}{\lvert}{\rvert}
\DeclarePairedDelimiter{\norm}{\lVert}{\rVert}

\bibliography{main}

\newcommand{\simon}[1]{\textcolor{pink}{[\textbf{Simon:} #1]}}

\newcommand{\DB}[1]{{\color{red}\textsl{DB: #1}}}

\usepackage{array}
\makeatletter
\def\dsarray@name{array}
\renewcommand{\array}{\col@sep\arraycolsep
  \def\d@llarbegin{$\ifx\@currenvir\dsarray@name\displaystyle\fi}%
  \def\d@llarend{$}%
  \def\@halignto{}%
  \@tabarray}
\makeatother
  
\title[TAP Formula for Spherical 2-Spin Models]{TAP variational formula for $2$-spin models with general interaction matrix and marginality}

\author{David Belius$^*$, Simon Briend, Volodymyr Fomichov$^*$} 
\email{david.belius@cantab.net}
\email{simon.briend@unidistance.ch}
\email{volodymyr.fomichov@unidistance.ch}

\thanks{$^*$ D. B. and V. F. were supported by the SNSF grants 176918 and 206148.}
\address{Faculty of Mathematics and Computer Science, UniDistance Suisse, 3900 Brig, Switzerland}

\begin{document}

\begin{abstract}
    We establish a Thouless--Anderson--Palmer (TAP) variational formula for the free energy of spherical $2$-spin models whose interaction matrices have a compactly supported limiting spectral distribution. The formula holds at every inverse temperature and, in particular, extends the TAP representation previously obtained for the spherical Sherrington--Kirkpatrick model.

    Our main contribution is a new microcanonical proof designed to handle marginality. We partition the sphere according to the smallest collection of top eigendirections carrying a prescribed fraction of the norm. On every element of this partition, the remaining degrees of freedom are effectively at high temperature and their contribution can be computed through spherical-integral asymptotics. The resulting auxiliary variational problem separates the macroscopic low-temperature contribution from the high-temperature fluctuation term; its maximization recovers the usual TAP functional and Plefka condition. This decomposition also shows why neighborhoods of the TAP maximizers alone do not capture the Gibbs mass: all eigendirections associated with large eigenvalues contribute on the same exponential scale, a manifestation of the marginality of spherical $2$-spin states.
\end{abstract}

\maketitle

\section{Introduction}

Mean-field spin glasses are models of large interacting systems in which disorder and frustration produce complex collective behavior. Originating in the study of disordered magnetic materials, they now also play an important role in probability, theoretical computer science, and high-dimensional statistics; see, for example, \cite{mezard1987spin,talagrand2010mean,panchenko2013sherrington}. A central problem is to determine the limiting free energy and to understand which regions of the configuration space contribute to it.

The Thouless--Anderson--Palmer (TAP) approach addresses both questions through an auxiliary variational functional on magnetization vectors \cite{thouless1977solution}. At a heuristic level, the TAP functional separates two contributions. A magnetization $m$ records the macroscopic part of a configuration, while fluctuations around $m$ form a residual system that should be effectively at high temperature. The TAP correction, or Onsager term, is the free energy of this residual system. A TAP formula therefore does more than give the value of the free energy: it expresses that value through a decomposition of the underlying configurations into  two parts, respectively at low effective temperature and high effective temperature.

Turning this principle into a rigorous proof is delicate. Several versions of the TAP representation have been developed in recent years. Generalized TAP representations for mixed $p$-spin models were established in \cite{MR4588318}, while a TAP-based computation of the free energy of spherical pure $p$-spin models was obtained in \cite{subag2023free}.
A geometric, microcanonical TAP upper bound, based on a covering of the configuration space, was developed in \cite{belius2022high}. For the spherical Sherrington--Kirkpatrick model with a linear external field, a TAP variational principle valid at every temperature was proven in \cite{BK}.

The purpose of the present paper is to establish such a formula for spherical $2$-spin models with a general limiting spectral distribution, and to do so by a different, microcanonical argument. The main difficulty is marginality: at low temperature, many directions contribute to the partition function on the same exponential scale as the top eigendirection. This is why a decomposition using only neighborhoods of the TAP maximizers does not capture all the configurations contributing to the free energy.

\subsection{The model and the TAP formula}

Let $M_N$ be a real symmetric $N\times N$ matrix and let
\begin{equation}\label{eq:def_eigenvalues_MN}
    \lambda_1\les \cdots\les \lambda_N
\end{equation}
be its eigenvalues. We denote by $B_N$ the closed unit ball, by $B_N^\circ$ its interior, and by $S^{N-1}$ the unit sphere,
\begin{equation}
  S^{N-1}=\{\sigma\in\mathbb R^N:|\sigma|=1\},
\end{equation}
For $\sigma\in B_N$ we consider the Hamiltonian
\begin{equation}\label{eq:def_HN}
  H_N(\sigma)=N\sigma^{\mathsf T}M_N\sigma
\end{equation}
and the free energy
\begin{equation}\label{eq:def_FE_no_ext_field}
    F_N(\beta)
    =
    \frac1N
    \log
    \Es \left[
        \exp(
            \beta H_N(\sigma)
        )
    \right]
    =
    \frac1N
    \log Z_N(\beta)
    ,
\end{equation}
where $\Es$ denotes expectation with respect to the uniform probability measure on $S^{N-1}$. More generally, we consider
\begin{equation}\label{eq:def_HN_ext_field}
    H_N^f(\sigma)=H_N(\sigma)+Nf_N(\sigma)
\end{equation}
and the corresponding free energy
\begin{equation}\label{eq:def_FE_with_ext_field}
    F^f_N(\beta)
    \coloneqq
    \frac{1}{N}
    \log
    \Es[ \exp(\beta H^f_N(\sigma))]
    =
    \frac{1}{N}\log Z_N^f(\beta),
\end{equation}
for the class of generalized external fields specified in Theorem~\ref{thm:main_with_ext_field}.

All our results hold under the following assumption.
\begin{assumption}\label{main_assumption}
    \mbox{}
    \begin{enumerate}[label=(\arabic*)]
        \item\par
        \noindent The empirical spectral distribution $\mu_N=\frac{1}{N}\sum_{i=1}^N \delta_{\lambda_i}$ converges weakly to a probability measure $\mu$.
        
        \item
        $\mu$ is compactly supported on an interval $I = [\lm,\lM]$, and has a density with respect to Lebesgue measure which is positive on the interior of this interval.

        \item The minimal and maximal eigenvalues converge to $\lm$ and $\lM$ respectively, that is $\lambda_1 \to \lm$ and $\lambda_N \to \lM$ as $N\to\infty$.        
        
        \item The Stieltjes transform of $\mu$ at the right end-point of the support is finite, i.e.
        \begin{equation}\label{eq:def_smu}
            s_\mu \coloneqq S_\mu(\lM)<\infty,
        \end{equation}
        where
        \begin{equation}\label{eq:def_stieltjes_transform}
            S_\mu(\lambda)
            \coloneqq
            \int_I \frac{\mu(dx)}{\lambda-x},\quad\lambda\notin (\lm,\lM).
        \end{equation}
        
    \end{enumerate}
\end{assumption}
\noindent The critical inverse temperature is defined from $s_\mu$,
\begin{equation}\label{eq:def_beta_c}
      \beta_c=\frac{s_\mu}{2}.
\end{equation}
This setting includes, in particular, the Sherrington--Kirkpatrick model, where the Hamiltonian is given by $N \sigma^T G \sigma$ for $G$ a GOE matrix. There, $\mu$ is the semicircle law on $[-\sqrt{2},\sqrt{2}]$ and $s_\mu = \sqrt{2}$. To define the Onsager term, introduce the log-potential
\begin{equation}\label{eq:def_log_potential}
    L_\mu(x)=\frac12\int \log(x-y)\,\mu(dy),\qquad x\geq \lM,
\end{equation}
and its Legendre transform 
\begin{equation}\label{eq:LegendreLstar}
      L_\mu^*(z)=\sup_{x\geq\lM}\{L_\mu(x)-zx\},\qquad z>0.
\end{equation}
The high-temperature limiting free energy associated with $\mu$ is
\begin{equation}\label{eq:F_def}
  F_\mu(z)=-L_\mu^*(z)-\frac12\log(2z)-\frac12, \qquad z>0,
\end{equation}
with $F_\mu(0)=0$. Accordingly, the Onsager correction at overlap $q$ is
\begin{equation}\label{eq:OnsagerDefinition}
    \operatorname{On}(\beta,q)=F_\mu\bigl(\beta(1-q)\bigr),
\end{equation}
and the TAP functional is
\begin{equation}\label{eq:FTAP_first_def}
  F_{\mathrm{TAP}}(m)
  =\frac{\beta}{N}H_N^f(m)
   +\frac12\log(1-|m|^2)
   +\operatorname{On}(\beta,|m|^2),
  \qquad m\in B_N^\circ,
\end{equation}
Here and below, the normalization of $H_N^f(m)$ is chosen so that every term in \eqref{eq:FTAP_first_def} is of order one. We also define the TAP free energy as a function of Hamiltonian energy per spin $E$ and norm squared $q$ as
\begin{equation}\label{eq:ftapDefinition}
    f_\TAP(\beta,E,q)
    =
    \beta E
    +
    \frac{1}{2}\log(1-q)
    +
    \On(\beta,q).
\end{equation}
Note that the usual TAP energy as a function of $m$ in \eqref{eq:FTAP_first_def} then equals
\begin{equation}\label{eq:FtapDefinition}
    F_{\rm TAP}(m)
    =
    f_\TAP
    \left(
        \beta,
        N^{-1}H^f_N(m),
        |m|^2
    \right)
\end{equation}

\subsection{Main contribution}

Our main result states that, under the assumptions above,
\begin{equation}\label{eq:intro_TAP_FE_formula}
  F_N^f(\beta)
  =\sup_{\substack{m\in B_N:\\\beta(1-|m|^2)\les\beta_c}}
       F_{\mathrm{TAP}}(m)+o(1).
\end{equation}
When $f_N=0$, the variational problem reduces to one scalar parameter. Setting $q^*(0)=0$ and for $\beta>0$
\begin{equation}\label{eq:def_q_star}
    q^*
    :=
    q^*(\beta)
    =
    \left(1-\frac{\beta_c}{\beta}\right)_+,
\end{equation}
we obtain
\begin{equation}\label{eq:intro-zero-field-result-revised}
  F_N(\beta)
  =\sup_{\substack{q\in[0,1):\\ \beta(1-q)\les\beta_c}}
    \left\{\beta q\lM
      +\frac12\log(1-q)+F_\mu\bigl(\beta(1-q)\bigr)\right\}+o(1),
\end{equation}
which is the statement of our main theorem.

\begin{theorem}\label{thm:main}
    Let $M_N$ be a sequence of $N\times N$ symmetric matrices such that \Cref{main_assumption} holds. Then for any $B>0$, it holds uniformly for any $0\les \beta \les B$ that as $N\to\infty$
    \begin{equation}
        F_N(\beta)
        =
        \sup_{
        q \in [0,1)
        :
        \;
        \beta(1-q)\les\beta_c
        }
        f_\TAP(\beta,q \lM,q)
        +o(1)
    \end{equation}
    and the supremum is attained at $q=q_*(\beta)$.
\end{theorem}
The constraint
\begin{equation}
    \beta(1-|m|^2)\les\beta_c
\end{equation}
is the analogue of Plefka's condition (see \cite{MR663708}): it says precisely that the system remaining after the magnetization $m$ has been extracted is at or below its critical inverse temperature. Note that, if we are already at low inverse temperature ($\beta\les \beta_c$) then $q^*=0$ and the supremum is reached at $q=0$, $m=0$ in the absence of an external field. The complete result, with external field, is stated in Theorem~\ref{thm:main_with_ext_field}.

For the semicircle law, \eqref{eq:intro_TAP_FE_formula} becomes the TAP formula of \cite{BK}: the Onsager correction is $\frac12\beta^2(1-q)^2$ and the Plefka condition takes its familiar form. In the absence of an external field, the limiting free energy can also be obtained directly from the asymptotics of rank-one spherical integrals established by Guionnet and Maïda \cite{MR2132396}. Our contribution is a new microcanonical TAP proof that describes how the different regions of the configuration space contribute to this limit. Finally, we build on a partition of the sphere introduced in \cite{belius2022high} to deal with the case with external field.

\begin{theorem}\label{thm:main_with_ext_field}
     Let $M_N$ be a sequence of $N\times N$ symmetric matrices such that \Cref{main_assumption} holds. Let $U_N$ be a sequence of subspaces of $\mathbb{R}^N$ such that $\dim(U_N) = O(1)$, and let $f=f_N$ be a sequence of functions $f: B_N \to \mathbb{R}$ such that $f(\sigma) = f(P^{U_N}\sigma)$ for all $\sigma \in B_N$. Assume furthermore that $f$ has finite Lipschitz constant $c$, uniformly in $N$. Then for any $\beta>0$, as $N\to \infty$
    \begin{equation}    
            F^f_N(\beta)
            =
            \sup_{m \in B_N \colon\beta(1-|m|^2) \les\beta_c} F^f_{{\rm TAP}}(m)
        +o(1)
    \end{equation}
\end{theorem}

\subsection{The high-temperature/low-temperature decomposition}

We next explain the mechanism behind the proof. Since the reference measure on the sphere is invariant under orthogonal transformations, we may diagonalize $M_N$ and write
\begin{equation}\label{eq:intro-diagonal-revised}
  H_N(\sigma)=N\sum_{i=1}^N\lambda_i\sigma_i^2.
\end{equation}
At high temperature, no single coordinate is macroscopic: the leading configurations have $\sigma_i^2$ of order $N^{-1}$ throughout the spectrum and the magnetization associated is $m=0$. The free energy can then be computed from the limiting spectral distribution, giving $F_\mu(\beta)$ for $\beta\les\beta_c$.

At low temperature, by contrast, a positive fraction $q$ of the norm is carried by directions whose eigenvalues are close to $\lM$, hence the magnetization associated is expected to be $\sqrt{q}e_N$. Once this fraction is removed, the remaining norm is $1-q$ and the residual system has effective inverse temperature $\beta(1-q)$. Choosing $q\geq q_*(\beta)$ makes this residual system high-temperature. This is the TAP decomposition in its simplest form:
\begin{equation}
  \underbrace{\beta q\lM}_{\text{macroscopic energy}}
  +\underbrace{\frac12\log(1-q)}_{\text{entropy cost}}
  +\underbrace{F_\mu(\beta(1-q))}_{\text{high-temperature fluctuations}}.
\end{equation}
The three terms are exactly those appearing in \eqref{eq:intro-zero-field-result-revised}.


A natural geometric interpretation of the TAP functional associates a magnetization vector $m$ with a \textit{slice} $A$ of configurations, centered at $m$ and satisfying $A=\{\sigma\in S^{N-1}; \ \langle \sigma,m\rangle\approx|m|^2\}$. Within this slice, a configuration decomposes approximately as $\sigma=m+\widehat{\sigma}$, where $\widehat{\sigma}$ is orthogonal to $m$ and has squared norm $1-|m|^2$. The TAP functional describes the contribution of this slice through the energy of $m$ (that is $N^{-1} \beta H_N^f(m)$), the entropy cost of fixing its projection (that is $\tfrac{1}{2}\log (1-q)$), and the free energy of the remaining fluctuations (that is $\Flim{\mu}(\beta(1-q))$), and so, for $F_N(\beta,A)$ the free energy restricted to $A$

\begin{equation}\label{eq:intro_decompo_explained_A}
    F_N(\beta,A)
    \approx 
    \beta N^{-1} H_N^f(m) +\frac{1}{2}\log(1-q) +\Flim{\mu}\big(\beta(1-q)\big)
    =
    F_{\TAP}(m).
\end{equation}
By covering the $S^{N-1}$ by a collection of slices $A_1,\dots,A_K$, it is direct that

\begin{equation}
    \sup_{i} F_N(\beta,A_i)
    \les 
    F_N(\beta)
    \les
    \sup_{i} F_N(\beta,A_i)
    +
    \frac{1}{N}\log K,
\end{equation}
which together with \eqref{eq:intro_decompo_explained_A} hints at the results of Theorem \ref{thm:main} and \ref{thm:main_with_ext_field}. Further, in the zero-field model, the usual candidates for slices maximizing the TAP free energy at low temperature are the slices associated with $m=\pm\sqrt{q^*}\,e_N$, where $e_N$ is a top eigenvector.

Marginality creates an obstruction to obtaining an exhaustive decomposition from these bands alone. For each fixed $K$, the eigenvalue $\lambda_{N-K}$ also converges to $\lM$. Consequently, bands associated with $\pm\sqrt{q^*}\,e_{N-K}$ have the same limiting restricted free energy as those associated with the top eigendirection, although they describe different regions of the sphere. Thus, at exponential scale, configurations outside the bands associated with the TAP maximizers contribute to the total free energy. 

We generalize the notion of a TAP neighborhood by specifying a spectral mass profile rather than a single magnetization vector. For each configuration, we expose the eigendirections in decreasing order of their eigenvalues until their cumulative squared norm reaches $q^*$. We then record the squared norms carried by the coordinates above and below the threshold coordinate. This produces sets on which the distribution of the norm between the two spectral blocks is approximately fixed, while the directions within each block remain free. When the threshold is reached at the top coordinate, these sets recover, up to the choice of sign, the usual bands associated with the top eigendirection. Allowing the threshold to occur farther into the spectrum extends this construction to configurations whose norm is distributed over several eigendirections. The resulting sets cover the sphere and lead to an auxiliary variational problem whose optimum recovers the usual TAP value.



For each $\sigma$, let $K(\sigma)$ be the smallest integer such that the last $K(\sigma)+1$ coordinates carry at least the critical norm $q_*(\beta)$:
\begin{equation}\label{eq:intro-K-revised}
  K(\sigma)=\min\left\{K\in\{0,\ldots,N-1\}:
    \sum_{i=N-K}^N\sigma_i^2\geq q_*(\beta)\right\}.
\end{equation}
We then refine the resulting partition according to the norm below and above the threshold coordinate. If $K/N\to\alpha$, these pieces distinguish three quantities: the proportion $\alpha$ of top eigendirections that have been exposed, the norm $\tau_R$ carried by those directions, and the remaining norm $\tau_L$. Conditional on these quantities, the two groups of coordinates are uniform on lower-dimensional spheres. Their free energies can therefore be estimated through truncated spectral measures, while the measure of the piece contributes an explicit relative-entropy cost.

This produces an auxiliary functional
\begin{equation}
    M_{\mathrm{TAP}}(\beta,\tau_L,\tau_R,\alpha),
\end{equation}
defined precisely in Section~\ref{sec:partition}. It retains the full range of marginal near-edge configurations: the standard slice around $e_N$ is the boundary case $\alpha=0$, whereas $\alpha>0$ accounts for configurations spread over a growing family of top eigendirections. The key deterministic optimization in the proof shows that
\begin{equation}\label{eq:intro-aux-collapse-revised}
  \sup_{\alpha,\tau_L,\tau_R}
   M_{\mathrm{TAP}}(\beta,\tau_L,\tau_R,\alpha)
  = M_{\mathrm{TAP}}(\beta,1-q_*,0,0)
  =f_\TAP(\beta,\lM q^\ast,q^\ast)
\end{equation}
and the right-hand side is exactly the usual TAP value at $q_*$. In other words, marginal configurations must be included in the decomposition, but they do not increase the final variational value.

This distinction is useful. The value of the free energy alone can be obtained from spherical-integral methods. Our partition explains, on the level of configurations, how the free energy separates into low-temperature spectral concentration and high-temperature fluctuations, and why marginal directions cannot be discarded, even at exponential scale.

\subsection{Generalized external fields}

To pass from the zero-field model to $H_N^f$, we combine the preceding partition with the construction of another partition of \cite{belius2022high}. The sphere is covered by a subexponential collection of thickened affine sections on which the external field is nearly constant and the linear term created by recentering the quadratic Hamiltonian vanishes. If $m$ is the center of such a section, its orthogonal fluctuations lie on a sphere of radius $\sqrt{1-|m|^2}$ and hence evolve at effective inverse temperature $\beta(1-|m|^2)$. Applying the zero-field decomposition inside each section gives the Onsager term in \eqref{eq:intro_TAP_FE_formula}. A final optimization shows that sections violating Plefka's condition cannot improve the supremum, which yields \eqref{eq:intro_TAP_FE_formula}.

Although the proof uses the diagonalizability of spherical $2$-spin Hamiltonians and therefore does not directly extend to general mixed $p$-spin models, the obstruction it isolates is not specific to the calculation of spherical integrals. Any microcanonical TAP approach that aims to control marginal states must account for families of configurations that are distinct from the TAP maximizers but indistinguishable from them at the leading exponential scale. The cumulative decomposition above gives one concrete way to do so in a setting where the resulting variational problem remains tractable.

\subsection{Organization of the paper}

Section \ref{sec:legendre_facts} collects properties of the log-potential $L_\mu$, its Legendre transform, and the high-temperature free energy $F_\mu$. Section \ref{sec:Legendre_FE_high_temp} proves the high-temperature estimates, uniformly enough to apply them to truncated empirical spectral distributions. Section \ref{sec:partition} introduces the cumulative top-eigenspace partition and bounds the contribution of each of its elements. Section \ref{sec:intermediate_TAP} passes to the limiting auxiliary functional $M_{\mathrm{TAP}}$. Section \ref{sec:where_is_TAP_max} solves the resulting variational problem and proves the zero-field theorem. Finally, Section \ref{sec:external_field_elimination} combines the spectral decomposition with the affine-slice construction to treat generalized external fields.

\newpage

\section{Preliminary results on the Legendre transform}
\label{sec:legendre_facts}

In this short section we gather some useful results about the Legendre transform defined in \eqref{eq:LegendreLstar}.

\begin{lemma}\label{lem:Legendre_trans_trivial_bounds}
     For any probability distribution $\mu$ with bounded support
    \begin{equation}
        -z\lM\le L^{\ast}_{\mu}(z)-\left(-\frac{1}{2}\log(2z)-\frac{1}{2}\right)\le-z\lm
        \qquad
        \forall z>0.
    \end{equation}
\end{lemma}

\begin{proof}
    Letting $l_m(x)=\tfrac{1}{2}\log(x - m)$ we have $l_{\lM}(x)\leq L_\mu(x) \leq l_{\lm}(x)$ for $x \ge \lM$. Thus $l^\ast_{\lM}(z)\leq L^\ast_\mu(z) \leq l_{\lm}^\ast(z)$. Lastly $l^\ast_m(z) = -m z - \frac12 \log(2z) - \frac12$.
\end{proof}

\begin{corollary}\label{cor:free_energy_trivial_bounds}
     For any probability distribution $\mu$ with bounded support and any $\beta > 0$
    \begin{equation}
        \beta\lm
        \leq
        \Flim{\mu}(\beta)
        \leq
        \beta \lM.
    \end{equation}
    In particular, $F_\mu(\beta)$ is continuous at $0$.
\end{corollary}

\begin{proof}
    Follows directly from \Cref{lem:Legendre_trans_trivial_bounds} and \eqref{eq:F_def}.
\end{proof}

\begin{lemma}\label{lem:xstar_legendre}
    Let $\mu$ be a probability distribution with bounded support. Then for any $z>0$ there is a unique maximizer $x^*(z)$ of the supremum in \eqref{eq:LegendreLstar}, and it satisfies
    \begin{equation}
        \left(\frac{1}{2z}+\lm\right)\vee \lM
        \leq
        x^*(z)
        \leq
        \frac{1}{2z}+\lM
        \qquad
        \forall z> 0.
    \end{equation}
\end{lemma}

\begin{proof}
    The function $x\mapsto \logpotential{\mu}(x)-zx$ is strictly concave, which implies that there is a unique maximizer. It is also differentiable on $(\lM,\infty)$, with $\tfrac{\partial}{\partial x}\left\{\logpotential{\mu}(x)-zx\right\} = \frac12 S_\mu(x)-z$. If $s_\mu < \infty$ this also holds at $x=\lM$. If $s_\mu<\infty$ and $\tfrac{\partial}{\partial x}\left\{\logpotential{\mu}(x)-zx\right\}|_{x=\lM}\leq 0$, then $x^*(z) = \lM$, so the upper bound holds trivially. The lower bound also holds since $z \geq \frac12 s_{\mu}\geq \tfrac{1}{2}\tfrac{1}{\lM-\lm}$. On the other hand, if $s_\mu<\infty$ and $\tfrac{\partial}{\partial x}\left\{\logpotential{\mu}(x)-zx\right\}|_{x=\lM}>0 $, or if $s_\mu=\infty$, then $x^*(z)>\lM$ and $\tfrac{\partial}{\partial x}\left\{\logpotential{\mu}(x)-zx\right\}|_{x=x^*}=0$. The claim then follows since for $x > \lM$
    \begin{equation}
        0
        =
        \frac{\partial}{\partial x}\left\{\logpotential{\mu}(x)-zx\right\}\overset{\eqref{eq:def_log_potential}}{=}\frac{1}{2}\int_{\lm}^{\lM}\frac{\mu(dy)}{x-y} -z \in \left[ \frac{1}{2(x-\lm)}-z, \frac{1}{2(x-\lM)}-z  \right].
    \end{equation}    
\end{proof}

\begin{lemma}\label{lem:legendre_diff}
    Let $\mu$ be a probability with bounded support. Then $\logpotential{\mu}^*(z)$ is differentiable on $(0,\infty)$, and
    \begin{equation}\label{eq:legndre_diff}
        \logpotential{\mu}^{*'}(z)
        =
         -x^*(z),
    \end{equation}
    where $x^*(z)$ is the  unique maximizer in \eqref{eq:LegendreLstar}.
\end{lemma}

\begin{proof}
    We have $\logpotential{\mu}^*(z) = \logpotential{\mu}(x^*(z)) - z x^*(z)$ for all $z>0$. Note that $x^*(z) = \lM$ for all $z \ge \frac12 s_\mu$. Thus $L^*(z) = L_\mu(\lM) - \lM z$ for such $z$, which is clearly differentiable in $z$ and satisfies \eqref{eq:legndre_diff}. If on the other hand $z \in (0,s_\mu/2)$, then $x^*(z) > \lM$, and $x^*(z)$ is the unique solution of $\tfrac{\partial}{\partial x}\big\{L_{\mu}(x)-zx\big\}(x^*)= 0$. The implicit function theorem then implies that $x^{*}(z)$ is differentiable. The chain rule then implies that $L_\mu^*(z)$ is differentiable for $z<s_\mu/2$, and that \eqref{eq:legndre_diff} holds for such $z$.


\end{proof}

\begin{lemma}\label{lem:FE_lipschitz}
    Let $\mu$ be a probability with bounded support. Then $\Flim{\mu}$ is $|\lm|\vee |\lM|$-Lipschitz.
\end{lemma}

\begin{proof}
    The function $\Flim{\mu}$ is continuous on $[0,\infty)$ by \Cref{cor:free_energy_trivial_bounds}. By the definition of \eqref{eq:F_def} and Lemma \ref{lem:legendre_diff} it is also differentiable on $(0,\infty)$, with
    \begin{equation}
        \Flim{\mu}'(z)
        =
        x^*(z) - \frac{1}{2z}.
    \end{equation}
    Thus Lemma \ref{lem:xstar_legendre} concludes the proof of the lemma. 
\end{proof}

\section{Free-energy computation via the Legendre transform}\label{sec:Legendre_FE_high_temp}

The following result is a generalization of formula~(8) in~\cite[Theorem, p.~4]{genovese2015}, which was proved for the case when $\mu$ is the semicircle law. The proof is different to that in \cite{genovese2015}.

Let $\nu$ be an atomic measure with $N$ atoms 

\begin{equation}
    \nu=\frac{1}{N}\sum_{i=1}^N \delta_{\lambda_i}; \qquad \lambda_1\leq \dots \leq \lambda_N,
\end{equation}
and define

\begin{equation}\label{eq:F_nu_def}
    \Ffin{\nu}(\beta)
    :=
    \frac{1}{N}
    \log
    \Es\left[
        \exp\left(
            N \beta \sum_{i=1}^N \lambda_i \sigma_i^2
        \right)
    \right],
\end{equation}
where $\Es$ denotes the uniform measure on the sphere $S^{N-1}$. Also, let $W_1$ denote the Wasserstein-1 distance between probability measures. We suppose that there exists $\Lambda>0$ such that, uniformly in $N$, $\supp(\nu)\subset [-\Lambda,\Lambda]$. In particular, because $\mu$ satisfies Assumption \ref{main_assumption} we can choose $\Lambda$ such that, uniformly in $N$,

\begin{equation}\label{eq:assumption_sec3}
    \supp(\mu)\cup \supp(\nu) \subset[-\Lambda,\Lambda].
\end{equation}


\begin{theorem}\label{thm:sec2_main}
Let $\mu$ satisfy Assumption \ref{main_assumption}, and let $\Lambda\ge1$
be such that $\supp(\mu)\cup \supp(\nu) \subset[-\Lambda,\Lambda]$. Define

\begin{equation}
    \Delta_N(\nu,\mu)
    :=\sqrt{W_1(\nu,\mu)}
      +|\nu_{\max}-\lM|+\frac1{\sqrt N}.
\end{equation}
Then the following statements hold.
\begin{enumerate}
    \item For every $B>0$, there exists $C=C(B,\Lambda)>0$
    such that
    \begin{equation}
        \widehat F_\nu(\beta)
        \le F_\mu(\beta)+C\Delta_N(\nu,\mu)
        \qquad\text{for all }\beta\in[0,B].
    \end{equation}

    \item For every fixed $\beta\in(0,\beta_c)$, there exists
    $C=C(\mu,\beta,\Lambda)>0$ such that
    \begin{equation}
        \widehat F_\nu(\beta)
        \ge F_\mu(\beta)-C\Delta_N(\nu,\mu).
    \end{equation}
\end{enumerate}
Consequently, if $\mu_N$ is the empirical spectral measure of
a sequence satisfying Assumption~1, then
\begin{equation}
    \sup_{\beta\in[0,\beta_c]}
    \bigl|\widehat F_{\mu_N}(\beta)-F_\mu(\beta)\bigr|
    \longrightarrow0.
\end{equation}
\end{theorem}


\subsection{Upper bound}

We start by proving the upper bound part of Theorem \ref{thm:sec2_main}.

\begin{lemma}[Upper bound]\label{lem:fe_upper_bound}
    For any $\beta\ges 0$, if Assumption \eqref{eq:assumption_sec3} holds, then
    \begin{equation}\label{eq:fe_upper_bound}
        \Ffin{\nu}(\beta)
        \les
        \Flim{\mu}(\beta)
        + 
        C_\beta\left(
            \sqrt{W_1(\nu,\mu)}
            +
            \abs{\nu_{\max} - \lM}
            +
            \frac{1}{\sqrt{N}}
        \right),
    \end{equation}
    where $C_\beta=128[1+\beta\Lambda]e^{8(1+\beta\Lambda)}$.
\end{lemma}

Let $X=(X_1,\ldots,X_N)$ be a vector of i.i.d. standard Gaussian random variables. Then $X/\norm{X}$ is uniformly distributed on $S^{N-1}$, and we can rewrite \eqref{eq:F_nu_def} as
\begin{equation}\label{eq:replaceUniformByGaussian}
    \Ffin{\nu}(\beta)    
    =
    \frac{1}{N}
    \log
    \E\left[
        \exp\left(
            \beta H_N\left( \frac{X}{|X|}\right)
        \right)
    \right].
\end{equation}
The first step towards proving Lemma \ref{lem:fe_upper_bound} consists in restricting the range of possible values for $\norm{X}^2$. To this end, we introduce some large $A$ and consider the expression in \eqref{eq:replaceUniformByGaussian} restricted to $G_N$ and $G_N^c$, where $G_N=\{\norm{X}^2\in [NA^{-1},NA]\}$. Our first technical lemma checks that for $A$ chosen large enough only the set $G_N$ carries significant weight.

\begin{lemma}\label{lem:control_error_limit_A}
    Let $\beta \ges 0$ and $A = e^{4(1+\beta \Lambda)}$. Then for all $N\ges 1$
    \begin{equation}\label{eq:control_error_limit_A}
        \E\left[
            \exp\left(
                \beta
                N\sum_{i=1}^N
                \lambda_i\frac{X_i^2}{\norm{X}^2}
            \right)
            \indf_{G_N^c}
        \right]
        \les
        8e^{-N}
        \E\left[
            \exp\left(
                \beta N
                \sum_{i=1}^N
                \lambda_i\frac{X_i^2}{\norm{X}^2}
            \right)
            \indf_{G_N}
        \right].
    \end{equation}
\end{lemma}

\begin{proof}    
    Standard tail bounds for the chi-squared distribution (e. g. \cite[Theorems 1 and 2]{ghosh2021exponential}) yield
    \begin{equation*}
        \P(G_N^c) \les e^{-N I(A)} + e^{-N I(A^{-1})},
    \end{equation*}
    where
    \begin{equation}
        I(x) = \frac{1}{2}(x - 1 - \log x).
    \end{equation}
    Note that $I(x) \ges I(x^{-1}) \ges \frac{1}{2}(\log x - 1)$ for $x\ges 1$, and so $I(A)\ges I(A^{-1}) \ges 1 + 2\beta\Lambda$. Therefore, $\P(G_N^c) \les 2 e^{-N (1 + 2\beta\Lambda)}\les 2 e^{-1} \les \frac{3}{4}$, which implies that $\P(G_N) \ges \frac{1}{4}$ and so
    \begin{equation*}
        \P(G_N^c)
        \les
        8 e^{ -N (1 + 2 \beta \Lambda)}
        \P(G_N).
    \end{equation*}
    Using this together with
    \begin{equation*}
        \abs*{\beta N \sum_{i=1}^N \lambda_i\frac{X_i^2}{\norm{X}^2}} \les \beta N \Lambda\quad\text{a.s.},
    \end{equation*}
    we obtain \eqref{eq:control_error_limit_A}.
\end{proof}

In the following lemma a discretization of $\norm{X}^2$ is introduced.

\begin{lemma}\label{lem:discretisation}
    Let $\beta\ges 0$, $0<\e<1$, $A>1$, $G_N=\{|X|^2 \in [NA^{-1},NA]\}$ and suppose Assumption \eqref{eq:assumption_sec3} holds. Also, set
    \begin{gather*}
        R_k^2=(1+\e)^k,\quad k\in\Z,\\
        G_{N,k} = \{\norm{X}^2 \in [NR_k^2, NR_{k+1}^2)\},\quad k\in\Z,\\
        K_{\e,A}
        =
        \left\{
            k\in\Z
            :
            |k|\les 1 + \frac{\log A}{\log(1+\e)}
        \right\}.
    \end{gather*}
    Then for $N\ges 1$
    \begin{equation}\label{eq:discretisation}
            \E\left[
                    \exp\left(
                        \beta
                        N\sum_{i=1}^N
                        \lambda_i\frac{X_i^2}{\norm{X}^2}
                    \right)
                    \indf_{G_N}
            \right]
            \les            
            e^{ \e \beta N \Lambda }
            \sum_{k\in K_{\e,A}}
            \E\left[
                \exp\left(
                    \beta
                    \sum_{i=1}^N
                    \lambda_i\frac{X_i^2}{R_k^2}
                \right)
                \indf_{G_{N,k}}
            \right].
    \end{equation}
\end{lemma}

\begin{proof}
    We note that $G_N\subset \bigcup_{k\in K_{\e,A}} G_{N,k}$. On each $G_{N,k}$, we have
    \begin{equation}\label{eq:bound_on_Gnk}
        \abs{\norm{X}^2 - NR_k^2}\les NR_k^2\e,
    \end{equation}
    and so
    \begin{equation}
    \label{eq:bound_on_BNk}
        \left|
            N
            \sum_{i=1}^N
            \lambda_i
            \frac{X_i^2}{|X|^2}
            -
            \sum_{i=1}^N
            \lambda_i
            \frac{X_i^2}{R_k^2}
        \right|
        \les
        \e N \Lambda.
    \end{equation}    
    Summing the corresponding bounds, we obtain the required inequality.
\end{proof}

In the next lemma, an explicit upper bound for the expectation is given, and the proof is based on the introduction of a Lagrange multiplier.

\begin{lemma}\label{lem:discreteFE_UB}
    Let $\beta \ges 0$, $\e\in (0,1)$, and let $G_{N,k}$, $R_k$ be as in \Cref{lem:discretisation}. Then for any $N\ges 1$ and $k\in \Z$
    \begin{equation}\label{eq:lem_discreteFE}
        \mathbb{E}\left[
            \exp\left(
                \beta \sum_{i=1}^N \lambda_i\frac{X_i^2}{R_k^2}
            \right)
            \indf_{G_{N,k}}
        \right]
        \les
        \inf_{\eta} \left[e^{ N \e \abs{\eta} R_k^2 - N \eta R_k^2}
        \prod_{i=1}^N
        \frac{1}{\sqrt{1-\frac{2\beta\lambda_i}{R_k^2}-2\eta}}\right],
    \end{equation}
    where the infimum is taken over all $\eta<\tfrac{1}{2}-\tfrac{\beta\lambda_N}{R_k^2}$.
\end{lemma}

\begin{proof}
    Using \eqref{eq:bound_on_Gnk}, we get
    \begin{equation}
        \indf_{G_{N,k}} \les \exp\left( \eta (\norm{X}^2-NR_k^2) + N \e |\eta| R_k^2 \right),
    \end{equation}
    and therefore
    \begin{equation}\label{eq:discrete_FE_UB_step1}
        \text{l.h.s. of }\eqref{eq:lem_discreteFE}
        \les
        \exp\left(
            N \e \abs{\eta} R_k^2
            -
            N\eta R_k^2
        \right)
        \E\left[
            \exp\left(\sum_{i=1}^N X_i^2 \left(\frac{\beta \lambda_i}{R_k^2}+\eta\right)\right)
        \right].
    \end{equation}
    Recall that $X_i$ are i.i.d. standard Gaussians and that for a standard Gaussian random variable $\xi$ we have $\E[\exp(\theta\xi^2)]=(1 - 2\theta)^{-1/2}$ for $\theta<1/2$. Using this with $\theta_i = \beta \lambda_i / R_k^2 + \eta < 1/2$ for each $i$ yields
    \begin{equation*}
        \E\left[\exp\left(\sum_{i=1}^N X_i^2\left(\frac{\beta\lambda_i}{R_k^2}+\eta\right)\right)\right]
        =        
        \prod_{i=1}^N \frac{1}{\sqrt{1-\frac{2\beta\lambda_i}{ R_k^2}-2\eta}}.
    \end{equation*}
    Taking the infimum over all possible values of $\eta$ completes the proof.
\end{proof}

\begin{lemma}\label{lem:UB_putting_together}
    Let $\beta\ges0$, $C=\beta \Lambda$ and suppose Assumption \eqref{eq:assumption_sec3} holds. Then for any $N\ges1$, $A\geq e^{4(1+\beta\Lambda)} $ and $0<\e<1$
    \begin{equation}\label{eq:UB_putting_together}
        Z_N(\beta)
        \les
        9 e^{N C \e}
        \sum_{k\in K_{\e,A}}
        \inf_{\eta < \frac{1}{2} - \frac{\beta \lambda_N}{R_k^2}}
        e^{N \e (1 + \e) |\eta| A - N \eta R_k^2}
        \prod_{i=1}^N
        \frac{1}{\sqrt{1-\frac{2\beta\lambda_i}{R_k^2}-2\eta}}.
    \end{equation}
\end{lemma}

\begin{proof}
    This follows by combining the above lemmas. First use \Cref{lem:control_error_limit_A} (and the trivial inequality $8e^{-N}\les 8$), then \Cref{lem:discretisation} (whose error gives rise to the constant $C$), and finally \Cref{lem:discreteFE_UB} (recalling that $R_k^2 \les A(1+\e)$ for all $k\in K_{\e,A}$).
\end{proof}

In the next technical lemma, the convergence at exponential scale of the right-hand side of \eqref{eq:lem_discreteFE} is studied.

\begin{lemma}\label{lem:change_of_var}
    Let $\beta>0$, $R>0$ and  $\eta<\tfrac{1}{2}-\tfrac{\beta\lambda_N}{R^2}$. Then for all $N\ges1$
    \begin{equation}\label{eq:log_simplification}
        \frac{1}{N}
        \log\left(
            e^{-N\eta  R^2}
            \prod_{i=1}^N
            \frac{1}{\sqrt{1-\frac{2\beta\lambda_i}{R^2}-2\eta}}
        \right)
        =
        \log R
        -
        \frac{R^2}{2}
        -
        \frac{1}{2}\log(2\beta)
        +
        \vf_{\nu}(\gamma;\beta),
    \end{equation}
    where $\gamma=\tfrac{(1-2\eta)R^2}{2\beta} > \lambda_N$ and 
    \begin{equation}\label{eq:vf_mu_def}
        \vf_{\nu}(\gamma;\beta)
        =
        \beta \gamma
        -        
        \logpotential{\nu}(\gamma),
    \end{equation}
    where $\logpotential{\nu}$ is defined in \eqref{eq:def_log_potential}.
\end{lemma}

\begin{proof}
    Using the change of variables
    \begin{equation}\label{eq:change_of_var_eta_to_gamma}
        \eta = \frac{1}{2} - \frac{\beta \gamma}{R^2},
    \end{equation}
    we obtain that the product in \eqref{eq:log_simplification} equals
    \begin{equation}
        \left( \frac{R}{\sqrt{2\beta}} \right)^N
        \prod_{i=1}^N \frac{1}{\sqrt{\gamma - \lambda_i}}
        =
        \exp\left(
            N\left(
                \log R
                -
                \frac{1}{2}\log(2\beta)
                -
                \logpotential{\nu}(\gamma)
            \right)
        \right).       
    \end{equation}
    Substituting $-\eta R^2 = -\frac{R^2}{2} + \beta \gamma$ in the exponent in \eqref{eq:log_simplification} completes the proof.
\end{proof}

\begin{lemma}\label{lem:UB_putting_together_2}
    Suppose Assumption \eqref{eq:assumption_sec3} holds. Then, for any $\beta>0$, $A\geq e^{4(1+\beta\Lambda)}$ and $0<\e<1$, $N\ges1$
    \begin{equation}\label{eq:UB_putting_together_2}
        \Ffin{\nu}(\beta)
        \les
        -\frac{1}{2}
        -
        \frac{1}{2}\log(2\beta)
        +
        \inf_{\gamma > \lambda_N}
        \left[
            \vf_{\nu}(\gamma;\beta) + c_1 \abs{\gamma} \e
        \right]
        +
        c_2 \e
        +
        (N \e)^{-1} c_3,
    \end{equation}
    where $c_1 = 4 A^2 \beta$, $c_2 = \beta\Lambda + A$ and $c_3 = 5 + 4 \log A$.
\end{lemma}

\begin{proof}
    Using the change of variables \eqref{eq:change_of_var_eta_to_gamma} and \Cref{lem:change_of_var} shows that the sum \eqref{eq:UB_putting_together} is at most
    \begin{equation}\label{eq:sum_UB_1}
        \sum_{k\in K_{\e,A}}
        \inf_{\gamma>\lambda_N}
        \exp\left(
            N\e(1+\e)A\left(\frac{1}{2}
            +
            \frac{\beta\abs{\gamma}}{R_k^2}
            \right)
            +
            N\left(
                \log R_k - \frac{R_k^2}{2}            
                -
                \frac{1}{2}\log(2\beta)
                +
                \vf_{\nu}(\gamma;\beta)
            \right)
        \right).
    \end{equation}
    Note that $\log R - R^2/2 \les -1/2$ for all $R>0$ (with equality iff $R=1$). Furthermore,
    \begin{equation*}
    R_k^{-2} = R_{-k}^2 \les (1 + \e) A \les 2 A
    \end{equation*}
    for $k \in K_{\e,A}$ (see Lemma \ref{lem:discretisation}). Thus, the sum \eqref{eq:sum_UB_1} is not greater than
    \begin{equation}\label{eq:sum_UB_2}
        |K_{\e,A}|
        \exp\left(            
            N\left(
                -\frac{1}{2}
                -
                \frac{1}{2}\log(2\beta)
                +
                \inf_{\gamma > \lambda_N}
                \left[
                    \vf_{\nu}(\gamma;\beta)
                    +
                    4 A^2 \beta |\gamma| \e
                \right]
                +
                A \e
            \right)
        \right).        
    \end{equation}
    Next, we note that
    \begin{equation}
        \frac{1}{2}\e\les\frac{\e}{1+\e}\les\log(1+\e)\les\e,
    \end{equation}
    and so
    \begin{equation}
        \abs{K_{\e,A}} \les 2\left(1 + \frac{\log A}{\log(1+\e)}\right) + 1\les 3 + \frac{4\log A}{\e},
    \end{equation}
    and
    \begin{equation}
        \frac{1}{N}\log\abs{K_{\e,A}}\les\frac{1}{N} \left(2 + \frac{4\log A}{\e}\right)\les\frac{2 + 4\log A}{N\e}.
    \end{equation}
    Also,
    \begin{equation}
        \frac{1}{N}\log 9\les\frac{3}{N\e}.
    \end{equation}
    First bounding the sum in \eqref{eq:UB_putting_together} by \eqref{eq:sum_UB_2} and then applying $\frac{1}{N}\log$ to both sides of the resulting inequality yields \eqref{eq:UB_putting_together_2}.
\end{proof}

\begin{lemma}\label{lem:phi_min_UB}
    For any, not necessarily discrete, probability measure $\nu$ with compact support $[\nu_{\min},\nu_{\max}]\subset (-\infty,a]$, it holds
    \begin{equation}
        \vf_\nu'(
        a
        +
        (2\beta)^{-1} ; \beta)
        \ges
        0.
    \end{equation}
\end{lemma}

\begin{proof}
    Remark that $\vf_{\nu}'=\beta-L_{\nu}'$. Moreover, $2L_{\nu}'$ is the derivative of the log-potential, that is, the integral of $x\mapsto (t-x)^{-1}$ with respect to $\nu$. The evaluation at $t = a + (2\beta)^{-1}$ ensures that for all $x\les\nu_{\max}$ we have $(a + (2\beta)^{-1}-x)^{-1}\les 2\beta$, which implies the result.
\end{proof}

\begin{lemma}\label{lem:sum_convergence_fixedR}
    Let $\delta>0$. Then for any $N\ges 1$
    \begin{equation}\label{eq:sum_convergence_fixedR}
        \abs*{
            \vf_\nu^{(k)}(\gamma;\beta)
            -
            \vf_\mu^{(k)}(\gamma;\beta)
        }
        \les
        \frac{1}{2\delta^{k+1}}
        W_1(\nu,\mu)
        \quad
        \text{if}
        \quad
        k=0,1,
        \quad
        \gamma
        \ges \max(\lambda_N,\lM)
        +
        \delta.
    \end{equation}
\end{lemma}

\begin{proof}
    We notice that $\vf_{\nu}^{(k)}(\gamma;\beta)-\vf_\mu^{(k)}(\gamma;\beta)$ is the difference of the integrals of $x\mapsto \tfrac{1}{2}\log^{(k)}(\gamma-x)$ with respect to the measures $\nu$ and $\mu$. Because both of these measures have bounded support, the Kantorovich--Rubinstein duality theorem states that for any $1$-Lipschitz function $f$ we have
    \begin{equation}
        \abs*{\int f(x)\nu(dx) - \int f(x)\mu(dx)} 
        \les
        W_1(\nu,\mu).
    \end{equation}
    Since $\gamma\ges\lambda_N+\delta$ and $\gamma\ges\lM+\delta$, the function $x\mapsto \log^{(k)}(\gamma-x)$ is $1/\delta^{k+1}$-Lipschitz on the union of the supports of $\nu$ and $\mu$, and so \eqref{eq:sum_convergence_fixedR} follows.
\end{proof}

Finally, in the last technical lemma on the way to prove Lemma \ref{lem:fe_upper_bound} we study the error made by limiting the range over which the infimum is taken in the Lagrange transform of the log-potential.

\begin{lemma}\label{lem:limit_range_gamma}
Let $\nu$ be a compactly supported probability measure, and let
$a\ge\nu_{\max}$. Then, for every $\beta>0$ and $\delta>0$,
\begin{equation}
    0\le
    L_\nu^*(\beta)
    +\inf_{\gamma>a+\delta}\phi_\nu(\gamma;\beta)
    \le \beta(a+\delta-\nu_{\max}).
\end{equation}
If additionally $\delta<(2\beta)^{-1}$, then the same bounds hold
with the infimum restricted to $a+\delta<\gamma\le a+(2\beta)^{-1}$.
\end{lemma}

\begin{proof}
Recall that
\begin{equation}
    L_\nu^*(\beta)
    =-\inf_{\gamma\ge\nu_{\max}}\phi_\nu(\gamma;\beta).
\end{equation}
Let $\gamma^*$ denote the minimizer of
$\phi_\nu(\cdot;\beta)$ on $[\nu_{\max},\infty)$.
Since
\begin{equation}
    \phi_\nu'(\gamma;\beta)
    =\beta-\frac12S_\nu(\gamma)
\end{equation}
is increasing, the function $\phi_\nu(\cdot;\beta)$ is
non-decreasing on $[\gamma^*,\infty)$.

If $\gamma^*\ge a+\delta$, then
\begin{equation}
    L_\nu^*(\beta)
    +\inf_{\gamma>a+\delta}\phi_\nu(\gamma;\beta)=0,
\end{equation}
Otherwise, $\gamma^*<a+\delta$, and
\begin{equation}
\begin{aligned}
    L_\nu^*(\beta)
    +\inf_{\gamma>a+\delta}\phi_\nu(\gamma;\beta)
    &=\phi_\nu(a+\delta;\beta)-\phi_\nu(\gamma^*;\beta)\\
    &\le\beta(a+\delta-\gamma^*)\\
    &\le\beta(a+\delta-\nu_{\max}),
\end{aligned}
\end{equation}
where we used that $\phi_\nu'(\gamma;\beta)\le\beta$.
The lower bound follows directly by restricting the domain
of the infimum. Thus
\begin{equation}
    0\le
    L_\nu^*(\beta)
    +\inf_{\gamma>a+\delta}\phi_\nu(\gamma;\beta)
    \le\beta(a+\delta-\nu_{\max}).
\end{equation}

Finally, suppose that $\delta<(2\beta)^{-1}$.
By Lemma \ref{lem:phi_min_UB},
\begin{equation}
    \phi_\nu'\bigl(a+(2\beta)^{-1};\beta\bigr)\ge0.
\end{equation}
Since $\phi_\nu'(\cdot;\beta)$ is increasing,
$\phi_\nu(\cdot;\beta)$ is non-decreasing on
$[a+(2\beta)^{-1},\infty)$. Consequently,
\begin{equation}
    \inf_{\gamma>a+\delta} \phi_\nu(\gamma;\beta)
    =
    \inf_{a+\delta<\gamma\le a+(2\beta)^{-1}}
        \phi_\nu(\gamma;\beta),
\end{equation}
which proves the assertion for the restricted infimum.
\end{proof}



\begin{lemma}\label{lem:inf_phi_N<=inf_phi}
    Let $\beta,\e > 0$, and $0<\delta<(2\beta)^{-1}$. Then for any probability measures $\nu$ and $\mu$
    \begin{equation}\label{eq:inf_phi_N<=inf_phi}
        \inf_{\gamma > \lambda_N}
        \left[
            \vf_\nu(\gamma;\beta)
            +
            \e |\gamma|
        \right]
        \les
        -L^*_\mu(\beta)
        +
        c_4 \e
        +
        c_5 \abs{\lM - \lambda_N}
        +
        c_5 \delta
        +
        \frac{1}{\delta}W_1(\nu,\mu),
    \end{equation}
    where $c_4 =2\Lambda + (2\beta)^{-1}$ and $c_5 = \beta$.
\end{lemma}

\begin{proof}
    Let $a = \max\{\lambda_N,\lM\}$. Consider the infimum in the l.h.s. of \eqref{eq:inf_phi_N<=inf_phi}.
    Shrinking the range of the infimum to $(a+\delta,a+(2\beta)^{-1})$, bounding $|\gamma|\les |a|+(2\beta)^{-1} \les c_4$, and using \Cref{lem:sum_convergence_fixedR} with $k=0$, we deduce that the l.h.s. of \eqref{eq:inf_phi_N<=inf_phi} is at most
    \begin{equation}\label{eq:using_sum_conv}
        \inf_{\gamma \in (a + \delta, a + (2\beta)^{-1}) }
        \vf_\mu(\gamma;\beta)
        +
        c_4 \e
        +
        \frac{1}{\delta}W_1(\nu,\mu).
    \end{equation}
    Bounding the infimum in \eqref{eq:using_sum_conv} using \Cref{lem:limit_range_gamma} then yields the claim.
\end{proof}

Armed with this collection of technical lemmas we are ready to prove Lemma \ref{lem:fe_upper_bound}.

\begin{proof}[Proof of Lemma \ref{lem:fe_upper_bound}]

    The proof is trivial for $\beta=0$ or $N=0$. Suppose $\beta>0$, $N>0$. By \Cref{lem:UB_putting_together_2} and \Cref{lem:inf_phi_N<=inf_phi} we obtain, for $0<\delta<(2\beta)^{-1}$,
    \begin{equation}
        \Ffin{\nu}(\beta)
        \les
        \Flim{\mu}(\beta)
        +
        (c_2 + c_1 c_3) \e
        +
        c_4 (N\e)^{-1}
        +
        c_5 |\lM - \lambda_N|
        +
        c_5 \delta 
        +
        \frac{1}{\delta}W_1(\nu,\mu),
    \end{equation}
    where $A= e^{4(1+\beta\Lambda)}$, $c_1 = 4 A^2 \beta$, $c_2 = \beta \Lambda + A$, $c_3 = 2\Lambda + (2\beta)^{-1}$, $c_4 = 5 + 4 \log A$ and $c_5 = \beta$. Remark that all the constant factors appearing in the above equation are upper bounded by
    \begin{equation}
        16(\beta\Lambda+1)(A+1)^2\les 64[1+\beta\Lambda]e^{8(1+\beta\Lambda)}.
    \end{equation}
    
    If $\sqrt{W_1(\nu,\mu)}<(2\beta)^{-1}$, then, setting $\delta = \sqrt{W_1(\nu,\mu)}$ is possible and choosing $\e=1/\sqrt{N}$ yields the claim with $C=128[1+\beta\Lambda]e^{8(1+\beta\Lambda)}$ (if $W_1(\nu,\mu)=0$, pick $\delta=1/N$).

    Now assume that $\sqrt{W_1(\nu,\mu)}\ges (2\beta)^{-1}$. Then, also
    \begin{equation}
        \Ffin{\nu}(\beta)-\Flim{\mu}(\beta)\les
        C\left(
        \sqrt{W_1(\nu,\mu)}
        +
        \abs{\nu_{\max} - \lM}
        +
        \frac{1}{\sqrt{N}}
        \right),
    \end{equation}
    Indeed, we have
    \begin{equation}
        \Ffin{\nu}(\beta)\les\frac{1}{N}\log\Es e^{N\beta\max\limits_{1\les i\les N} \abs{\lambda_i}}\les\beta\Lambda,
    \end{equation}
    and by \Cref{cor:free_energy_trivial_bounds}
    \begin{equation}
        \Flim{\mu}(\beta)\ges -\beta\Lambda,
    \end{equation}
    so that
    \begin{equation}
        \Ffin{\nu}(\beta)-\Flim{\mu}(\beta)\les2\beta\Lambda\les 4\beta^2\Lambda\sqrt{W_1(\nu,\mu)}.
    \end{equation}
    Furthermore,
    \begin{equation}
        4\beta^2\Lambda\sqrt{W_1(\nu,\mu)}\les 
        C\left(
        \sqrt{W_1(\nu,\mu)}
        +
        \abs{\nu_{\max} - \lM}
        +
        \frac{1}{\sqrt{N}}
        \right),
    \end{equation}
    which proves the claim.

    Combining the two cases depending on the relation between $\sqrt{W_1(\nu,\mu)}$ and $(2\beta)^{-1}$, we obtain the desired result.
\end{proof}

\subsection{Lower bound}

After proving the upper bound, we can now prove the lower bound part of Theorem \ref{thm:sec2_main}.


\begin{lemma}\label{lem:fe_lower_bound}
    Suppose that Assumption \eqref{eq:assumption_sec3} holds. Then, for any $\beta\in (0,\beta_c)$, there exits a constant $C>0$ such that
    \begin{equation}\label{eq:fe_lower_bound}
        \Ffin{\nu}(\beta)\ges \Flim{\mu}(\beta)-C\left(\sqrt{W_1(\nu,\mu)}+\abs{\lambda_N-\lM}+\frac{1}{\sqrt{N}}\right).
    \end{equation}
\end{lemma}

The proof is based on several short technical lemmas.

\begin{lemma}
    For any $\beta \in (0, \beta_c)$, then there exists $\delta>0$ such that $\vf_\mu'(\lM+\delta;\beta)<0$.
\end{lemma}

\begin{proof}
    For $\beta \in (0, \beta_c)$, we have
    \begin{equation}
        \vf_\mu'(\lM;\beta)<\vf_\mu'(\lM;\beta_c)=\beta_c-\frac{1}{2}S_\mu(\lM)=0.
    \end{equation}
    The statement now follows from the continuity of the function $\vf_\mu'(\cdot;\beta)$.
\end{proof}

\begin{lemma}\label{lem:away_from_lambda_max}
For any $\beta\in (0,\beta_c)$, let $\delta>0$ be such that
    \begin{gather}
        \abs{\lambda_N-\lM}\les\frac{1}{2}\delta,\\
        \vf'_\mu(\lM + 2\delta;\beta) \les -\frac{2}{\delta^2}W_1(\nu,\mu).
    \end{gather}
    Then
    \begin{equation}
        \vf'_\nu(\lambda_N + \delta;\beta)\les 0.
    \end{equation}
\end{lemma}

\begin{proof}
    It follows from the assumptions that
    \begin{equation}
        \max(\lambda_N,\lM)
        +
        \frac{1}{2}\delta
        \les
        \lambda_N
        + 
        \delta
        \les
        \lM
        + 
        2\delta.      
    \end{equation}
    Thus, by \Cref{lem:sum_convergence_fixedR}
    \begin{equation}
        \vf'_\nu(\lambda_N + \delta;\beta)
        \les
        \vf'_\mu(\lambda_N + \delta;\beta)
        +
        \frac{2}{\delta^2}
        W_1(\nu,\mu)
        \les
        \vf'_\mu(\lM + 2\delta;\beta)
        +
        \frac{2}{\delta^2}
        W_1(\nu,\mu)
        \les
        0,
    \end{equation}
where we used that $\vf'_\mu(\cdot;\beta)$ is increasing.
\end{proof}

For the remainder of the proof, we consider a probability measure $\Q_\eta$ with density
\begin{equation}\label{eq:Q_def}
    \frac{d\Q_\eta}{d\mbP}
    =
    \frac{
        \exp\left(\frac{\beta}{R^2}\sum_{i=1}^N \lambda_iX_i^2+\eta(\norm{X}^2-NR^2)\right)
    }{
        \E\left[\exp\left(\frac{\beta}{R^2}\sum_{i=1}^N \lambda_iX_i^2+\eta(\norm{X}^2-NR^2)\right)\right]
    },    
\end{equation}
and the function
\begin{equation}\label{eq:def_g_mu}
    g_\nu(\eta)
    =
    \log\E\left[
        \exp\left(
            \frac{\beta}{R^2}\sum_{i=1}^N\lambda_iX_i^2+\eta(\norm{X}^2-NR^2)
        \right)
    \right],\quad\eta<\frac{1}{2}-\frac{\beta\lambda_N}{R^2}.
\end{equation}

\begin{lemma}\label{lem:LB_cov_eta_star}
    For any $\beta\in (0,\beta_c)$, let $\delta>0$ be such that the conditions of \Cref{lem:away_from_lambda_max} are satisfied. Then the function $g_\nu$ has a unique minimizer $\eta^\ast<\frac{1}{2}-\frac{\beta\lambda_N}{R^2}$, which 
    satisfies
    \begin{equation}\label{eq:Q_moment_bounds}
        \Q_{\eta^\ast}[\norm{X}^2]
        =
        NR^2,
        \quad\quad
        \Var_{\Q_{\eta^\ast}}[\norm{X}^2]
        \les
        \frac{NR^4}{2\beta^2\delta^2},
    \end{equation}
    and
    \begin{equation}\label{eq:eta_star_bounds}
        \frac{1}{2}
        -
        \frac{\beta\lambda_N}{R^2}
        -
        \frac{1}{2R^2}
        \les
        \eta^\ast
        <
        \frac{1}{2}-\frac{\beta\lambda_N}{R^2}.
    \end{equation}
\end{lemma}

\begin{proof}
    We note that the function $g_\nu$ satisfies
    \begin{equation}\label{eq:phi_N_derivatives}
        g'_\nu(\eta)
        =
        \Q_\eta[\norm{X}^2] - NR^2,
        \qquad
        g''_\nu(\eta)
        =
        \Var_{\Q_\eta}[\norm{X}^2]\ges 0.
    \end{equation}
    and is convex. By \Cref{lem:change_of_var}, for all $\eta<\frac{1}{2}-\frac{\beta\lambda_N}{R^2}$ we have
    \begin{equation}
        \frac{1}{N}g_\nu(\eta) =
            \log R
            -
            \frac{R^2}{2}
            -
            \frac{1}{2}\log(2\beta)
            +
            \vf_\nu(\gamma;\beta),
    \end{equation}
    where
    \begin{equation}\label{eq:LB_change_of_var}
        \gamma = \gamma(\eta) = \frac{(1-2\eta)R^2}{2\beta},
    \end{equation}
    and so
    \begin{equation}
        g_\nu'(\eta) = -\frac{NR^2}{\beta} \vf_\nu'(\gamma;\beta).
    \end{equation}
    Therefore, by \Cref{lem:away_from_lambda_max}, $g_\nu'(\eta_\delta)\ges 0$ for $\eta_\delta$ such that $\gamma(\eta_\delta)=\lambda_N+\delta$. Furthermore, $\vf'_\nu(\cdot;\beta)$ is strictly increasing and $\vf_\nu'(\gamma;\beta) \to \beta$ as $\gamma\to +\infty$, so $g_\nu'$ is strictly increasing and $ g_\nu'(\eta) \to -NR^2$ as $\eta \to -\infty$. Thus, $g_\nu$ has a unique minimizer $\eta^*\les\eta_\delta$, and this minimizer satisfies $g_\nu'(\eta^*)=0$ and $\gamma^*=\gamma(\eta^*)\ges \gamma(\eta_\delta)=\lambda_N+\delta$.
    Therefore, \eqref{eq:phi_N_derivatives} implies that
    \begin{equation}
        \Q_{\eta_*}[\norm{X}^2]=NR^2.
    \end{equation}
    Furthermore,
    \begin{equation}
        0
        \les
        \vf_\nu''(\gamma;\beta)
        =
        -L_\nu''(\gamma)
        =
        \frac{1}{2N}\sum_{i=1}^N \frac{1}{(\gamma-\lambda_i)^2}
        \les
        \frac{1}{2(\gamma-\lambda_N)^2},
    \end{equation}
    and so
    \begin{equation}
        0
        \les
        g_\nu''(\eta^*)
        =
        \frac{NR^4}{\beta^2}
        \vf_\nu''(\gamma^*;\beta)
        \les
        \frac{NR^4}{2\beta^2\delta^2}.
    \end{equation}

    Lastly, using \Cref{lem:phi_min_UB}, we get
    \begin{equation}
        \eta^\ast=\frac{1}{2}-\frac{\beta\gamma^*}{R^2}\ges\frac{1}{2}-\frac{\beta}{R^2}\left(\lambda_N+\frac{1}{2\beta}\right)=\left(\frac{1}{2}-\frac{\beta\lambda_N}{R^2}\right)-\frac{1}{2R^2},
    \end{equation}
    which implies \eqref{eq:eta_star_bounds} and completes the proof.
\end{proof}

\begin{lemma}\label{lem:g_lower_bound}
    For $\eta^\ast$ defined in \Cref{lem:LB_cov_eta_star} and $R=1$
    \begin{equation}\label{eq:g_lower_bound}
        \frac{1}{N}g_\nu(\eta^\ast)\ges \Flim{\mu}(\beta)-(1+\beta)\sqrt{W_1(\nu,\mu)}-\beta\abs{\lambda_N-\lM}.
    \end{equation}
\end{lemma}

\begin{proof}
    By the definition of $\eta^\ast$, we have
    \begin{equation}
        \frac{1}{N}g_\nu(\eta^\ast)=\inf_{\eta<\frac{1}{2}-\beta\lambda_N} \frac{1}{N}\log\E\left[\exp\left(\beta\sum_{i=1}^N \lambda_iX_i^2+\eta(\norm{X}^2-N)\right)\right].
    \end{equation}
    Changing the variable to $\gamma=\tfrac{1-2\eta}{2\beta}$, we get
    \begin{equation}
        \begin{aligned}
        \frac{1}{N}g_\nu(\eta^\ast)=\inf_{\gamma>\lambda_N}
        \left\{
        -\frac{1}{2}-\frac{1}{2}\log(2\beta)+\beta\gamma-\frac{1}{2N}\sum_{i=1}^N \log(\gamma-\lambda_i)
        \right\}.
        \end{aligned}
    \end{equation}
    
    Next, fix some $\delta>0$ and apply the result of Lemma \ref{lem:limit_range_gamma} with measure $\nu$, $(\lM-\lambda_N)_++\delta$ instead of $\delta$ and $a=\lambda_N$ to get
    \begin{equation}
        \begin{aligned}
        \frac{1}{N}g_\nu(\eta^\ast)\ges \ & \left(-\frac{1}{2}-\frac{1}{2}\log(2\beta)
        +\inf_{\gamma>\max(\lambda_N,\lM)+\delta}
        \left\{
        \beta \gamma -\frac{1}{2N}\sum_{i=1}^N \log(\gamma-\lambda_i)
        \right\}\right)\\
        &-\beta\abs{\lambda_N-\lM}-\beta\delta.
        \end{aligned}
    \end{equation}
    Finally, using Lemma \ref{lem:sum_convergence_fixedR} for the expression in parentheses, we obtain
    \begin{equation}
        \frac{1}{N}g_\nu(\eta^\ast)\ges \Flim{\mu}(\beta)-\frac{1}{\delta}W_1(\nu,\mu)-\beta\abs{\lambda_N-\lM}-\beta\delta,
    \end{equation}
    and taking $\delta=\sqrt{W_1(\nu,\mu)}$ finishes the proof.
\end{proof}

Armed with the abover technical Lemmas we can prove the lower bound part of Theorem \ref{thm:sec2_main}.

\begin{proof}[Proof of Lemma \ref{lem:fe_lower_bound}]
    First, we note that
    \begin{equation}
        \Ffin{\nu}(\beta)\ges\frac{1}{N}\log\Es e^{-N\beta\max\limits_{1\les i\les N} \abs{\lambda_i}}\ges -\beta\Lambda,
    \end{equation}
    and by \Cref{cor:free_energy_trivial_bounds}
    \begin{equation}
        \Flim{\mu}(\beta)\les\beta\abs{\lM}\les\beta\Lambda,
    \end{equation}
    Therefore,
    \begin{equation}
        \Ffin{\nu}(\beta)-\Flim{\mu}(\beta)\ges -2\beta\Lambda.
    \end{equation}

    Now, for arbitrary $R>0$ and $\e\in (0,1)$, define $I_{R,\e}=[R^2(1-\e),R^2(1+\e)]$; later we pick $R=1$. By \eqref{eq:bound_on_BNk}, we have
    \begin{equation}\label{eq:fe_lower_bound_first_step}
        \E\left[
            \exp\left(\beta N\sum_{i=1}^N \lambda_i\frac{X_i^2}{\norm{X}^2}\right)
        \right]
        \ges
        \E\left[
            \exp\left(\frac{\beta}{R^2}\sum_{i=1}^N \lambda_iX_i^2\right)
            \indf_{ N^{-1}\norm{X}^2\in I_{R,\e} }
        \right]
        \exp\left(
            -N \beta \Lambda \e
        \right).
    \end{equation}
    For arbitrary $\eta\in\mbR$, for which the expectation in the denominator is finite, let us consider the measure $\Q_\eta$ defined in \eqref{eq:Q_def}. Using it to rewrite the expectation in the r.h.s. of \eqref{eq:fe_lower_bound_first_step} shows that the expectation is at least
    \begin{equation}
        e^{
            -NR^2\abs{\eta}\e
        }\cdot
        \Q_\eta[N^{-1}\norm{X}^2\in I_{R,\e}]
        \cdot\E\left[\exp\left(\frac{\beta}{R^2}\sum_{i=1}^N \lambda_iX_i^2+\eta(\norm{X}^2-NR^2)\right)\right].
    \end{equation}
    
    Fix an arbitrary $\delta>0$ such that the conditions of \Cref{lem:away_from_lambda_max} are satisfied, and let $\eta^\ast$ be defined as in \Cref{lem:LB_cov_eta_star}. Using Chebyshev's inequality and \eqref{eq:Q_moment_bounds}, we get
    \begin{equation}
        \begin{aligned}
        \Q_{\eta^\ast}[N^{-1}\norm{X}^2\in I_{R,\e}]
        &=
        \Q_{\eta^\ast} \left[\abs{\norm{X}^2-\Q_{\eta^\ast}[\norm{X}^2]}\les NR^2\e\right]
        \\
        &\ges
        1
        -
        \frac{1}{N^2R^4\e^2}
        \Var_{\eta^\ast}[\norm{X}^2]
        \\
        &\ges
        1
        -
        \frac{1}{2N\beta^2\delta^2\e^2},
        \end{aligned}
    \end{equation}
    and so we conclude that
    \begin{equation}
        \Q_{\eta^\ast}[N^{-1}\norm{X}^2\in I_{R,\e}]\ges\frac{1}{2},
    \end{equation}
    if we take
    \begin{equation}
        \label{eq:epsilon_def}
        \e=\frac{1}{\beta\delta\sqrt{N}}
    \end{equation}
    with $N$ large enough for $\e$ to be in $(0,1/2)$, i.e.
    \begin{equation}
        \label{eq:condition_on_N}
        N\ges\frac{4}{\beta^2\delta^2}.
    \end{equation}
    Therefore, picking $R=1$ we arrive at the bound
    \begin{equation}
        \frac{1}{N}\log\E\left[\exp\left(\beta N\sum_{i=1}^N \lambda_i\frac{X_i^2}{\norm{X}^2}\right)\right]
        \ges -4\beta\Lambda\e-\abs{\eta^\ast}\e+\frac{1}{N}\log\frac{1}{2}+\frac{1}{N}g_{\hat{\mu}}(\eta^\ast)
    \end{equation}
    with $\e$ given in \eqref{eq:epsilon_def}. Using \eqref{eq:eta_star_bounds}, we get
    \begin{equation}
        \abs{\eta^\ast}
        <
        \frac{1}{2}+\frac{\beta\Lambda}{R^2}+\frac{1}{2R^2}.
    \end{equation}
    Hence, for $R=1$ and all $N\ges 4/(\beta^2\delta^2)$, we get
    \begin{equation}
        \Ffin{\nu}(\beta)\ges -\frac{1+5\beta\Lambda}{\beta\delta\sqrt{N}}-\frac{\log 2}{N}+\frac{1}{N}g_\nu(\eta^\ast).
    \end{equation}
    Using \Cref{lem:g_lower_bound}, we obtain
    \begin{equation}
        \Ffin{\nu}(\beta)\ges\Flim{\mu}(\beta)-\left(
        \frac{1+5\beta\Lambda}{\beta\delta\sqrt{N}}
        +
        \frac{\log 2}{N}
        +
        (1+\beta)\sqrt{W_1(\nu,\mu)}
        +
        \beta\abs{\nu_{\max}-\lM}
        \right),
    \end{equation}
    where $\delta>0$ is arbitrary for which the conditions of \Cref{lem:away_from_lambda_max} are satisfied.

    Finally, suppose that $\delta>0$ is such that $\vf'_\mu(\lM + 2\delta;\beta)<0$, but does not satisfy the conditions of \Cref{lem:away_from_lambda_max} or the inequality \eqref{eq:condition_on_N}. This means that at least one of the inequalities
    \begin{gather}
        \frac{2\abs{\nu_{\max}-\lM}}{\delta}\ges 1,\\
        \frac{2W_1(\nu,\mu)}{\delta^2\abs{\vf'_\mu(\lM + 2\delta;\beta)}}\ges 1,\\
        \frac{4}{\beta^2\delta^2N}>1,
    \end{gather}
    holds. However, in this case
    \begin{equation}
        \left(\frac{2\abs{\nu_{\max}-\lM}}{\delta}
        +
        \sqrt{\frac{2W_1(\nu,\mu)}{\delta^2\abs{\vf'_\mu(\lM + 2\delta;\beta)}}}
        +
        \frac{4}{\beta^2\delta^2N}\right)
        \cdot
        2\beta\Lambda
        \ges
        2\beta\Lambda,
    \end{equation}
    and so $\e_\delta$ defined in \eqref{eq:epsilon_delta} is not less than the rough bound $2\beta\Lambda$.

    Thus, we get that for any $\delta>0$ such that $\vf'_\mu(\lM + 2\delta;\beta)<0$
    \begin{equation}\label{eq:fe_lower_bound_bis}
        \Ffin{\nu}(\beta)\ges \Flim{\mu}(\beta)-\e_\delta,
    \end{equation}
    where
    \begin{equation}\label{eq:epsilon_delta}
        \begin{aligned}
        \e_\delta
        &=
        \frac{1+5\beta\Lambda}{\beta\delta\sqrt{N}}
        +
        \frac{\log 2}{N}
        +
        (1+\beta)\sqrt{W_1(\nu,\mu)}
        +
        \beta\abs{\lambda_N-\lM}\\
        &+
        \left(\frac{2\abs{\lambda_N-\lM}}{\delta}
        +
        \sqrt{\frac{2W_1(\nu,\mu)}{\delta^2\abs{\vf'_\mu(\lM + 2\delta;\beta)}}}
        +
        \frac{4}{\beta^2\delta^2N}\right)
        2\beta\Lambda.
        \end{aligned}
    \end{equation}
    The statement of the lemma then follows trivially.
\end{proof}

\begin{proof}[Proof of \Cref{thm:sec2_main}]
    The first statement of the theorem follows directly from Lemmas~\ref{lem:fe_upper_bound} and~\ref{lem:fe_lower_bound}.

    It remains to prove the uniform convergence assertion.
    Assumption \ref{main_assumption} implies that the spectra are uniformly bounded,
    that $W_1(\mu_N,\mu)\to0$, and that
    $\lambda_N\to\lM$.
    The preceding estimates therefore give
    \begin{equation}
        \widehat F_{\mu_N}(\beta)\longrightarrow F_\mu(\beta)
        \qquad\text{for every }\beta\in(0,\beta_c).
    \end{equation}
    The same holds at $\beta=0$, since both functions vanish there.
    
    Moreover, for all $s,t\ge0$,
    \begin{equation}
        |\widehat F_{\mu_N}(t)-\widehat F_{\mu_N}(s)|
        \le\Lambda|t-s|,
    \end{equation}
    while Lemma~2.5 gives
    \begin{equation}
        |F_\mu(t)-F_\mu(s)|\le\Lambda|t-s|.
    \end{equation}
    Hence, for every $\eta\in(0,\beta_c)$,
    \begin{equation}
        \begin{aligned}
        |\widehat F_{\mu_N}(\beta_c)-F_\mu(\beta_c)|
        &\le
        |\widehat F_{\mu_N}(\beta_c-\eta)
          -F_\mu(\beta_c-\eta)|
          +2\Lambda\eta.
    \end{aligned}
    \end{equation}
    Taking first $N\to\infty$ and then $\eta\downarrow0$ proves
    convergence at $\beta_c$. Finally, the common Lipschitz bound
    and a finite-net argument imply
    \begin{equation}
        \sup_{\beta\in[0,\beta_c]}
        |\widehat F_{\mu_N}(\beta)-F_\mu(\beta)|
        \longrightarrow0.
    \end{equation}
\end{proof}

\section{A spectral partition of the sphere}\label{sec:partition}

In this section we introduce a partition of the sphere $S^{N-1}$, over which we bound the free energy. It consists of $N$ sets corresponding to the minimal number of coordinates of $\sigma$ (coordinates in the top eigendirections) needed to reach the given contribution $q^\ast>0$. This partition allows us to isolate the contribution to the free energy of different configurations of the spins and to isolate an effective low temperature and effective high temperature contribution. Throughout this section, fix $\beta>\beta_c$, so that
$q^*=1-\beta_c/\beta\in(0,1)$. The case $\beta\le\beta_c$ is covered by Theorem \ref{thm:sec2_main}. We also assume $N\ge4$, which is sufficient for the asymptotic statements below. In the remainder of the paper we use that, consequently to Assumption \ref{main_assumption}, there exists $\Lambda$ such that uniformly in $N$,

\begin{equation}\label{eq:def_Lambda}
    \supp(\mu)\cup \supp(\mu_N) \subset[-\Lambda,\Lambda].
\end{equation}
Define for $K \in \{0,\ldots,N-1\}$
\begin{equation}\label{eq:def_q_tauLR}    
    \tau_{L,K}(\sigma)
    =
    \sum_{i=1}^{N-K-1} \sigma_i^2,
    \qquad
    q_K(\sigma)
    :=
    \sum_{i=N-K}^{N} \sigma_{i}^2,
    \qquad
    \tau_{R,K}(\sigma)
    =
    \sum_{i=N-K+1}^N \sigma_i^2,
\end{equation}
with the convention
\begin{equation}\label{eq:def_q_tau_LR_convention}
    \tau_{L,N-1}(\sigma)
    =
    \tau_{R,0}(\sigma)
    =
    0,
\end{equation}
so that
\begin{equation}
    q_K(\sigma)
    =
    \sigma_{N-K}^2
    +
    \tau_{R,K}(\sigma)
\end{equation}
and
\begin{equation}
    \tau_{L,K}(\sigma)
    +
    \sigma_{N-K}^2
    +
    \tau_{R,K}(\sigma)
    =
    1
    ,\quad K=0,\ldots,N-1.
\end{equation}
Furthermore, let
\begin{equation}\label{eq:def_K}
    K(\sigma)
    =
    \min
    \big\{
        K\in\{0,1,\ldots,N-1\}
        \mid
        q_K(\sigma) \ge q^\ast
    \big\}.
\end{equation}
Note that by construction
\begin{equation}\label{eq:q_tau_LR_K(sigma)}
    \forall \sigma \in S^{N-1}\colon
    \qquad
    q_{K(\sigma)} \ge q^\ast,
    \qquad
    \tau_{L,K(\sigma)} \le 1-q^\ast,
    \qquad
    \tau_{R,K(\sigma)} < q^\ast.
\end{equation}
We consider the partition
\begin{equation}
    S^{N-1} = \bigcup_{K=0}^{N-1} A_K,
\end{equation}
where
\begin{equation}\label{eq:def_AK}
    A_K
    =
    \left\{
    \sigma\in S^{N-1}\mid K(\sigma)=K
    \right\}
    ,\quad K=0,1,\ldots,N-1,
\end{equation}
Defining
\begin{equation}\label{eq:def_ZN_A}
    Z_N(\beta,A) = \Es[e^{\beta H_N(\sigma)} \indf_A(\sigma)],
\end{equation}
we get
\begin{equation}
    \Es[e^{\beta H_N(\sigma)}]
    =
    \sum_{K=0}^{N-1}
     Z_N(\beta,A_K)
    \les N\max_{0\les K\les N-1} 
     Z_N(\beta,A_K),
\end{equation}
since the sets $A_K$ are disjoint. We do not directly estimate the free energy restricted to the sets $A_K$ but further partition them depending on how big the spin $\sigma_{N-K}^2$ is.

Note that $A_K$ can be written
\begin{equation}\label{eq:A_def}
    A_K
    =
    \{
        \sigma \in S^{N-1}        
        \mid        
        \tau_{K,R}(\sigma) < q^\ast \leq q_K(\sigma)
    \},
    \quad K=0,1,\ldots,N-1,
\end{equation}
A natural way to further partition $A_K$ is to define for $\tau_L,\tau_R\ge 0$ and $\e>0$
\begin{equation}\label{eq:defAkqlqrdelta}
    A_{K,\e}^{\tau_L,\tau_R}
    =
    \left\{
        \sigma\in A_K
        \mid
        \tau_L-\e \leq \tau_{L,K}(\sigma) \leq \tau_L+\e,
        \;
        \tau_R-\e \leq \tau_{R,K}(\sigma) \leq \tau_R+\e
    \right\}
\end{equation}
so that recalling \eqref{eq:q_tau_LR_K(sigma)} we have
\begin{equation}
    A_K
    =
    \bigcup_{
        \substack{
            \tau_L\in\pi_L
            \\            
            \tau_R\in\pi_R      
        }
    }
    A_{K,\e}^{\tau_L,\tau_R}.
\end{equation}
where

\begin{equation}\label{eq:def_set_piL_piR}
\begin{aligned}
    \pi_L=\{0,2\e,\dots,2\e\lfloor(1-q^*)/(2\e) \rfloor\} \cup \{1-q^*\},\\
    \pi_R=\{0,2\e,\dots,2\e\lfloor q^*/(2\e) \rfloor\}\cup \{q^*\}.
\end{aligned}
\end{equation}
Then also
\begin{equation}\label{eq:covering_sphere_AK}
    S^{N-1}
    =
    \bigcup_{K=0,\dots,N-1}
    \bigcup_{
        \substack{
            \tau_L\in \pi_L 
            \\            
            \tau_R\in \pi_R     
        }
    }
    A_{K,\e}^{\tau_L,\tau_R}.
\end{equation}
Note that because of this decomposition, in the following we often only prove results for $\tau_L,\tau_R=0,2\epsilon,\dots$. We are effectively separating the contribution of the directions associated with the large eigenvalues and those associated with the small eigenvalues. It corresponds to the part effectively at high and low temperature.


We introduce the truncated spectral measures
\begin{equation}\label{eq:def_mu_hatLR}
    \hat{\mu}_{L,K}
    =
    \frac{1}{N-K-1}
    \sum_{i=1}^{N-K-1} \delta_{\lambda_i},
    \qquad \qquad
    \hat{\mu}_{R,K}
    =
    \frac{1}{K}
    \sum_{i=N-K+1}^N \delta_{\lambda_i}.
\end{equation}
We also introduce the entropy of the discrete distribution $(\alpha,0,1-\alpha)$ relative to $(\tau_R,1-\tau_L-\tau_R,\tau_L)$. Recall that for two distributions $P$ and $Q$ over $[n]$ we have
\begin{equation}
    D\big( P|| Q\big)
    =
    \sum_{x=1}^n P(x)\log\left(\frac{P(x)}{Q(x)}\right).
\end{equation}
In our case, we denote in short 
\begin{equation}\label{eq:def_entropy}
    \begin{array}{rcl}
        D(\alpha||\tau_L,\tau_R)
        &:=&
        D\big((\alpha,0,1-\alpha)||(\tau_R,1-\tau_L-\tau_R,\tau_L)\big)\\
        & =&
        \alpha\log\left(\frac{\alpha}{\tau_R}\right)
        +
        (1-\alpha)\log\left(\frac{1-\alpha}{\tau_L}\right),
    \end{array}    
\end{equation}
for $\alpha\in(0,1)$ and $\tau_L,\tau_R<0$, $\tau_L+\tau_R\les1$. The measure of the set $A_{K,\e}^{\tau_L,\tau_R}$ is approximated by $\exp(-\frac{N}2 D(\alpha||\tau_L,\tau_R))$ (cf. Lemma \ref{lem:measure_AK_bulk}). For $\alpha=0$ we use the convention
\begin{equation}\label{eq:def_entropy_alpha=0}
    D(0||\tau_L,\tau_R)
    :=
    \begin{cases}
        \infty & \text{ if }\tau_R>0,\\
        \log \tau_L^{-1} & \text{otherwise}.
    \end{cases}
\end{equation}
and for $\alpha=1$
\begin{equation}\label{eq:def_entropy_alpha=1}
    D(1||\tau_L,\tau_R)
    :=
    \begin{cases}
        \infty & \text{ if }\tau_L>0,\\
        \log \tau_R^{-1} & \text{otherwise}.
    \end{cases}    
\end{equation}
In addition $D(\alpha||\tau_L,\tau_R)=\infty$ if $\alpha\in(0,1)$ and either $\tau_L=0$ or $\tau_R=0$.


In the following proposition, we provide an upper bound for the free energy on the set $A_{K,\e}^{\tau_L,\tau_R}$ in terms of the truncated free energies and relative entropy in the range $K\in [N-2]$. For $K=1,\dots,N-2$, let
\begin{equation}\label{eq:def_MNK}
    \begin{aligned}
        M_{N,K}(\tau_L,\tau_R)
        = \ &
        \frac{N-K-1}{N}\Ffin{\hat{\mu}_{L,K}}\left( \frac{N}{N-K-1}\beta \tau_L \right)
        +
        \beta \lambda_{N-K}(1-\tau_L-\tau_R)\\
        &
        +
        \frac{K}{N}\Ffin{\hat{\mu}_{R,K}}\left( \frac{N}{K}\beta \tau_R \right)
        -
        \frac{1}{2}D(\alpha||\tau_L,\tau_R),
    \end{aligned}
\end{equation}
where $\alpha = \frac{K}{N}$.
For $K=0$ let
\begin{equation}\label{eq:def_MN0}
    M_{N,0}(\tau_L,\tau_R)
    =
    M_{N,0}(\tau_L)
    =
    \frac{N-1}{N}
    \Ffin{\hat{\mu}_{L,0}}\left(\frac{N}{N-1} \beta \tau_L\right)
    +
    \beta\lambda_N(1-\tau_L)
    -
    \frac12 D(0||\tau_L,0),
\end{equation}
and for $K=N-1$ let
\begin{equation}\label{eq:def_MNN}
    M_{N,N-1}(\tau_L,\tau_R)
    =
    M_{N,N-1}(\tau_R)
    =
    \frac{N-1}{N}
    \Ffin{\hat{\mu}_{R,N-1}}\left(\frac{N}{N-1} \beta \tau_R\right)
    +
    \beta\lambda_1(1-\tau_R)
    -
    \frac12 D(1||0,\tau_R).
\end{equation}

\begin{proposition}\label{prop:sec4main}
    There exists $c>0$ such that for any $\beta>\beta_c$, $0<\epsilon<\min\{q^*,1-q^*\}$, $\tau_L\in [0,1-q^*]$, $\tau_R\in [0,q^*]$ and $K=1,\dots N-2$
    \begin{equation}
        \frac{1}{N}
        \log
        Z_N (\beta,A_{K,\e}^{\tau_L,\tau_R})
        \les
        \sup_{x,y\in C_\epsilon(\tau_L,\tau_R)} M_{N,K}(x,y)
        +
        c\frac{\log N \epsilon^{-1}}{N}
        +
        6\beta\Lambda\epsilon,
    \end{equation}

    \begin{equation}
        \frac{1}{N}
        \log
        Z_N (\beta,A_{0,\e}^{\tau_L,\tau_R})
        \les
        \sup_{x,y\in C_\epsilon(\tau_L,\tau_R)} M_{N,0}(x,0)
        +
        c\frac{\log N \epsilon^{-1}}{N}
        +
        6\beta\Lambda\epsilon,
    \end{equation}
    and
    \begin{equation}
        \frac{1}{N}
        \log
        Z_N (\beta,A_{N-1,\e}^{\tau_L,\tau_R})
        \les
        \sup_{x,y\in C_\epsilon(\tau_L,\tau_R)} M_{N,N-1}(0,y)
        +
        c\frac{\log N \epsilon^{-1}}{N}
        +
        6\beta\Lambda\epsilon,
    \end{equation}
    where

    \begin{equation}\label{eq:simplex_tauL_tau_R}
        C_{\epsilon}(\tau_L,\tau_R)
        =
        \left\{
        x,y\in[0,1-q^*]\times [0,q^*]: \ \abs{x-\tau_L}\les\e, \  \abs{y-\tau_R}\les\e 
        \right\}
    \end{equation}
\end{proposition}
The proof of this proposition takes place in a collection of lemmas. We define for $K=1,\dots,N-2$

\begin{equation}\label{eq:def_tildeMNK}
    \begin{aligned}
        \tilde{M}_{N,K}(\tau_L,\tau_R)
        = \ &
        \frac{N-K-1}{N}\Ffin{\hat{\mu}_{L,K}}\left( \frac{N}{N-K-1}\beta \tau_L \right)
        +
        \beta \lambda_{N-K}(1-\tau_L-\tau_R)\\
        &
        +
        \frac{K}{N}\Ffin{\hat{\mu}_{R,K}}\left( \frac{N}{K}\beta \tau_R \right),
    \end{aligned}
\end{equation}

\noindent and analogously for $K=0,N-1$. It is morally $M_{N,K}+\tfrac{1}{2}D$, an approximation of the free energy conditioned on $A_{K,\e}^{\tau_L,\tau_R}$. We start by bounding $Z_N(\beta,A_{K,\e}^{\tau_L,\tau_R})$ in terms of $\tilde{M}_{N,K}(\tau_L,\tau_R)$ and the measure of $A_{K,\e}^{\tau_L,\tau_R}$.

\begin{lemma}\label{lem:Z_Nslice_decompo}
    For any $\beta>0$, $\tau_L\in [0,1-q^*]$, $\tau_R\in [0,q^*]$ and $K=0,\dots N-1$ such that $A_{K,\e}^{\tau_L,\tau_R}$ has non-zero measure
    \begin{equation}
        \left|
        \frac{1}{N}
        \log
        Z_N (\beta,A_{K,\e}^{\tau_L,\tau_R})
        -
        \tilde{M}_{N,K}(\tau_L,\tau_R)
        -
        \frac{1}{N}\log \Es \left[ A_{K,\e}^{\tau_L,\tau_R} \right]
        \right|
        \leq
        6\beta\Lambda\epsilon.
    \end{equation}
\end{lemma}

\begin{proof}
     We start by fixing some $K=1,\dots,N-2$. Define the $\sigma$-algebra $\cF=\sigma(\tau_{L,K}(\sigma),\sigma_{N-K}^2,\tau_{R,K}(\sigma))$, and note that
    \begin{equation}
        Z_N(\beta,A_{K,\e}^{\tau_L,\tau_R})
        =
        \Es\left[
            \Es[e^{\beta H_N(\sigma)}\mid\cF]
            \indf_{A_{K,\e}^{\tau_L,\tau_R}(\sigma)}
        \right].
    \end{equation}
    Under $\Es[\cdot| \cF]$, the random variable $\sigma_{N-K}^2$ is a.s. constant and $(\sigma_i)_{i=1}^{N-K-1}$ and $(\sigma_i)_{i=N-K+1}^N$ are independent. Thus
    \begin{equation}
        \Es[e^{\beta H_N(\sigma)}\mid\cF]
        =
        \Es\left[e^{N\beta\sum_{i=1}^{N-K-1} \lambda_i\sigma_i^2}\mid\cF\right]
        e^{N\beta\lambda_{N-K}\sigma_{N-K}^2}
        \Es\left[e^{ N\beta\sum_{i=N-K+1}^N \lambda_i\sigma_i^2}\mid\cF\right].
    \end{equation}
    Furthermore, under $\Es[\cdot| \cF]$ the vectors $(\sigma_i)_{i=1}^{N-K-1}$ and $(\sigma_i)_{i=N-K+1}^N$ are uniform on $S_{N-K-2}(\tau_{L,K}^{1/2})$ and $S_{K-1}(\tau_{R,K}^{1/2})$ respectively. Thus    
    \begin{equation}
        \Es\left[e^{N\beta\sum_{i=1}^{N-K-1} \lambda_i\sigma_i^2}\mid\cF\right]
        =
        E_{N-K-2}\left[ e^{ N \beta \tau_{L,K} \sum_{i=1}^{N-K-1} \lambda_i\sigma_i^2} \right]
        \overset{\eqref{eq:F_nu_def},\eqref{eq:def_mu_hatLR}}{=}
        e^{(N-K-1)\Ffin{\hat{\mu}_{L,K}}\left(\frac{N}{N-K-1} \beta \tau_{L,K} \right)},
    \end{equation}
    and similarly
    \begin{equation}
        \Es\left[e^{ N\beta\sum_{i=N-K+1}^N \lambda_i\sigma_i^2}\mid\cF\right]
        \overset{\eqref{eq:F_nu_def},\eqref{eq:def_mu_hatLR}}{=}
        e^{K \Ffin{\hat{\mu}_{R,K}}\left(\frac{N}{K} \beta \tau_{R,K} \right)}.
    \end{equation}    
    Combining the above and using that $\sigma_{N-K}^2 = 1 - \tau_{L,K} - \tau_{R,K}$ we obtain
    \begin{equation}\label{eq:ZNA}
        Z_N(\beta,A_{K,\e}^{\tau_L,\tau_R})
        =
        \Es\left[
            e^{N\tilde{M}_{N,K}(\tau_{L,K}(\sigma),\tau_{R,K}(\sigma))
            }
        \indf_{A_{K,\e}^{\tau_L,\tau_R}}(\sigma)
        \right].
    \end{equation}
    The proof of the lemma for $K=1,\dots,N-2$ is complete by recalling that $\tilde{M}_{N,K}$ is $3\beta\Lambda$-Lipschitz. Indeed, $x\to\frac{N-K-1}{N}\Ffin{\hat{\mu}_{L,K}}\left(\frac{N}{N-K-1}\beta x\right)$ and $x\mapsto \frac{K}{N}\Ffin{\hat{\mu}_{R,K}}\left(\frac{N}{K}\beta x\right)$
    are $\beta\Lambda$-Lipschitz. Then, from the definition \eqref{eq:def_tildeMNK}, $\tilde{M}_{N,K}$ is $3\beta\Lambda$-Lipschitz. The proof for $K=0,N-1$ is analogous, using the corresponding definition of $\tilde{M}_{N,K}$.
\end{proof}

To conclude it remains to bound the measure of the sets $A_{K,\e}^{\tau_L,\tau_R}$ in term of the entropy.

\begin{lemma}\label{lem:measure_AK_bulk}
    There exits $c>0$ such that for any $0<\epsilon<\min(q^*,1-q^*)$, $\tau_L\in [0,1-q^*]$, $\tau_R\in [0,q^*]$ and $K=1,\dots N-2$
    \begin{equation}
        \frac{1}{N}\log \Es \left[ A_{K,\e}^{\tau_L,\tau_R} \right]
        \les
        - 
        \frac{1}{2}D(\alpha||u_L,u_R)
        +
        c\frac{\log N \epsilon^{-1}}{N},
    \end{equation}
    where $u_L=\min(\tau_L+\epsilon,1-q^*)$ and $u_R=\min(\tau_R+\epsilon,q^*)$.
\end{lemma}

\begin{proof}
    We start by noting that if $\tau_{L,K}(\sigma),\tau_{R,K}(\sigma)$ have joint density $p_{N,K}(x,y)$ under $E$ then
    \begin{equation}\label{eq:measure_set_as_integral}
        \Es \left[ A_{K,\e}^{\tau_L,\tau_R}\right]
        =
        \iint\limits_{ \substack{\abs{x-\tau_L}\les\e, \ x\les 1-q^*\\ \abs{y-\tau_R}\les\e, \ y\les q^*}} p_{N,K}(x,y) dx dy.
    \end{equation}
    For $K=1,\dots,N-2$ 
    \begin{equation}\label{eq:qL_qR_density}
        p_{N,K}(x,y)
        =
        C_{N,K} x^{\frac{N-K-3}{2}}y^{\frac{K-2}{2}}(1-x-y)^{-\frac12} \indf_{\{x>0; \ y>0; \ x+y < 1\}},
    \end{equation}
    (see \cite[Lemma~8]{BK} or \cite[Theorem~1]{harman2010}), where
    \begin{equation}\label{eq:C_NK_def}
        C_{N,K}
        =
        \frac{\Gamma\left(\frac{N}{2}\right)}{\Gamma\left(\frac{N-K-1}{2}\right)\Gamma\left(\frac{K}{2}\right)\Gamma\left(\frac{1}{2}\right)}.
    \end{equation}
    A control of $\tfrac{1}{N}\log C_{N,K}$ is given at the end of this section, in Lemma \ref{lem:bound_CNK}. Because both $u_L\ges \epsilon$ and $u_R\ges \epsilon$ we obtain, for $x,y$ in the integration domain of \eqref{eq:measure_set_as_integral}

    \begin{equation}
         p_{N,K}(x,y)
        \le
        C_{N,K} \cdot
        \epsilon^{-3}
        u_L^{(N-K)/2}u_R^{K/2} \cdot
        \frac{\xi_{K,N}(x,y)}{\sqrt{1-x-y}} \indf_{\{x+y < 1\}}.
    \end{equation}
    where 

    \begin{equation}
        \xi_{K,N}(x,y)
        =
        \begin{cases}
            1 \text{ if } 2\les K \les N-3,\\
            x^{-1/2} \text{ if } K=N-2,\\
            y^{-1/2} \text{ if } K=1.
        \end{cases}
    \end{equation}
    Using Lemma \ref{lem:bound_CNK} and that there exists $c>0$ such that   \begin{equation}\label{eq:uniform_bound_inverse_root_2D_integral}
        \int_0^t\int_0^{1-t}\frac{\xi_{K,N}(x,y)}{\sqrt{1-x-y}}  dx dy \les c \quad \forall t \in[0,1],
    \end{equation} yields
    
    \begin{equation}
        \frac{1}{N}\log \Es \left[ A_{K,\e}^{\tau_L,\tau_R} \right]
        \les
        - 
        \frac{1}{2}D(\alpha||u_L,u_R)
        +
        c\frac{\log N \epsilon^{-1}}{N},
    \end{equation}
    for some possibly larger constant $c$.
\end{proof}

\begin{lemma}\label{lem:measure_AK_K=0_N}
    There exits $c>0$ such that for any $0<\epsilon<\min(q^*,1-q^*)$, $\tau_L\in \pi_L$ and $\tau_R\in \pi_R$ 
    \begin{equation}
        \frac{1}{N}\log \Es \left[ A_{0,\e}^{\tau_L,\tau_R} \right]
        \les
        - 
        \frac{1}{2}D(0||u_L,0)
        +
        c\frac{\log N \epsilon^{-1}}{N},
    \end{equation}
    and
    \begin{equation}
        \frac{1}{N}\log \Es \left[ A_{N-1,\e}^{\tau_L,\tau_R} \right]
        \les
        - 
        \frac{1}{2}D(1||0,u_R)
        +
        c\frac{\log N \epsilon^{-1}}{N}
    \end{equation}
    where $u_L=\min(\tau_L+\epsilon,1-q^*)$ and $u_R=\min(\tau_R+\epsilon,q^*)$.
\end{lemma}

\begin{proof}
    If $\tau_R>\epsilon$ the proof is trivial for $K=0$, while for $\tau_L>\epsilon$ it is trivial for $K=N-1$. Indeed, in this range the sets $A_{0,\e}^{\tau_L,\tau_R}$ and $A_{N-1,\e}^{\tau_L,\tau_R}$ are empty by \eqref{eq:defAkqlqrdelta} and \eqref{eq:def_q_tau_LR_convention}. In the following, fix $K=0$, $\tau_R=0$. Recall that $\tau_{R,0}=0$ a.s and that under $\Es$, $\tau_{L,K}$ has density (see \cite[Lemma~8]{BK} or \cite[Theorem~1]{harman2010})
    \begin{equation}
        p_{N,0}(\tau_L)
            =
            \dfrac{\Gamma\left(\frac{N}{2}\right)}{\Gamma\left(\frac{N-1}{2}\right)\Gamma\left(\frac{1}{2}\right)}
            \tau_L^{\frac{N-3}{2}}
            (1-\tau_L)^{-\frac12}
            \indf_{\{0<\tau_L< 1\}}.
    \end{equation}
    Omitting details, using \eqref{eq:measure_set_as_integral} as in Lemma \ref{lem:measure_AK_bulk} yields
    
    \begin{equation}
        \frac{1}{N}\log \Es \left[ A_{0,\e}^{\tau_L,0} \right]
        \les
        - 
        \frac{1}{2}D(0||u_L,0)
        +
        c\frac{\log N \epsilon^{-1}}{N},
    \end{equation}
    and concludes the proof of the lemma for $K=0$. The result for $K=N-1$ follows analogously.
\end{proof}

We can now prove Proposition \ref{prop:sec4main}.

\begin{proof}[Proof of Proposition \ref{prop:sec4main}]
    The proof follows from combining Lemma \ref{lem:Z_Nslice_decompo}, \ref{lem:measure_AK_bulk} and \ref{lem:measure_AK_K=0_N}.
\end{proof}

We conclude this section with the proof of a technical result.

\begin{lemma}\label{lem:bound_CNK}
    There exists a positive constant $c$ such that for any $N$ and $K\in [0,N-1]$,
    \begin{equation}
        \left|
        \frac{1}{N}\log C_{N,K}
        +
        \frac{1-\alpha}{2}
        \log (1-\alpha)
        +
        \frac{\alpha}{2}
        \log \alpha
        \right|
        \les
        \frac{c}{N}\log N,
    \end{equation}
    where for $K=1,\dots,N-2$ $C_{N,K}$ is defined in \eqref{eq:C_NK_def} and $C_{N,0}=C_{N,N-1}=\frac{\Gamma\left(\frac{N}{2}\right)}{\Gamma\left(\frac{N-1}{2}\right)\Gamma\left(\frac{1}{2}\right)}$ .
\end{lemma}

\begin{proof}
  
    By Stirling's formula, there exist positive universal constant $c_1$ and $c_2$ such that $c_1\le\Gamma(x)/(x^{-\frac{1}{2}}(\frac{x}{e})^{x})\le c_2$
    for all $x\ge\frac{1}{2}$. Thus letting $\delta=N^{-1}$, we have for $K=1,\dots,N-2$
    \begin{equation}
        \left|
        \frac{1}{N}\log C_{N,K}
        +
        \frac12
        (1 - \alpha - \delta)
        \log (1-\alpha - \delta)
        +
        \frac12
        \alpha
        \log \alpha
        \right|
        \les
        \frac{c}{N}\log N
    \end{equation}
    for $\alpha=K/N \le 1 - 2\delta$. Using that $x \to (1-x) \log (1-x)$ is $L$-Lipschitz on $[0,1-\delta]$, with constant $L=1 + \log \delta^{-1}$ concludes the proof of the lemma for $K=1,\dots,N-2$. The proof for $K=0$, $K=N-1$ follows analogously.
\end{proof}

\section{Variational upper bound of the truncated free energy}\label{sec:intermediate_TAP}

The goal of this section is to bound $F_N(\beta)$ using a control of the free energy restricted to the sets $A^{\tau_L,\tau_R}_{K,\e}$, which cover the whole sphere $S^{N-1}$. We start by introducing necessary notations. Let $I(x)$ denote the inverse of the CDF of $\mu$, so that
\begin{equation}\label{eq:eqdef_Ialpha}
    \mu( [\lm, I(x)] ) = x
    \quad
    \forall x \in [0, 1].
\end{equation}
Let
\begin{equation}\label{eq:eqdef_muLalpha_muRalpha}
    \mu_{L,\alpha}( A )
    =
    (1-\alpha)^{-1}
    \mu\big( A \cap (\lm,I(1-\alpha)) \big),
    \qquad
    \mu_{R,\alpha}(A)
    =
    \alpha^{-1}
    \mu\big(A \cap (I(1-\alpha), \lM )\big)
\end{equation}
be the truncated probability measures obtained by splitting $\mu$ into parts left and right of $I(1-\alpha)$. We also introduce two master error terms, encapsulating the speed of convergence of $\mu_N$ towards $\mu$ by adding several terms. It will make the following bounds more compact, but detailed bounds (not using the whole master error) are easily retrievable. 
\begin{equation}\label{eq:def_master_error}
\begin{aligned}
     \mathrm{Err}_N^{(*)} \big(\mu_N,\mu\big)
    =& \ 
    \sqrt{W_1(\mu_N,\mu)}
    +
    \sup_{1\les k \les N} |\lambda_k-I(k/N)|
    +
    \sup_{ |t-v|\les 1/N} |I(t)-I(v)|
    +
    \sqrt{\frac{1}{N}}.
\end{aligned}
\end{equation}
We prove in Lemma \ref{lem:error_term_goes_to_zero} that this error term indeed goes to $0$. We also let 

\begin{equation}\label{eq:def_master_error_bis}
\begin{aligned}
     \mathrm{Err}_N^{(**)} \big(\mu_N\big)
    =& 
     | \lambda_{\lceil \delta_N N \rceil}-\lambda_1| 
        +
        | \lambda_{1}-\lm|
        +
        |\lambda_N-\lM|
        +
        |\lM-I(1-\delta_N)|,
\end{aligned}
\end{equation}
where $\delta_N=\log(1/\mathrm{Err}_N^{(*)}(\mu_N,\mu))^{-1/2}\vee\frac1{\sqrt N}$. Because $\delta_N$ goes to $0$ it is direct that $\mathrm{Err}_N^{(**)}\big(\mu_N\big)$ goes to $0$. Finally, for $\alpha\in(0,1)$ we define a TAP free energy with top and bottom eigenvalue contributions separated as

\newcommand{\fTAPalpha}{M_\TAP}

\begin{equation}\label{eq:def_ftap}
    \begin{aligned}
        \fTAPalpha(\beta, \tau_L, \tau_R, \alpha)
        \coloneqq \ & 
        (1-\alpha)
        \Flim{\mu_{L,\alpha}}
        \left(\frac{ \beta \tau_L }{1-\alpha}\right)
        +
        \beta I(1-\alpha) (1- \tau_L - \tau_R)
        +
        \alpha
        \Flim{\mu_{R,\alpha}}
        \left(
            \frac{ \beta \tau_R }{\alpha}
        \right)\\
        &-
        \frac{1}{2} D(\alpha||\tau_L,\tau_R),
    \end{aligned}
\end{equation}
while for $\alpha=0$ we use the convention
\begin{equation}\label{eq:def_ftap_alpha=0}
    \fTAPalpha(\beta, \tau_L, \tau_R, 0)
    =
    \begin{cases}                
        \Flim{\mu}( \beta \tau_L )
        +
        \beta \lM (1 - \tau_L)
        -
        \frac{1}{2} D(0||\tau_L,0)
        &
        \text{ if } \tau_R=0,\\
        -\infty &
        \text{ otherwise},
    \end{cases}
\end{equation}
and for $\alpha=1$
\begin{equation}\label{eq:def_ftap_alpha=1}
    \fTAPalpha(\beta, \tau_L, \tau_R, 1)
    =
    \begin{cases}                
        \Flim{\mu}( \beta \tau_R )
        +
        \beta \lm (1-\tau_R)
        -
        \frac{1}{2} D(1||0,\tau_R)
        &
        \text{ if } \tau_L=0,\\
        -\infty &
        \text{ otherwise}.
    \end{cases}
\end{equation}

\noindent We define $\mathfrak{S}_{\beta}$ the range of parameters $\alpha,\tau_L,\tau_R$ 
\begin{equation}\label{eq:def_sup_set_sec5}
    \mathfrak{S}_{\beta}
    :=
    [0,1] \times [0,1-q^*]
    \times
    [0,q^*].
\end{equation}

The main result of this section is as follows.
\begin{proposition}\label{prop:aux}
 For any $B>\beta_c$ and $\beta \in (\beta_c,B]$
    \begin{equation}    
        F_N(\beta)
        \les
        \sup_{(\alpha,\tau_L,\tau_R) \in \mathfrak{S}_{\beta}}
        \fTAPalpha(\beta,\tau_L,\tau_R,\alpha)
        +
        \mathrm{Err}_N(B,\mu_N),
    \end{equation}
    where $\fTAPalpha$ is defined in \eqref{eq:def_ftap} and $\mathrm{Err}_N(B,\mu_N)$ goes to $0$ with $N$.  
\end{proposition}

Because we aim to bound the free energy by the contribution of each of the sets $A^{\tau_L,\tau_R}_{K,\e}$ we build on the results of the previous section. There, $Z_N(\beta,A^{\tau_L,\tau_R}_{K,\e})$ was bounded with respect to the truncated free energies and $M_{N,K}$. Hence, our first task is to give explicit bounds for $M_{N,K}$. The main technical result of this section is the following.

\begin{lemma}\label{lem:sec5main}
    For any $\beta>0$, $ 0 \les K\les N-1$ and  $\tau_L\in[0,1-q^*]$, $\tau_R\in [0,q^*]$ with $\tau_R=0$ when $K=0$, and $\tau_L=0$ when $K=N-1$, it holds that
    \begin{equation}
        M_{N,K}(\tau_L,\tau_R)
        \les
        \fTAPalpha(\beta, \tau_L, \tau_R, \alpha)
        +
        \mathrm{Err}_N^{(5)}
        \big(\beta,\mu_N\big),
    \end{equation}
    where $\alpha=K/N$ if $K=0,\dots,N-2$, $\alpha=1$ if $K=N-1$, and $\mathrm{Err}_N^{(5)}$ is given by 
    
    \begin{equation}
    \begin{aligned}
        \mathrm{Err}_N^{(5)} \big(\beta,\mu_N\big)
        =& \ 
        \frac{3\beta\Lambda}{\sqrt{N}}
        +
        256 (\beta\Lambda+1)^2 \sqrt{\mathrm{Err}_N^{(*)}(\mu_N,\mu)}
        +
        \beta\mathrm{Err}_N^{(**)} \big(\mu_N\big).
    \end{aligned}
    \end{equation}
\end{lemma}

The proof of Lemma \ref{lem:sec5main} is given after a collection of lemmas bounding each component of $M_{N,K}$ in different regimes of the parameters.



Recalling the definition of $M_{N,K}$ given in \eqref{eq:def_MNK}, we start by bounding $\Ffin{\hat{\mu}_{L,K}}$ in terms of $\Flim{\mu_{L,\alpha}}$. 

\begin{lemma}\label{lem:FL_UB_delta_macro}
    For any $\beta>0$, $\delta>2/N$, $0 \les K\les (1-\delta)N$ and $\tau_L\in(0,1-q^*]$ we have
    \begin{equation}
        \frac{N-K-1}{N}\Ffin{\hat{\mu}_{L,K}}\left(
            \frac{N}{N-K-1}\beta \tau_L
        \right)
        \les
        (1-\alpha)
        \Flim{\mu_{L,\alpha}}
        \left(
            \frac{ \beta \tau_L }{1-\alpha}
        \right)
        +
        \mathrm{Err}_N^{(1)} \big(\beta,\delta,\mu_N\big)
        ,
    \end{equation}
    where $\alpha=K/N\les 1-\delta$  and the error term $\mathrm{Err}_N^{(1)}$ is given by
    \begin{equation}\label{eq:lemma_FLK_main_regime}
        \begin{aligned}
            \mathrm{Err}_N^{(1)} \big(\beta,\delta,\mu_N\big)
            =&
            \frac{1}{2N}
            +
            \frac{2\Lambda \beta}{\delta N}
            +
            C_{\beta/\delta}
            \left(\frac{\beta}{\delta}+1\right)
            \mathrm{Err}_N^{(*)}(\mu_N,\mu)
        \end{aligned}
    \end{equation}
    for $C_{\beta/\delta}=128[1+\beta\Lambda/\delta]e^{8(1+\beta\Lambda/\delta)}$.
\end{lemma}

\begin{proof}
    We apply Lemma \ref{lem:fe_upper_bound} with $\nu=\hat{\mu}_{L,K}$, $\mu = \mu_{L,\alpha}$, $N-K-1$ in place of $N$, and $\frac{N}{N-K-1}\beta\tau_L$ in place of $\beta$. This yields
    \begin{equation}
        \frac{N-K-1}{N}\Ffin{\hat{\mu}_{L,K}}\left(\frac{N \beta \tau_L}{N-K-1}\right)
        \les
        (1-\alpha')\Flim{\mu_{L,\alpha}}
        \left(
            \frac{\beta \tau_L}{1-\alpha'}
        \right)
        +
        (1-\alpha')\mathrm{Err}_N,
    \end{equation}
    where $\alpha' = \frac{K+1}{N}$ and the error term $\mathrm{Err}_N$ is given by
    \begin{equation}
        \begin{aligned}
            \mathrm{Err}_N
            = & \            
            C
            \left(
                \tfrac{\beta\tau_L}{1-\alpha'} + 1
            \right) \left(
                \sqrt{W_1\left(\hat{\mu}_{L,K}, \mu_{L,\alpha}\right)} 
                +
                \left|
                    \lambda_{N-K-1}
                    -
                    I(1-\alpha)
                \right|
                +
                \sqrt{\tfrac{1}{(1-\alpha')N}}
            \right).
        \end{aligned}
    \end{equation}
    It remains to use that $\Flim{\mu_{L,\alpha}}(z)$ is both bounded by $z\Lambda $ and $\Lambda$-Lipschitz to get

    \begin{equation}\label{eq:FLalpha_lipschitz}
        (1-\alpha')\Flim{\mu_{L,\alpha}}
        \left(\frac{\beta \tau_L}{1-\alpha'}\right)
        -
        (1-\alpha)\Flim{\mu_{L,\alpha}}
        \left(\frac{\beta \tau_L}{1-\alpha}\right)
        \leq
        \frac{2\Lambda \beta}{\delta N}.
    \end{equation}
    To conclude, note that
    
    \begin{equation}
        \left|
        \lambda_{N-K-1}
        -
        I(1-\alpha)
        \right|
        \les
        \sup_{1\les k \les N} |\lambda_k-I(k/N)|
        +
        \sup_{ |h-u|\les 1/N} |I(h)-I(u)|,
    \end{equation}
    and use the triangle inequality: $W_1\left(\hat{\mu}_{L,K}, \mu_{L,\alpha}\right)\les W_1\left(\hat{\mu}_{L,K}, \mu_{L,\alpha'}\right)+W_1\left(\mu_{L,\alpha'}, \mu_{L,\alpha}\right)$. It is a direct consequence of \cite[equation 2.48]{villani2021topics} that
    \begin{equation}
        (1-\alpha')W_1\left(\hat{\mu}_{L,K}, \mu_{L,\alpha'}\right)\les W_1\left(\mu_N, \mu\right),
    \end{equation}
    while a simple coupling argument yields $W_1\left(\mu_{L,\alpha'}, \mu_{L,\alpha}\right)\leq \frac{2\Lambda}{(1-\alpha')N}$,
    leading to the error bound $\mathrm{Err}^{(1)}$.
\end{proof}

\noindent Next, we bound $\Ffin{\hat{\mu}_{R,K}}$ in terms of $\Flim{\mu_{R,\alpha}}$. 

\begin{lemma}\label{lem:FR_UB_delta_macro}
    For any $\beta>0$, $\delta>0$, $\delta N\les K\les N-1$, $\tau_R\in [0,q^*]$
    \begin{equation}
        \frac{K}{N}\Ffin{\hat{\mu}_{R,K}}\left( \frac{N}{K}\beta \tau_R \right) 
        \les
        \alpha
        \Flim{\mu_{R,\alpha}}
        \left(
            \frac{ \beta \tau_R }{\alpha}
        \right)
        +
        \mathrm{Err}_N^{(2)}\big(\beta,\delta,\mu_N\big),
    \end{equation}
    where $\alpha=K/N\ges \delta$ and the error term $\mathrm{Err}_N^{(2)}$ is given by
    \begin{equation}
        \mathrm{Err}_N^{(2)} \big(\beta,\delta,\mu_N\big)
        =
         C_{\beta/\delta} \left(\frac{\beta}{\delta}+1\right)
         \mathrm{Err}_N^{(*)}(\mu_N,\mu)
    \end{equation}
    for $C_{\beta/\delta}=128[1+\beta\Lambda/\delta]e^{8(1+\beta\Lambda/\delta)}$.
\end{lemma}

\begin{proof}
    The proof is similar to that of Lemma \ref{lem:FL_UB_delta_macro}, without the substitution of $\alpha'$ by $\alpha$. Hence, we do not give all details.  We apply Lemma \ref{lem:fe_upper_bound} with $\nu=\hat{\mu}_{R,K}$, $\mu = \mu_{R,\alpha}$, $K$ in place of $N$, and $\frac{N}{K}\beta\tau_R$ in place of $\beta$.
\end{proof}

It remains to give a bound in the corner cases of $\alpha\coloneqq K/N$ either close to $0$ or close to $1$. To this effect, we use ad-hoc bounds to obtain the following lemmas.

\begin{lemma}\label{lem:FL_UB_delta_micro}
    For any $\beta>0$, $2/N<\delta< \mu \big( [\lm, \lm+1]\big)$, $(1-\delta)N \les K\les N-2$ and $\tau_L\in(0,1-q^*]$ we have
    \begin{equation}
        \frac{N-K-1}{N}\Ffin{\hat{\mu}_{L,K}}\left(
            \frac{N}{N-K-1}\beta \tau_L
        \right)
        \les
        (1-\alpha)
        \Flim{\mu_{L,\alpha}}
        \left(
            \frac{ \beta \tau_L }{1-\alpha}
        \right)
        +
        \mathrm{Err}_N^{(3)} \big(\beta,\delta,\mu_N\big)
        ,
    \end{equation}
    where $\alpha=K/N\ges 1-\delta$  and the error term $\mathrm{Err}_N^{(3)}$ is given by
    \begin{equation}
        \mathrm{Err}_N^{(3)} \big(\beta,\delta,\mu_N\big)
        =
        \beta  | \lambda_{\lceil \delta N \rceil}-\lambda_1| 
        +
        \beta  | \lambda_{1}-\lm|
    \end{equation}
\end{lemma}

\begin{proof}
    First, from the definition \eqref{eq:F_nu_def}, for any $\beta>0$, any $1\les K\les N-1$ and any $\tau_L\in (0,1-q^*]$ we have
    
    \begin{equation}\label{eq:adhoc_FL_UB}
         \frac{N-K-1}{N}\Ffin{\hat{\mu}_{L,K}}\left(
            \frac{N}{N-K-1}\beta \tau_L
        \right)
        \les
        \beta \tau_L \lambda_{\lceil \delta N \rceil}.
    \end{equation}
    Next, Corollary \ref{cor:free_energy_trivial_bounds} 
    evaluated at $\mu=\mu_{L,\alpha}$ and $z=\tfrac{\beta\tau_L}{1-\alpha}$ yields

    \begin{equation}
        (1-\alpha)\Flim{\mu_{L,\alpha}}
        \left(
            \frac{ \beta \tau_L }{1-\alpha}
        \right)
        \geq 
        \beta \tau_L \lm,
    \end{equation}
    where we used that the support of $\mu=\mu_{L,\alpha}$ is $[\lm,I(1-\alpha)]$. The proof of the lemma follows.

\end{proof}

\begin{lemma}\label{lem:FR_UB}
    For any $\beta>0$, any $0<\delta\leq \mu\big([\lM-1,\lM]\big)$, any $1\les K\les \delta N$ and $\tau_R\in [0,q^*]$, letting $\alpha=K/N\les \delta$ we have
    \begin{equation}
        \frac{K}{N}\Ffin{\hat{\mu}_{R,K}}\left(\frac{N}{K}\beta \tau_R\right)
        \les
        \alpha\Flim{\mu_{R,\alpha}}( \beta \tau_R/\alpha ) +
        \mathrm{Err}_N^{(4)} \big(\beta,\delta,\mu_N\big),
    \end{equation}
    where the error term $\mathrm{Err}_N^{(4)}$ is given by

    \begin{equation}
        \mathrm{Err}_N^{(4)}\big(\beta,\delta,\mu_N\big)
        =
        \beta  |\lambda_N-\lM|
        +
        \beta  |\lM-I(1-\delta)|
    \end{equation}
\end{lemma}

\begin{proof}
    First, from the definition \eqref{eq:F_nu_def}, for any $\beta>0$, any $1\les K\les N-1$ and any $\tau_R\in [0,q^*]$ we have
    \begin{equation}\label{eq:adhoc_FR_UB}
        \frac{K}{N}\Ffin{\hat{\mu}_{R,K}}\left(\frac{N}{K}\beta \tau_R\right)\les \beta \tau_R \lambda_N.
    \end{equation}
    Next, Corollary \ref{cor:free_energy_trivial_bounds} evaluated at $\mu=\mu_{R,\alpha}$ and $z=\tfrac{\beta \tau_R}{\alpha}$ yields

    \begin{equation}
        \alpha \Flim{\mu_{R,\alpha}}\left( \frac{\beta \tau_R}{\alpha} \right)
        \geq
        \beta \tau_R I(1-\alpha),
    \end{equation}
    where we used that the support of $\mu=\mu_{R,\alpha}$ is $[I(1-\alpha),\lM]$. The proof of the lemma follows.

    
\end{proof}

We are now ready to prove Lemma \ref{lem:sec5main}.

\begin{proof}[Proof of Lemma \ref{lem:sec5main}]
    First, fix some $2/N<\delta< \mu \big( [\lm, \lm+1]\big)\wedge\mu\big([\lM-1,\lM]\big)$. We start by treating the case $K\in [1,N-2]$. In the definition of $M_{N,K}$ given in \eqref{eq:def_MNK}, the terms $\Ffin{\hat{\mu}_{L,K}}$ and $\Ffin{\hat{\mu}_{R,K}}$ appear. Depending on the range of $\alpha=K/N$, we proved the corresponding upper bounds in Lemmas \ref{lem:FL_UB_delta_macro}, \ref{lem:FR_UB_delta_macro}, \ref{lem:FL_UB_delta_micro} and \ref{lem:FR_UB}. By combining these bounds and recognizing the expression \eqref{eq:def_ftap} in the resulting bound, we obtain, for $\tau_L\in[0,1-q^*]$ and $\tau_R\in[0,q^*]$   \begin{equation}\label{eq:Mn_is_Ftap_minus_CNK}
        M_{N,K}(\tau_L,\tau_R)
        \les
        \fTAPalpha(\beta, \tau_L, \tau_R, \alpha)
        +
        \sum_{i=1}^4 \mathrm{Err}_N^{(i)}(\beta,\delta,\mu_N).
    \end{equation}
    
    Turning to the case $K=0$, we similarly apply Lemma \ref{lem:FL_UB_delta_macro} to the r.h.s. of \eqref{eq:def_MN0}, recognize the definition of $\fTAPalpha(\beta,\tau_L,\tau_R,0)$ from \eqref{eq:def_ftap_alpha=0} and conclude that    \begin{equation}\label{eq:Mn_is_Ftap_minus_CN0}
        M_{N,0}(\tau_L)
        \les
        \fTAPalpha(\beta, \tau_L, \tau_R, 0)
        +
        \mathrm{Err}_N^{(1)} \big(\beta,\delta,\mu_N\big).
    \end{equation}
    The same holds for $K=N-1$, and combining results we prove that

    \begin{equation}
        M_{N,K}(\tau_L,\tau_R)
        \les
        \fTAPalpha(\beta, \tau_L, \tau_R, \alpha)
        +
        \sum_{i=1}^4\mathrm{Err}_N^{(i)} \big(\beta,\delta,\mu_N\big),
    \end{equation}
    and
    \begin{equation}\label{eq:def_err7}
    \begin{aligned}
        \sum_{i=1}^4\mathrm{Err}_N^{(i)} \big(.\big)
        = & \ 
       \frac{1}{2N}
        +
        \frac{2\Lambda \beta}{\delta N}
        +
        256[1+\beta\Lambda/\delta] e^{8(1+\beta\Lambda/\delta)}
        \left(\frac{\beta}{\delta}+1\right)
        \mathrm{Err}_N^{(*)}(\mu_N,\mu)\\
        &+
        \beta \left(  | \lambda_{\lceil \delta N \rceil}-\lambda_1| 
        +
        | \lambda_{1}-\lm|
        +
        |\lambda_N-\lM|
        +
        |\lM-I(1-\delta)|
        \right).
    \end{aligned}
    \end{equation}
    It remains to choose $\delta_N=\log(1/\mathrm{Err}_N^{(*)}(\mu_N,\mu))^{-1/2}\vee\frac1{\sqrt N}$ to conclude the proof of the lemma. Remark that this choice verifies $2/N<\delta_N<\mu \big( [\lm, \lm+1]\big)\wedge\mu\big([\lM-1,\lM]\big)$ for $N$ large enough because $\mathrm{Err}_N^{(*)}(\mu_N,\mu)$ vanishes (see Lemma \ref{lem:error_term_goes_to_zero}).

\end{proof}

\subsection{Proof of the variational upper bound}

It is now possible to combine the bound on $M_{N,K}$ with the analysis of $F_N$ to prove the main result of this section, Proposition \ref{prop:aux}.

\begin{proof}[Proof of Proposition \ref{prop:aux}]
    We start by fixing some $\beta>\frac{\beta_c}{1-1/N}$ and $2/N<\delta<\mu \big( [\lm, \lm+1]\big)\wedge\mu\big([\lM-1,\lM]\big)$. Then, we fix $\epsilon=1/N$ (which satisfies $\epsilon\leq q^*$ from the choice of $\beta$) and recall \eqref{eq:covering_sphere_AK}, that is, a collection of at most $N^3$ sets of the form $A_{K,1/N}^{\tau_L,\tau_R}$ covers $S^{N-1}$. Then, putting together the results of Proposition \ref{prop:sec4main} and Lemma \ref{lem:sec5main} leads to 
    \begin{equation}
        F_N(\beta)
        \les
        \sup_{(\alpha,\tau_L,\tau_R) \in \mathfrak{S}_{\beta}}
        \fTAPalpha(\beta,\tau_L,\tau_R,\alpha)
        +
        \mathrm{Err}_N^{(5)} \big(\beta,\mu_N\big)
        +
        \frac{c}{N}\log N.
    \end{equation}
    for $\beta_c\les\beta\les\frac{\beta_c}{1-1/N}$, use Theorem \ref{thm:sec2_main} to get

    \begin{equation}
        F_N(\beta)
        \les
        \Flim{\mu}(\beta)+\mathrm{Err}^{(*)}(\mu_N\mu)
        =
        \fTAPalpha(\beta,1-q^*,0,0)+\mathrm{Err}^{(*)}(\mu_N\mu).
    \end{equation}
    Because $\mathrm{Err}_N^{(5)} \big(\beta,\mu_N\big)+\mathrm{Err}_N^{(*)} \big(\mu_N,\mu\big)
    +
    \frac{c}{N}\log N$ is increasing in $\beta$, fixing

    \begin{equation}
        \mathrm{Err}_N^{(6)}(B,\mu_N)
        = 
        \mathrm{Err}_N^{(5)} \big(B,\mu_N\big)
        +
        \mathrm{Err}_N^{(*)} \big(\mu_N,\mu\big)
        +
        \frac{c}{N}\log N.
    \end{equation}
    concludes the proof of Proposition \ref{prop:aux}.
\end{proof}

We conclude the section with the following technical lemma, proving that $\mathrm{Err}_N^{(*)}$ vanishes.

\begin{lemma}\label{lem:error_term_goes_to_zero}
    The master error term $\mathrm{Err}_N^{(*)}$ defined in \eqref{eq:def_master_error} goes to $0$.
\end{lemma}

\begin{proof}
    Recalling the definition \eqref{eq:def_master_error}, we start by noting that $\sup_{ |h-u|\les 1/N} |I(h)-I(u)|$ is the modulus of continuity, because $I$ is absolutely continuous it goes to $0$. Then, it suffices to check that

    \begin{equation}
        \lim_{N\to \infty}\sup_{1\les k\les N} \left\{
        \left| \lambda_{k}-I(k/N) \right| \right\}
        =
        0.
    \end{equation}
    We use \cite[Theorem 1.2]{bogoya2016convergence}, with the caveat that their result only holds if the $\lambda_i$ are all in the support of $\mu$. We deal with this by considering the modified eigenvalues $\tilde{\lambda}_i$ defined by

    \begin{equation}
    \tilde{\lambda}_i
    =
    \begin{cases}
        \lm \quad \text{if } \lambda_i<\lm\\
        \lM \quad \text{if } \lambda_i>\lM\\
        \lambda_i \quad \text{otherwise.}
    \end{cases}
    \end{equation}
    Doing so, \cite[Theorem 1.2]{bogoya2016convergence} can be applied to obtain
    
    \begin{equation}
        \lim_{N\to \infty}\sup_{1\les k\les N} \left\{
        \left| \tilde{\lambda}_{k}-I(k/N) \right| \right\}
        =
        0.
    \end{equation}
    It remains to note that

    \begin{equation}
        \sup_{1\les k\les N} \left\{
        \left| \lambda_{k}-I(k/N) \right| \right\}
        \les
        \sup_{1\les k\les N} \left\{
        \left| \tilde{\lambda}_{k}-I(k/N) \right| \right\}
        +
        |\lambda_1-\lm|
        +
        |\lambda_N-\lM|
    \end{equation}
\end{proof}

\section{The zero-field TAP formula}\label{sec:where_is_TAP_max}

In the previous section we proved a variational formula, linking the free energy and the supremum of the $\TAP$ free energy over some given intervals. In this section we give an explicit expression for this supremum. First, we give a simplified expression for the TAP free energy, stemming directly from \eqref{eq:def_ftap}.

\begin{lemma}\label{lem:ftap_explicit_pre}
    For any $\beta>0, \ \alpha \in (0,1)$ and $\tau_L,\tau_R>0$ s.t. $\tau_L+\tau_R\le1$
    \begin{equation}\label{eq:def_ftap_explicit_pre}
    \begin{aligned}
        \fTAPalpha(\beta, \tau_L, \tau_R, \alpha)
        = &\ 
        -\frac{1}{2}
        -
        \frac{1}{2} \log (2\beta)        
        -
        (1-\alpha)L^*_{\mu_{L,\alpha}}\left(\frac{\beta\tau_L}{1-\alpha}\right)
        -
        \alpha L^*_{\mu_{R,\alpha}}\left(\frac{\beta\tau_R}{\alpha}\right)\\
        &+
        \beta I(1-\alpha) 
        (1 - \tau_L - \tau_R).
    \end{aligned}
    \end{equation}
    Also for $\alpha=0,\tau_L\in(0,1]$
    \begin{equation}\label{eq:f_TAP_q_star_alpha=0_pre}
    \begin{aligned}
        \fTAPalpha(\beta, \tau_L, 0, 0)
        =& \ 
        -
        \frac{1}{2}
        -
        \frac{1}{2} \log (2\beta)
        -L^*_{\mu_{L,\alpha}}\left(\beta\tau_L\right)
        +
        \beta \lM (1-\tau_L).
    \end{aligned}
    \end{equation}
\end{lemma}

\begin{proof}
    The claim \eqref{eq:def_ftap_explicit_pre} follows from substituting \eqref{eq:F_def} into \eqref{eq:def_ftap}, and noting that some of the resulting terms cancel $-\frac{1}{2} D(\alpha||\tau_L,\tau_R)$ (recall \eqref{eq:def_entropy}). The claim \eqref{eq:f_TAP_q_star_alpha=0_pre} similarly follows from \eqref{eq:def_ftap_alpha=0}, \eqref{eq:F_def} and \eqref{eq:def_entropy_alpha=0}.
\end{proof}

Define
\begin{equation}\label{eq:def_G_LR}
    \begin{aligned}
        G_L(\beta,\tau,\alpha)
        &= \inf_{\eta\ges I(1-\alpha)}
        \left\{
            \beta\tau\eta
            -
            \frac{1}{2}\int^{I(1-\alpha)}_{\lm}\log(\eta-x)\mu(dx)
        \right\},\\
        G_R(\beta,\tau,\alpha)
        &= \inf_{\eta\ges\lM}\left\{
            \beta\tau\eta
            -
            \frac{1}{2}
            \int^{\lM}_{I(1-\alpha)}
            \log(\eta-x)\mu(dx)
        \right\}.
    \end{aligned}
\end{equation}

\begin{lemma}\label{lem:ident_F_star_LR_G_LR}
    For all $\beta,\tau>0$
    \begin{equation}\label{eq:ident_F_star_LR_G_LR}
        -(1-\alpha)L^*_{\mu_{L,\alpha}}\left(\frac{\beta\tau}{1-\alpha}\right)
        \overset{\forall \alpha\in[0,1)}{=}
        G_L(\beta,\tau,\alpha),
        \qquad
        -\alpha L^*_{\mu_{R,\alpha}}\left(\frac{\beta\tau}{\alpha}\right)
        \overset{\forall \alpha\in(0,1]}{=}
        G_R(\beta,\tau,\alpha).
    \end{equation}
\end{lemma}

\begin{proof}
    Recall that $\mu_{L,\alpha}(dx)=\frac{1}{1-\alpha}1_{[\lm,I(1-\alpha)]}\mu(dx)$
    and $\mu_{R,\alpha}(dx)=\frac{1}{\alpha}1_{[I(1-\alpha),\lM]}\mu(dx)$. Using this, \eqref{eq:def_log_potential}, \eqref{eq:LegendreLstar}, and  \eqref{eq:def_G_LR} yields \eqref{eq:ident_F_star_LR_G_LR}.
\end{proof}

\begin{lemma}\label{lem:f_TAP_q_star_alpha=0}
    For any $\beta>0,\ \alpha \in (0,1)$ and $\tau_L,\tau_R>0$ s.t. $\tau_L+\tau_R\leq1$
    \begin{equation}\label{eq:def_ftap_explicit}
        \fTAPalpha(\beta, \tau_L, \tau_R, \alpha)
        =
        -\frac{1}{2}
        -
        \frac{1}{2} \log (2\beta)        
        +
        G_L(\beta,\tau_L,\alpha)
        +
        \beta I(1-\alpha) 
        (1 - \tau_L - \tau_R)
        +
        G_R(\beta,\tau_R,\alpha).
    \end{equation}
    Also for $\alpha=0$
    \begin{equation}\label{eq:f_TAP_q_star_alpha=0}
        \fTAPalpha(\beta, 1-q^\ast, 0, 0)
        =
        -
        \frac{1}{2}
        -
        \frac{1}{2} \log (2\beta)
        +
        G_L(\beta,1-q^\ast,0)
        +
        \beta q^\ast \lM.
    \end{equation}
\end{lemma}

\begin{proof}
    These follow by combining \Cref{lem:ftap_explicit_pre} and \Cref{lem:ident_F_star_LR_G_LR}.
\end{proof}

\begin{lemma}\label{lem:F_TAP_max_tau}
    For any $\beta>0,\alpha\in(0,1)$ the following holds.
    \begin{enumerate}
        \item
            If $\tau_R>0$:
            \begin{equation}\label{eq:F_TAP_max_tau_L}
                \sup_{\tau_L\in[0,1-q^*]} \fTAPalpha(\beta,\tau_L,\tau_R,\alpha)
                =
                \fTAPalpha(\beta,1-q^\ast,\tau_R,\alpha).
            \end{equation}
        \item
            for any $\tau_L>0$:
            \begin{equation}\label{eq:F_TAP_max_tau_R}
                \sup_{\tau_R\in [0,q^*]} \fTAPalpha(\beta,\tau_L,\tau_R,\alpha)
                =
                \fTAPalpha(\beta,\tau_L,q^*,\alpha).
            \end{equation}
        \item           
            \begin{equation}\label{eq:F_TAP_max_tau_LR}
                \sup_{\tau_L\in[0,1-q^*],\tau_R\in[0,q^*]}
                \fTAPalpha(\beta,\tau_L,\tau_R,\alpha)
                =
                \fTAPalpha(\beta,1-q^*,q^*,\alpha).
            \end{equation}            
    \end{enumerate}
\end{lemma}

\begin{proof}
    The terms in \eqref{eq:def_ftap_explicit} that depend on $\tau_L$ are $G_L(\beta,\tau_L,\alpha)-\beta I(1-\alpha)\tau_L$, and this expression equals    \begin{equation}\label{eq:F_TAP_max_tau_step_1}    
        \inf_{\eta\ges I(1-\alpha)}
        \left\{
            \beta\tau_L(\eta-I(1-\alpha))
            -
            \frac{1}{2}\int_{\lm}^{I(1-\alpha)}
            \log(\eta-x)\mu(dx)
        \right\}.
    \end{equation}    
    The term $\beta\tau_L(\eta-I(1-\alpha))$ is non-decreasing in $\tau_L$ for all $\eta$ in the range of the infmimum, so \eqref{eq:F_TAP_max_tau_step_1} is clearly non-decreasing in $\tau$. This proves \eqref{eq:F_TAP_max_tau_L}.
    
    Similarly, the terms in \eqref{eq:def_ftap_explicit} that depend on $\tau_R$ are $ G_R(\beta,\tau_R,\alpha)-\beta I(1-\alpha)\tau_R$, and this expression is non-decreasing in $\tau_R$, proving \eqref{eq:F_TAP_max_tau_R}.

    Finally \eqref{eq:F_TAP_max_tau_LR} follows from \eqref{eq:F_TAP_max_tau_L} and \eqref{eq:F_TAP_max_tau_R}.
\end{proof}

\begin{lemma}
    If $\beta>\beta_c$, then
    \begin{equation}
        G_L(\beta,1-q^\ast,0)
        =
        \beta (1-q^\ast) \lM
        -
        \frac{1}{2}
        \int_{\lm}^{\lM}
        \log(\lM-y)
        \mu(dy).
    \end{equation}
\end{lemma}

\begin{proof}
    If $\beta>\beta_c$, then $q^\ast=q^\ast(\beta)>0$, and the function
    \begin{equation}\label{eq:eta_to}
        \eta
        \mapsto
        \beta(1-q^\ast)\eta
        -
        \dfrac{1}{2}
        \int_{\lm}^{\lM}
        \log(\eta-y)\mu(dy)
    \end{equation}
    has derivative
    \begin{equation}\label{eq:eta_to_deriv}
        \beta (1-q^\ast)
        -
        \dfrac{1}{2}
        \int_{\lm}^{\lM} \dfrac{\mu(dy)}{\eta-y}
        \ges
        \beta(1-q^\ast)-\dfrac{1}{2}\int_{\lm}^{\lM} \dfrac{\mu(dy)}{\lM-y}
        \overset{\eqref{eq:def_smu}}{=}
        \beta(1-q^\ast)
        -
        \dfrac{1}{2} s_{\mu}.
    \end{equation}
    By the definitions \eqref{eq:def_beta_c}, \eqref{eq:def_q_star}, the right-hand side of \eqref{eq:eta_to_deriv} is zero. Thus \eqref{eq:eta_to} is non-decreasing.    
\end{proof}

\begin{lemma}\label{lem:GL_plus_GR_alpha_ineq}
    For all $\beta>0$ and $\alpha\in(0,1)$
    \begin{equation}\label{eq:GL_plus_GR_alpha_ineq}
        G_L(\beta,1-q^\ast,\alpha)
        +
        G_R(\beta,q^\ast,\alpha)
        \le
        G_L(\beta,1-q^\ast,0) + \beta q^\ast \lM.
    \end{equation}
\end{lemma}

\begin{proof}    
    Since the sum of infima is not greater than the infimum of the sum
    \begin{equation}\label{eq:lhs_bound}
        \begin{aligned}
            & G_L(\beta,1-q^\ast,\alpha)
            +
            G_R(\beta,q^\ast,\alpha)
            \\
            &\les
            \inf_{\eta\ges\lM} \left\{\beta(1-q^\ast)\eta-\frac{1}{2}\int^{I(1-\alpha)}_{\lm}\log(\eta-x)\mu(dx)\right\}+
            \inf_{\eta\ges\lM}\left\{\beta q^\ast\eta-\frac{1}{2}\int^{\lM}_{I(1-\alpha)}\log(\eta-x)\mu(dx)\right\}
            \\
            &\les\inf_{\eta\ges\lM} \left\{\beta\eta-\dfrac{1}{2}\int_{\lm}^{\lM} \log(\eta-x)\mu(dx)\right\}
        \end{aligned}
    \end{equation}
    If $q^\ast=0$, then the bottom line equals  the r.h.s. of \eqref{eq:GL_plus_GR_alpha_ineq}. If $q^\ast>0$, then by considering the case $\eta = \lM$ we see that the bottom line is bounded above by
    \begin{equation}
        \beta \lM
        -
        \dfrac{1}{2}\int_{\lm}^{\lM}
        \log(\lM-x)\mu(dx).
    \end{equation}
    which by \eqref{eq:eta_to} equals the r.h.s. of \eqref{eq:GL_plus_GR_alpha_ineq}. Thus in either case \eqref{eq:GL_plus_GR_alpha_ineq} holds.    
\end{proof}

\begin{lemma}\label{lem:f_TAP_q_star_sup_in_alpha}
    For any $\beta>0$
    \begin{equation}\label{eq:f_TAP_q_star_sup_in_alpha}
        \sup_{\alpha\in(0,1)} \fTAPalpha(\beta,1-q^*,q^*,\alpha)
        \les
        \fTAPalpha(\beta,1-q^*,0,0).
    \end{equation}
\end{lemma}

\begin{proof}
    This follows from Lemma \ref{lem:f_TAP_q_star_alpha=0} and \ref{lem:GL_plus_GR_alpha_ineq}.
\end{proof}

\begin{corollary}\label{cor:f_TAP_alpha_ineq}
    For all $\beta>0$
    \begin{equation}\label{eq:f_TAP_alpha_ineq}
        \sup_{\alpha, \tau_L,\tau_R\in \mathfrak{S}_{\beta}}
        \fTAPalpha (\beta,\tau_L,\tau_R,\alpha)
        =
        \fTAPalpha(\beta,1-q^\ast,0,0).
    \end{equation}
\end{corollary}

\begin{proof}
    The upper bound follows from statement (3) of Lemma \ref{lem:F_TAP_max_tau} and Lemma \ref{lem:f_TAP_q_star_sup_in_alpha}, while the lower bound simply follows from evaluating $\fTAPalpha$ at $(\beta,1-q^*,0,0)$.
\end{proof}

\begin{lemma}\label{lem:MTAP_fTAP}
    For all $\beta>0$
    \begin{equation}\label{eq:MTAP_fTAP}
        \fTAPalpha(\beta,1-q^\ast,0,0) = f_\TAP(\beta,\lM q^\ast,q^\ast)
        =
        \sup_{
        q \in [0,1)
        :
        \;
        \beta(1-q)\les\beta_c
        }
        f_\TAP(\beta,q \lM,q).
    \end{equation}
\end{lemma}

\begin{proof}
    The first equality follows from \eqref{eq:def_ftap_alpha=0} and \eqref{eq:ftapDefinition} while the second follows from

    \begin{equation}
        \frac{\partial  f_\TAP(\beta,q \lM,q)}{\partial q}
        =\beta\big(\lM-x^*(\beta(1-q))\big)\leq 0,
    \end{equation}
    where we used \eqref{eq:ftapDefinition}, Lemma \ref{lem:legendre_diff} and $x^*\geq \lM$.
\end{proof}

\begin{proof}[Proof of \Cref{thm:main}]
    Fix some $B>0$. The upper bound
    
    \begin{equation}
        F_N(\beta)
        \leq
        \sup_{
        q \in [0,1)
        :
        \;
        \beta(1-q)\les\beta_c
        }
        f_\TAP(\beta,q \lM,q)
        +o(1),
    \end{equation}
    which holds uniformly for any $\beta\in[0,B]$ follows by combining \Cref{prop:aux}, Corollary \ref{cor:f_TAP_alpha_ineq} and Lemma \ref{lem:MTAP_fTAP} for $\beta>\beta_c$. For $\beta\les \beta_c$, note that $q=0$ is a valid choice in the above supremum and that $f_\TAP(\beta,0,0)=\Flim{\mu}(\beta)$, using Theorem \ref{thm:sec2_main} concludes the proof of the upper bound. For the lower bound, it also follows from Theorem \ref{thm:sec2_main} for $\beta\les \beta_c$. For $\beta>\beta_c$ simply note that
    \begin{equation}
        F_N(\beta)
        \geq 
        \frac{1}{N}
        \log
        Z_N (\beta,A_{0,\e}^{1-q^*,0}).
    \end{equation}
    Moreover, Lemma \ref{lem:Z_Nslice_decompo} proves

    \begin{equation}
    \begin{aligned}
        \frac{1}{N}
        \log Z_N (\beta,A_{0,\e}^{1-q^*,0})
        \ges &
        \beta \lambda_N q^*
        +
        \frac{N-1}{N}
        F_{\hat{\mu}_{L,0}}\left(\frac{N}{N-1} \beta (1-q^*)\right) \\
        &+\frac{1}{N}\log \Es \big[ A_{0,\e}^{1-q^*,0} \big]
        -
        6\beta \Lambda\epsilon.
    \end{aligned}
    \end{equation}
    A straightforward adaptation of Lemma \ref{lem:measure_AK_K=0_N} proves that 

    \begin{equation}
        \frac{1}{N}\log \Es \left[ A_{0,\e}^{1-q^*,0} \right]
        \ges
        - 
        \frac{1}{2}D(0||1-q^*,0)
        -
        c\frac{\log N \epsilon^{-1}}{N}
        -
        \frac{\epsilon}{1-q^*},
    \end{equation}
    while using Lipschitzness and boundedness of $F_{\hat{\mu}_{L,0}}$ yields

    \begin{equation}
    \begin{aligned}
        \frac{1}{N}
        \log Z_N (\beta,A_{0,\e}^{1-q^*,0})
        \ges &
        \beta \lambda_N q^*
        +
        F_{\hat{\mu}_{L,0}}\left( \beta (1-q^*)\right)
        - 
        \frac{1}{2}D(0||1-q^*,0)\\
        &-
        c\frac{\log N \epsilon^{-1}}{N}
        -
        \frac{\epsilon}{1-q^*}
        -2\frac{\Lambda \beta(1-q^*)}{N}
        -6\beta\Lambda\epsilon.
    \end{aligned}
    \end{equation}
    Choosing $\epsilon=N^{-1}$ and lower bounding $F_{\hat{\mu}_{L,0}}$ with Theorem \ref{thm:sec2_main} concludes the proof of the Theorem for pointwise convergence (the sequence $\hat{\mu}_{L,0}$ satisfies Assumption \ref{main_assumption}). Because both $F_N$ and $F_\mu$ are $\Lambda$-Lipschitz in $\beta$, a finite-net argument upgrades pointwise convergence to uniform convergence on every compact temperature interval.
\end{proof}

\section{The TAP formula with external fields}\label{sec:external_field_elimination}

To prove Theorem~\ref{thm:main_with_ext_field} we must extend our results to the presence of external field. The proof has two main components, both relying on a partition of the sphere $S^{N-1}$. The high level idea of the proof is to first introduce a partition of the sphere in sets for which the external field and so called \textit{effective external field} are near constants. This is achieved using the construction introduced in~\cite{belius2022high}, partitioning the sphere into a collection of \textit{slices}, which are the intersection between $S^{N-1}$ and carefully chosen affine subspaces. Then, we must compute the free energy on each of these slices, which is done using the partition introduced in Section \ref{sec:partition}. This partition separates the free energy into an effective high-temperature and effective low-temperature contribution. We start by presenting a partition of the sphere into a collection of slices where the effective external field effect vanishes. Then, we estimate the free energy if restricted to a slice where the external field effect vanishes using our previous results. Putting the two together concludes the proof of Theorem~\ref{thm:main_with_ext_field}.

\subsection{A covering by affine slices}\label{subsec:slice_partition} 

For the sake of clarity, we write the proof for an external field of the form $N\gamma \langle u,\sigma \rangle $ for some $u\in S^{N-1}$ and $\gamma\geq 0$. The general setting of Theorem \ref{thm:main_with_ext_field} follows from the exact same proof.

\vspace{0.5cm}

The goal of this section is to introduce a partition of the sphere $S^{N-1}$ in sets where the external field is constant. The goal being to make the external field effect vanish, it is natural to start with an equator $\mathcal{E}:=\{ \sigma\in B_N, \ \langle\sigma.u\rangle = 0 \}$ and its thickened version $\bar{\mathcal{E}}:=\{ \sigma\in B_N, \ | \langle\sigma.u\rangle| \les \e \}$. On this equator, the external field effect is small, that is

\begin{equation}
    \left| H_N^f(\sigma)-H_N(\sigma)\right| \les N\gamma \e.
\end{equation}
To cover the whole sphere, one would expect to repeat this argument directly on the collection of thickened affine subspaces 

\begin{equation}\label{eq:def_calE_alpha1}
    \overline{\mathcal{E}_{a_1}}:=\{ \sigma, \ a_1-\e \les \langle \sigma.u\rangle \les a_1+\e \},
\end{equation}
where the external field is near constant. However, even if the external field is constant on each of the $\overline{\mathcal{E}_{a_1}}$, an \textit{effective external field} appears. Indeed, for $\sigma\in\mathcal{E}_{a_1}$ we write $\sigma=a_1 u +\widehat{\sigma}$, with $\langle u, \widehat{\sigma}_1\rangle=0$. Then 

\begin{equation}\label{eq:hamiltonian_center_decompo}
    H^f_N(\sigma)= H_N(a_1 u)+H_N(\hat{\sigma})+2a_1 N u^T M_N \hat{\sigma}+N\gamma a_1. 
\end{equation}
An effective external field of the form $2a_1 N\langle M_N u, \hat{\sigma} \rangle $ appears. To use the same argument as on the equator $\mathcal{E}$, we must restrict ourselves to an equator of $\mathcal{E}_{a_1}$, that is $\mathcal{E}_{a_1}'=\{ \sigma \in \e_{a_1}, |\langle M_Nu,\sigma  \rangle| =0 \}$, and more precisely to its thickened version $\overline{\mathcal{E}_{a_1}'}$. There, we have

\begin{equation}
    \left| 
    H_N^f(\sigma)-H_N(\sigma)-N\gamma a_1
    \right|
    \les
    N\gamma \e
\end{equation}
The goal being to cover the whole sphere, we are falling short because, in general, the union of the equators $\mathcal{E}_{a_1}'$ does not cover $S^{N-1}$. The natural following step is to decompose each of the $\mathcal{E}_{a_1}$ by the collection of affine subspaces

\begin{equation}\label{eq:def_calE_alpha12}
    \mathcal{E}_{a_1,a_2}:=\{ \sigma\in \mathcal{E}_{a}, \  \langle |M_n u|^{-1}M_N  u,\sigma\rangle = a_2 \}
\end{equation}
and their thickened versions $\overline{\mathcal{E}_{a_1,a_2}}$. There, again, an effective external field is appearing, and hence a restriction to a carefully chosen equator is needed. The idea is to repeat this construction until the whole sphere has been covered in sets where the effective external field is near constant. The detailed sequential definition of this partition is exposed in details in ~\cite[Section 2]{belius2022high} and we skip it for brevity. After $k$ steps of the construction, we obtain $k+1$ orthogonal directions $u,u_2,\dots,u_{k+1}$ and for $a=(a_1,a_2,\dots,a_k)\in [-1,1]^k$, the respective set $\mathcal{E}_{a}$ is given by

\begin{equation}\label{eq:def_calE_alpha}
    \mathcal{E}_{a}=\left\{ \sigma \in \mathcal{E}_{a_1,\dots,a_{k}}, \ \langle \sigma , u_k \rangle =a_k \ \text{for} \ k\in [k] \right\},
\end{equation}
its equator by

\begin{equation}\label{eq:def_calE_alphaprime}
    \mathcal{E}_{a}'=\left\{ \sigma \in \mathcal{E}_{a}, \ \langle \sigma , u_{k+1} \rangle =0 \right\},
\end{equation}
and its center $m_{a}$ is given by

\begin{equation}\label{eq:def_m_alpha}
    m_{a}=m_{(a_1,\dots,a_{k-1})}+a_k u_k.
\end{equation}
The key to this proof is that, by sequentially repeating the above construction, we end up covering the whole sphere by thickened equators of the form $\overline{\mathcal{E}_{a}'}$ in finitely many steps, as proven in ~\cite[Proposition 5.3]{belius2022high}. If we denote by $N_{\e}$ the number of sets associated, there exist a finite $K_{\epsilon}$ and a set $\mathcal{A}\subset \cup_{k\leq K_{\e}}[-1,1]^{k}$ with $|\mathcal{A}|=N_{\e}$ such that 

\begin{equation}\label{eq:sphere_covering_by_slices}
    S^{N-1} \subset \bigcup_{a\in \mathcal{A}} \overline{\mathcal{E}_{a}'}.
\end{equation}
Remark that in this construction all the sets $\mathcal{E}_{a}'$ have high dimension, in particular $\mathrm{Dim}\big(\mathcal{E}_{a}'^\perp \big)\leq K_\epsilon +1$. Because, by construction, on each of the set $\overline{\mathcal{E}_{a}'}$ we have $|H_N^f(\sigma)-H_N(\sigma)-N\gamma\langle m_{a},u\rangle|\leq N\gamma\e$, the partition \eqref{eq:sphere_covering_by_slices} leads to

\begin{equation}\label{eq:partition_vanishing_extfield_main}
    F_N^f(\beta)
    \les
    \sup_{a\in \mathcal{A}} F_N^f\big(\beta,\overline{\mathcal{E}_{a}'}\big)
    +
    \frac{1}{N}\log(N_{\e})
    \les
    \sup_{a\in \mathcal{A}} F_N\big(\beta,\overline{\mathcal{E}_{a}'}\big)+ \gamma\beta\langle m_{a},u\rangle
    +
    \beta \gamma \e
    +
    \frac{1}{N}\log(N_{\e}).
\end{equation}
The natural following step is to estimate, for each $a\in \mathcal{A}$, the free energy $F_N\big(\beta,\overline{\mathcal{E}_{a}'}\big)$.

\subsection{Free-energy estimates on affine slices}\label{subsec:contribution_slice}

In this section, we present technical lemmas giving a more tractable expression for the free energy on slices. Recall that the expectation denoted by $E$ is for $\sigma$ uniformly distributed on the sphere $S^{N-1}$ and, for $\mathcal{E}$ an affine subspace, $E_{\mathcal{E}}$ denotes the expectation for $\sigma$ uniform on $S^{N-1}\cap \mathcal{E}$. We also recall that for any set $\mathcal{S}$ we let

\begin{equation}\label{eq:def_FN_S}
    F_N(\beta,\mathcal{S})
    =
    \frac{1}{N}\log E\left[ \exp\left(\beta H_N(\sigma)\right)\indf_{\mathcal{S}} \right]
\end{equation}
be the free energy restricted to set $S$, and $F_N^f(\beta,\mathcal{S})$ the corresponding quantity when an external field is present. We will need to control $F_N(\beta,\mathcal{S})$ when $S$ is a \textit{thickened} affine subspace. The following two lemmas simply prove that we can forget the effect of thickening the affine subspace.

\begin{lemma}\label{lem:TAP_on_subspace}
    Let $\mathcal{E}$ be any affine subspace of  $\ \R^N$ with $d(0,\mathcal{E})=\alpha$, $\alpha<1$ and fix some $0<\e<1-\alpha$. Let $\overline{\mathcal{E}}=\{x\in \R^n, \ d(x,\mathcal{E})\leq \e\}$ be its thickened version. Then

    \begin{equation}
        \left|
        F^f_N (\beta,\overline{\mathcal{E}})
        -
        \frac{1}{N}\log \left(E\left[ \indf_{\overline{\mathcal{E}}} \right]E_{\mathcal{E}}\left[ \exp\left(\beta H^f_N(\sigma)\right) \right]\right)
        \right|
        \les
        \frac{2\beta(\Lambda+\gamma)\e}{\sqrt{1-(\alpha+\e)^2}}.
    \end{equation}
\end{lemma}

\begin{proof}
    Let us have $P$ the projection of $S^{N-1}\cap\overline{\mathcal{E}}$ onto $S^{N-1}\cap \mathcal{E}$. Because on $S_n$ the normalized Hamiltonian $H_N/N$ is $2(\Lambda+\gamma)$ Lipchitz and that for $\sigma \in S^{N-1}\cap\overline{\mathcal{E}}$ we have $|P(\sigma)-\sigma|\les \e/\sqrt{1-(\alpha+\e)^2}$

    \begin{equation}
        \left|
        \frac{1}{N}\log E\left[ \exp\left(\beta H^f_N(\sigma)\right)\indf_{\overline{\mathcal{E}}} \right]
        -
        \frac{1}{N}\log E\left[ \exp\left(\beta H^f_N(P(\sigma))\right)\indf_{\overline{\mathcal{E}}} \right] \right|
        \les
         \frac{2\beta(\Lambda+\gamma)\e}{\sqrt{1-(\alpha+\e)^2}}.
    \end{equation}
    Finally, remark that for $\sigma$ uniformly distributed on $S^{N-1} \cap \overline{\mathcal{E}}$, $P(\sigma)$ is uniformly distributed on $ S^{N-1} \cap \mathcal{E}$, concluding the proof of the lemma.  
\end{proof}

To use the previous lemma, one also needs a control on $\tfrac{1}{N}\log E\left[ \indf_{\overline{\mathcal{E}}} \right]$. Such a control can be found in ~\cite{belius2022high} but we state it here again adapted to our purpose.

\begin{lemma}\label{lem:measure_small_sphere}
Let $N\ge2$, $\mathcal{E}\subset\mathbb R^N$ be an affine subspace
with $\alpha=d(0,\mathcal{E})<1$, and let $\overline{\mathcal{E}}=\{x\in\mathbb R^N:d(x,\mathcal{E})\le\epsilon\}$. Then
\begin{equation}
    \frac{1}{N}\log \left(E\left[ \indf_{\overline{\mathcal{E}}} \right]\right)
    \les
    \frac12\log(1-\alpha^2)
    +
    \frac{\alpha\epsilon}{1-\alpha^2}
    -
    \frac{\log(1-\alpha^2)}{2N}.
\end{equation}
\end{lemma}

\begin{proof}
If $\alpha=0$ the result is trivial, so suppose $\alpha>0$. Let $m$ be the smallest norm point of $\mathcal{E}$ and set $e=m/|m|$. Since $m$ is orthogonal to $\mathcal{E}-m$, every $x\in \mathcal{E}$ satisfies $\langle x,e\rangle=\alpha$. Thus
\begin{equation}
    \overline{\mathcal{E}}\cap S^{N-1}
    \subset
    \{\sigma\in S^{N-1}; \ \langle\sigma,e\rangle\ge\alpha-\epsilon\}.
\end{equation}
By rotational invariance, write $\sigma=G/|G|$,
where $G=(G_1,\ldots,G_N)$ are independent standard Gaussian
variables, and take $e=e_1$. Setting
$S=\sum_{i=2}^N G_i^2$, the Gaussian tail bound gives
\begin{equation}
\begin{aligned}
\mathbb P(\langle\sigma,e\rangle\ge t)
&=
\mathbb P\left(
G_1\ge\frac{t}{\sqrt{1-t^2}}\sqrt S
\right)\\
&\le
\mathbb E\exp\left(-\frac{t^2S}{2(1-t^2)}\right)
=(1-t^2)^{(N-1)/2}.
\end{aligned}
\end{equation}
If $\epsilon<\alpha$, applying the above at $t=\alpha-\epsilon\in(0,1)$ proves 
\begin{equation}\label{eq:volume_slice_first_claim}
    \frac1N\log \left(E\left[ \indf_{\overline{\mathcal{E}}} \right]\right)
    \les
    \frac{N-1}{2N}
    \log\bigl(1-(\alpha-\epsilon)_+^2\bigr),
\end{equation}
and the above is also trivially true if $\epsilon\geq\alpha$. Then, set $s=1-\alpha^2$ and
$b=1-(\alpha-\epsilon)_+^2$. Then $s\le b\le s+2\alpha\epsilon$. Hence
\begin{equation}
    \begin{aligned}
    \frac{N-1}{2N}\log b
    &=
    \frac12\log s+\frac12\log(b/s)-\frac1{2N}\log b\\
    &\le
    \frac12\log s+\frac{\alpha\epsilon}{s}
    +\frac{|\log s|}{2N},
\end{aligned}
\end{equation}
which concludes the proof of the lemma.
\end{proof}

\begin{lemma}\label{lem:measure_sphere_LB}
    Let $N\geq k+2$, $\mathcal{E}\subset \R^N$ be an affine subspace of dimension $N-k$ with $\alpha=d(0,\mathcal{E})<1$ and, for $0<\epsilon < (1-\alpha^2)/6$, let $\overline{\mathcal{E}}=\{x\in\R^N \ : \ d(x,\mathcal{E})\les \epsilon\}$. Then

    \begin{equation}
    \frac{1}{N}\log \left(E\left[ \indf_{\overline{\mathcal{E}}} \right]\right)
    \ges
    \frac12\log(1-\alpha^2)
    -
    \frac{3\epsilon}{1-\alpha^2}
    +
    \frac{k}{N}\log \epsilon.
    \end{equation}
\end{lemma}

\begin{proof}
The statement is trivial if $k=0$, so suppose $k>0$. Letting $m$ be the smallest norm element of $\mathcal{E}$ we write $\mathcal{E}=m+V$, where $V$ is a linear subspace of dimension $N-k$ and $m\perp V$. Assume $1\le k\le N-2$ and $0<\epsilon\le (1-|m|^2)/6$. Since $d(\sigma,\mathcal E)=|P^{V^\perp}\sigma-m|$ we want to lower bound the probability of the ball $B=\{ y\in V^\perp; \ d(m,y)\leq \epsilon \}$ when $y=P^{V^\perp}\sigma$. Let $\ell_k$ denote Lebesgue measure on $V^\perp$, normalized so that $\ell_k(B(0,1))=1$. The beta-distribution identity
in \cite[Section~2, p.~403]{diaconis1987dozen} gives that $y$ has density, with respect to $\ell_k$,

\begin{equation}
    p(y)
    =
    \frac{\Gamma(N/2)}
    {\Gamma((N-k)/2)\Gamma(k/2+1)}
    (1-|y|^2)^{(N-k-2)/2}\mathbf1_{\{|y|<1\}}.
\end{equation}
Moreover, $1-|y|^2\ge 1-|m|^2-3\epsilon$ whenever $y\in B$ and this ball has volume $\epsilon^k$ with respect to $\ell_k$. Hence

\begin{equation}
    E[\mathbf1_{\overline{\mathcal E}}]
    \ge
    \frac{\Gamma(N/2)}
    {\Gamma((N-k)/2)\Gamma(k/2+1)}
    \epsilon^k\big(1-|m|^2-3\epsilon\big)^{(N-k-2)/2}
    \ges
    \epsilon^k \big(1-|m|^2-3\epsilon\big)^{N/2}.
\end{equation}
Therefore

\begin{equation}
\begin{aligned}
\frac1N\log E[\mathbf1_{\overline{\mathcal E}}]
&\ge
\frac12\log(1-|m|^2-3\epsilon)
+
\frac{k}{N}\log \epsilon\\
&\ge
\frac12\log(1-|m|^2)
-\frac{3\epsilon}{1-|m|^2}
+
\frac{k}{N}\log \epsilon.
\end{aligned}
\end{equation}
Replacing $|m|$ by $\alpha$ concludes the proof of the lemma.
\end{proof}

The next, and main, piece of the puzzle is to compute the free energy restricted to an affine subspace, that is $\tfrac{1}{N}\log E_{\mathcal{E}}\left[ \exp \beta H_N(\sigma) \right]$. We do not need to consider any such affine subspace but those where the effective external field vanishes. Generalizing \eqref{eq:hamiltonian_center_decompo} we write, for $\mathcal{E}$ an affine subspace intersecting $S^{N-1}$, $m$ its smallest norm element and $\sigma\in \mathcal{E}$ that $\sigma=m+\sigma'$. Then

\begin{equation}
    H_N(\sigma)= H_N(m)+H_N(\sigma')+2N m^T M_N \sigma'. 
\end{equation}
By the construction of Section \ref{subsec:slice_partition} we only consider affine subspaces where $m^T M_N \sigma'=0$, which we refer to as \textit{carefully chosen affine subspaces} hereafter. We introduce $H_N^m(\sigma)=H_N(\sigma-m)$ the recentered Hamiltonian and on carefully chosen affine subspaces

\begin{equation}\label{eq:hamiltonian_decompo}
    H_N(\sigma)= H_N(m)+H_N^m(\sigma). 
\end{equation}
To compute the free energy on affine subspaces where the external field vanishes, we partition them as in Section \ref{sec:partition}. To this effect, write $\mathcal{E}=m+V$, where $m\perp V$, and $n=\mathrm{Dim}(V)$. We assume $|m|<1$. Let $\lambda_1^{\mathcal{E}}\le\cdots\le\lambda_n^\mathcal{E}$ be the eigenvalues of the compression of $M_N$ to $V$, with corresponding orthonormal eigenvectors $e_1^{\mathcal{E}},\ldots,e_n^{\mathcal{E}}$. For $\sigma\in S^{N-1}\cap \mathcal{E}$, define
\begin{equation}
    \hat{\sigma}=\frac{\sigma-m}{\sqrt{1-|m|^2}},
    \qquad
    \hat{\sigma}_i= \langle\hat{\sigma},e_i^{\mathcal{E}}\rangle.
\end{equation}
Then $\hat{\sigma}$ belongs to the unit sphere of $V$. We also write $q^{*,\mathcal{E}}$ the critical $q$ at temperature $\beta(1-|m|^2)$, as defined in \eqref{eq:def_q_star}, with $q^*(0)=0$. Then, one can write

\begin{equation}
    S^{N-1}\cap \mathcal{E} = \bigcup_{K=0}^{n-1} A_K^{\mathcal{E}},
\end{equation}
where $A_K^{\mathcal{E}}$ is analogous to $A_K$ and is given by
\begin{equation}\label{eq:def_AK_calE}
    A_K^{\mathcal{E}}
    =
    \{\sigma\in S^{N-1}\cap\mathcal{E}\mid K^{\mathcal{E}}(\sigma)=K\},\quad K=0,1,\ldots,n-1,
\end{equation}
for
\begin{equation}
    K^{\mathcal{E}}(\sigma)=\max\bigg\{k\in\{0,1,\ldots,n-1\}\mid \sum_{i=n-k+1}^{n} \hat{\sigma}_i^2\les q^{*,\mathcal{E}}\bigg\}.
\end{equation}
Just as in Section \ref{sec:partition}, we introduce, for $\sigma\in S^{N-1}\cap\mathcal{E}$

\begin{equation}\label{eq:def_tau_q_calE}
    \tau_{L,K}^{\mathcal{E}}(\sigma)
    =
    \sum_{i=1}^{n-K-1} \hat{\sigma}_i^2,\qquad
    \tau_{R,K}^{\mathcal{E}}(\sigma)
    =
    \sum_{i=n-K+1}^{n} \hat{\sigma}_i^2,
\end{equation}
so that
\begin{equation}
    \tau_{L,K}^{\mathcal{E}}(\sigma)+\hat{\sigma}_{n-K}^2+\tau_{R,K}^{\mathcal{E}}(\sigma)=1,
\end{equation}
Then, for any $K\in[0,n-1]$
\begin{equation}
    A_K^{\mathcal{E}}
    =
    \left\{\sigma\in S^{N-1}\cap \mathcal{E}
    \mid
    \tau_{L,K}^{\mathcal{E}}(\sigma)<1-q^{*,\mathcal{E}}, \;
    \tau_{R,K}^{\mathcal{E}}(\sigma)\les q^{*,\mathcal{E}}\right\}.
\end{equation}
It is then possible to further partition each of the $A_K^{\mathcal{E}}$ by approximate value of $\tau_{L,K}$ and $\tau_{R,K}$, as done in Section \ref{sec:partition}. In particular, inspired by \eqref{eq:defAkqlqrdelta} we define

\begin{equation}\label{eq:def_A_tau_q_calE}
    A_{K,\tau_L,\tau_R,\delta}^{\mathcal{E}}
    =
    \left\{
    \sigma\in A_K^{\mathcal{E}}\mid \tau_L-\delta \les\tau_{L,K}^{\mathcal{E}}(\sigma)\les\tau_L+\delta
    ,\; 
    \tau_R-\delta\les\tau_{R,K}^{\mathcal{E}}(\sigma)\les\tau_R+\delta
    \right\}
\end{equation}
In the following lemma we check that the free energy, restricted to each of the 
$A_{K,\tau_L,\tau_R,\delta}^{\mathcal{E}}$, is approximated by the TAP free energy

\begin{equation}\label{eq:final_def_Ftap}
F_{\TAP}\big(\beta,m,\tau_L,\tau_R,\alpha\big)
:=
\frac{1}{N}\beta H_N(m)
+
\frac{1}{2}\log\left( 1-|m|^2 \right)
+
\On\big(\beta(1-|m|^2),\tau_L,\tau_R,\alpha\big),
\end{equation}
where the Onsager term is $\On\big(\beta,\tau_L,\tau_R,\alpha\big)=\fTAPalpha(\beta,\tau_L,\tau_R,\alpha)$ and is defined in \eqref{eq:def_ftap}. We decide to use an extra notation rather than reuse $\fTAPalpha$ to illustrate that in the above, the TAP free energy plays the role of the Onsager term. In the following lemma, the free energy for $\sigma$ conditioned to lie in $\mathcal{E}$ is bounded. Because doing so the probability of the slice is normalized to be $1$, an extra term $-\tfrac{1}{2}\log\left( 1-|m|^2 \right)$ appears (canceling the one present in the definition of $F_{\TAP}$).

\begin{lemma}\label{lem:sec7.1_main}
    Let $\mathcal{E}$ be any affine subspace of dimension $n$ intersecting $S^{N-1}$, of uniformly bounded codimension and chosen so that the effective external field on $\mathcal{E}$ is null. Let $m$ be its smallest norm element. Then

    \begin{equation}
    \begin{aligned}
        \frac{1}{N}\log E_{\mathcal{E}} \left[ \exp\left(\beta H_N(\sigma)\right)  \right]
        \les \ &
        \sup_{\alpha,\tau_L,\tau_R\in S_{\mathcal{E}}}F_{\TAP}\big(\beta,m,\tau_L,\tau_R,\alpha\big)\\
        &-
        \frac{1}{2}\log\left( 1-|m|^2 \right)
        +
        2\frac{N-n}{n}\beta\Lambda
        +
        \mathrm{Err}_n \big(\beta,\mu_N^{\mathcal{E}}\big),
    \end{aligned}
    \end{equation}
    where $\mu_N^{\mathcal{E}}$ is the empirical distribution stemming from the eigenvalues of $M_N$ restricted to the linear subspace $\mathcal{E}-m$, $\mathrm{Err}_N$ is as in Proposition \ref{prop:aux} and the supremum is taken over the set

    \begin{equation}\label{eq:def_sup_set}
        S_{\mathcal{E}}
        :=
        [0,1] \times [0,1-q^{*,\mathcal{E}}]
        \times
        [0,q^{*,\mathcal{E}}].
    \end{equation}
    

\end{lemma}

\begin{proof}
    Because $S^{N-1} \cap \mathcal{E}$ is also a sphere, but of smaller dimension and radius, the natural first step is to remark that for $\sigma$ uniform on $S^{N-1}\cap \mathcal{E}$, $\sigma-m$ is uniform on a sphere of radius $\sqrt{1-|m|^2}$. Hence using the
    decomposition of \eqref{eq:hamiltonian_decompo}, valid on affine subspace where effective external field vanishes, and the homogeneity of the Hamiltonian
    
    \begin{equation}
        \frac{1}{N}\log E_{\mathcal{E}}\left[ \exp \beta H_N(\sigma) \right]
        =
        \frac{1}{N}\beta H_N(m)
        +
        \frac{1}{N}\log E_{V}\left[ \exp \beta (1-|m|^2) H_N(\sigma) \right].
    \end{equation}
    We start by taking into account that the dimension of $V=\mathcal{E}-m$ is $n$ rather than $N$. We get

    \begin{equation}
    \begin{aligned}
        \frac{1}{N}\log E_{V}\left[ \exp \beta (1-|m|^2) H_N(\sigma) \right]
        = &
        \frac{n}{N}\frac{1}{n}\log E_{V}\left[ \exp \tfrac{N}{n}\beta (1-|m|^2) H_n(\sigma) \right]\\
        \les &
        \frac{1}{n}\log E_{V}\left[ \exp\beta (1-|m|^2) H_n(\sigma) \right]
        +
        2\frac{N-n}{n}\beta\Lambda,
    \end{aligned}
    \end{equation}
    where we used Lipschitzness and the abuse of notation $H_n(\sigma)=\frac{n}{N}H_N(\sigma)=n\sum_{i=1}^{\mathrm{Dim}(\mathcal{E})}\lambda_i^{\mathcal{E}}\sigma_i^2$. Then, the reasoning is as in Section \ref{sec:intermediate_TAP} but at temperature $\beta (1-|m|^2)$ and dimension $n$. We start by checking that $\mu_N^{\mathcal{E}}$ satisfies Assumption \ref{main_assumption}. Poincaré separation Theorem (also known as the Cauchy interlacing theorem), states that for $i\in [\mathrm{Dim}(\mathcal{E})]$ 
    
    \begin{equation}
        \lambda_i \les
        \lambda_i^{ \mathcal{E}}
        \les
        \lambda_{i+N-n}.
    \end{equation}
    As a direct consequence, if $\mathrm{Dim}(\mathcal{E}^\perp)$ is bounded then $\mu_N^{\mathcal{E}}$ converges to $\mu$ and satisfies all the assumptions of Assumption \ref{main_assumption}. Start by supposing $\beta(1-|m|^2)>\beta_c$ and fix some $\delta>0$. Recall the partition 

    \begin{equation}
        S^{N-1}\cap \mathcal{E}
        \subset
        \bigcup_{K,\tau_L,\tau_R\in \widehat{S}_{\mathcal{E}}}
        A_{K,\tau_L, \tau_R,\delta}^{\mathcal{E}}
    \end{equation}
    where $\widehat{S}_{\mathcal{E}}$ is a discretized version of $S_{\mathcal{E}}$

    \begin{equation}\label{eq:def_sup_set_discrete}
        \widehat{S}_{\mathcal{E}}
        :=
        [0,n-1] \times [0,2\delta,\dots,1-q^{*,\mathcal{E}}]
        \times
        [0,2\delta, \dots,q^{*,\mathcal{E}}].
    \end{equation}
    Proposition \ref{prop:aux} with an appropriate choice of $\delta>0$ gives

    \begin{equation}    
    \begin{aligned}
        \tfrac{1}{n}\log E_{V}\left[ \exp \left( \beta (1-|m|^2) H_n(\sigma)\right)  \right]
        \les \ &
        \sup_{\alpha,\tau_L,\tau_R\in S_{\mathcal{E}}}\On(\beta(1-|m|^2),\tau_L,\tau_R,\alpha)\\
        &+
        \mathrm{Err}_n \big(\beta,\mu_N^{\mathcal{E}}\big),
    \end{aligned}
    \end{equation}
    where, as noted earlier, the Onsager term is $\fTAPalpha(\beta,\tau_L,\tau_R,\alpha)$ and is defined in \eqref{eq:def_ftap}. On the other hand, if $\beta(1-|m|^2)\le\beta_c$, Theorem \ref{thm:sec2_main} gives the required bound since 
    
    \begin{equation}
        F_\mu\big(\beta(1-|m|^2)\big)=
        \On\big(\beta(1-|m|^2),1,0,0\big).
    \end{equation}
    and $(0,1,0)\in\mathcal S_{\mathcal{E}}$. Recalling the definition of $F_{\TAP}$ \eqref{eq:final_def_Ftap} concludes the proof of the Lemma.
\end{proof}

It is now possible to put together the results of Lemma \ref{lem:TAP_on_subspace}, \ref{lem:measure_small_sphere} and \ref{lem:sec7.1_main} to obtain the following corollary, giving an approximation of the free energy on the set $\overline{\mathcal{E}}$ as the supremum of the contribution of the free energy on the $A_{K,\tau_L,\tau_R,\delta}^{\mathcal{E}}$.

\begin{corollary}\label{cor:external_field_main}
    Let $\mathcal{E}$ be any affine subspace of dimension at least $N-k$ intersecting $S^{N-1}$, chosen so that the effective external field on $\mathcal{E}$ is null. Let $m$ be the center of $\mathcal{E}$ and $\overline{\mathcal{E}}=\left\{x\in \mathbb{R}^N, \ d(x,\mathcal{E})\les \e \right\}$ be its thickened version. Then, for any $0<\e<1-|m|$

    \begin{equation}
    \begin{aligned}
        F_N(\beta,\overline{\mathcal{E}})
        \les \ &
        \sup_{\alpha,\tau_L,\tau_R \in S_{\mathcal{E}}}F_{\rm TAP}\big(\beta,m,\tau_L,\tau_R,\alpha\big)
        +
        \mathrm{Err}_N^{(7)} \big(\beta,\mu_N^{\mathcal{E}},|m|,\e\big),
    \end{aligned}
    \end{equation}
    where the error term is given by

    \begin{equation}
    \begin{aligned}
        \mathrm{Err}_N^{(7)} \big(\beta,\mu_N^{\mathcal{E}},|m|,\e\big)
        = &
        \frac{2\beta(\Lambda+\gamma)\e}{\sqrt{1-(|m|+\e)^2}}
        +
        \frac{|m|\epsilon}{1-|m|^2}
        -
        \frac{\log(1-|m|^2)}{2N}\\
        &+
        2\frac{k}{N-k}\beta\Lambda
        +
        \mathrm{Err}_{N-k} \big(\beta,\mu_N^{\mathcal{E}}\big)
    \end{aligned}
    \end{equation}
\end{corollary}

\begin{proof}
    The proof is immediate combining the results of  Lemma \ref{lem:TAP_on_subspace}, \ref{lem:measure_small_sphere} and \ref{lem:sec7.1_main}.
\end{proof}
Before the above corollary can be plugged into the variational formula of \eqref{eq:partition_vanishing_extfield_main}, remark that the error term $\mathrm{Err}_N^{(7)}$ diverges as $|m|$ approaches $1$. To deal with this corner case, we use an ad-hoc bound to prove that, if $|m|$ is too large, the contribution to the total free energy is negligible.

\begin{lemma}\label{lem:extfield_exclude_small_alpha}
    For $N\ges 2$, $\e,\rho$ small enough that $\epsilon,\rho<\exp\big(-8\beta(\Lambda+\gamma)-6\big)$, for any affine subset with center $m$, if $|m|>1-\rho$, then

    \begin{equation}
        F_N^f(\beta,\overline{\mathcal{E}})\les F_N^f(\beta)-1.
    \end{equation}
\end{lemma}

\begin{proof}
    Recall that the Hamiltonian is bounded by $-N(\Lambda+\gamma) \les H_N^f\les N( \Lambda+\gamma)$, in particular $F_N^f(\beta)\ges -\beta(\Lambda+\gamma)$. In the meantime, Lemma \ref{lem:measure_small_sphere} and \eqref{eq:volume_slice_first_claim} give an upper bound on the measure of $\overline{\mathcal{E}}$, which leads to

    \begin{equation}
        F_N^f (\beta,\overline{\mathcal{E}})
        \les
        \beta (\Lambda+ \gamma)
        +
        \frac{1}{4}\log\left( 1-(|m|-\epsilon)_+^2 \right)
        \les
        \beta (\Lambda+ \gamma)
        +
        \frac{1}{4}\log(2\rho+2\epsilon).
    \end{equation}
    The proof of the lemma follows. 
\end{proof}

\subsection{The auxiliary TAP upper bound}

We are now ready to use the decomposition of the free energy of \eqref{eq:partition_vanishing_extfield_main} alongside Corollary \ref{cor:external_field_main} and Lemma \ref{lem:extfield_exclude_small_alpha}, which leads to

\begin{equation}
\begin{aligned}
    F_N^{f}(\beta)
    \les 
    \sup_{a\in \mathcal{A}, |m_{a}|\leq 1-\rho} & \Bigg\{  \sup_{\alpha,\tau_L,\tau_R \in S_{\mathcal{E}_{a}}}F_{\rm TAP}\big(\beta,m_{a} ,\tau_L,\tau_R,\alpha\big) 
     +
    \beta\gamma\langle m_{a},u\rangle\\
    & \qquad+
    \beta \gamma \e
    +
    \frac{1}{N}\log(N_{\e})
    +
    \mathrm{Err}_N^{(7)} \big(\beta,\mu_N^{\mathcal{E}'_a},|m|,\e\big)
    \Bigg\},
\end{aligned}
\end{equation}
for some $\rho$ small enough for Lemma \ref{lem:extfield_exclude_small_alpha} to hold and $0<\epsilon<\rho/2$. Remark that $F_{\rm TAP}\big(\beta,m_{a} ,\tau_L,\tau_R,\alpha\big)+ \beta\gamma\langle m_{a},u\rangle$ is simply $F^f_\TAP$, the TAP free energy in the presence of the external field. Also remark that letting

\begin{equation}
    \begin{aligned}
        \mathrm{Err}_N^{(8)} \big(\beta,\e,\mu_N\big)
        = &
        \frac{8\beta(\Lambda+\gamma)\e}{\rho}
        +
        \frac{\epsilon}{\rho^2}
        -
        \frac{\log \rho^2}{2N}
        +
        \beta \gamma \e
        +
        \frac{1}{N}\log(N_{\e})\\
        &+
        2\frac{K_{\epsilon}}{N-K_{\epsilon}}\beta\Lambda
        +
        \sup_{a\in \mathcal{A}}
        \mathrm{Err}_N \big(\beta,\mu_N^{\mathcal{E}_{a}'}\big),
    \end{aligned}
    \end{equation}
we have

\begin{equation}
    \sup_{a\in \mathcal{A}, |m_{a}|\leq 1-\rho}
    \beta \gamma \e
    +
    \frac{1}{N}\log(N_{\e})
    +
    \mathrm{Err}_N^{(7)} \big(\beta,\mu_N^{\mathcal{E}'_a},m,\e\big)
    \les 
    \mathrm{Err}_N^{(8)}\big( \beta,\e,\mu_N\big).
\end{equation}
The main bound of this section is now proved, that is,

\begin{equation}\label{eq:external_field_main_variational_UB}
    F_N^{f}(\beta)
    \les 
    \sup_{m\in B_N} \sup_{\alpha,\tau_L,\tau_R \in S_{m}}  F^f_\TAP\big(\beta,m,\tau_L,\tau_R,\alpha\big) 
    +
    \mathrm{Err}_N^{(8)}(\beta,\e,\mu_N).
\end{equation}

\subsection{Reduction to the Plefka region and proof of Theorem \ref{thm:main_with_ext_field}}

To conclude, it remains to check that if $m$ does not respect the Plefka condition it has no bearing in the supremum of \eqref{eq:external_field_main_variational_UB}. Here, an analysis of the TAP free energy specific to the 2-spin model is used. First, recall the definition of the TAP free energy of \eqref{eq:final_def_Ftap} and that there, the Onsager term is equal to $\fTAPalpha$. Then, recall that Section \ref{sec:where_is_TAP_max} provides an explicit formula for the supremum of $\fTAPalpha$. Hence,

\begin{equation}
    \sup_{\alpha,\tau_L,\tau_R \in S_{m}}
    \On\big(\beta(1-|m|^2),\tau_L,\tau_R,\alpha\big)
    =
    f_\TAP\big(\beta(1-|m|^2),q^{*,m}\lM,q^{*,m}\big),
\end{equation}
translating to


\begin{equation}\label{eq:external_field_main_variational_UB_bis}
\begin{aligned}
    F_N^{f}(\beta)
    \les &
    \sup_{m\in B_N} 
    \frac{1}{2}\log (1-|m|^2)
    +
    \beta N^{-1} H_N^f(m)  
    +
    f_\TAP\big(\beta(1-|m|^2),q^{*,m}\lM,q^{*,m}\big)\\
    &+
    \mathrm{Err}_N^{(8)}(\beta,\e,\mu_N).
\end{aligned}
\end{equation}
It remains to prove that the supremum is reached for some $m$ respecting Plefka's condition. When $m$ satisfies Plefka's condition, that is when $|m|^2\ges q^*$, we have $\beta(1-|m|^2)\les \beta_c$, and hence $q^{*,m}=0$. We directly recognize

\begin{equation}
    f_\TAP\big(\beta(1-|m|^2),q^{*,m}\lM,q^{*,m}\big)
    =
    f_\TAP\big(\beta(1-|m|^2),0,0\big)
    =
    F(\beta(1-|m|^2))
\end{equation}
Then, because $F(\beta(1-|m|^2))$ is the Onsager term $\On(\beta,|m|^2)$, see \ref{eq:OnsagerDefinition}, we recognize the expression for $F_{\TAP}(m)$ (see \eqref{eq:FTAP_first_def}) in \eqref{eq:external_field_main_variational_UB_bis}. Since $|m|^2\ges q^*$ is equivalent to $\beta(1-|m|^2)\les \beta_c$ we write

\begin{equation}\label{eq:Good_part_sup_external_field}
\begin{aligned}
    \sup_{m\in B_N: \ |m|^2\ges q^*} \left\{ \frac{1}{2}\log (1-|m|^2)
    +
    \beta N^{-1} H_N^f(m)  
    +
    F\big(\beta(1-|m|^2)\big) \right\}\\
    =
    \sup_{m\in B_N: \ \beta(1-|m|^2)\les \beta_c}  \left\{F_{\TAP}(m)\right\}.
\end{aligned}
\end{equation}
If $\beta\le\beta_c$, every $m\in B_N$ satisfies the
Plefka condition. For the remaining analysis, suppose $\beta>\beta_c$. When $m$ does not satisfy Plefka's condition, namely when $|m|^2\les q^*$, by the expression of $f_{\TAP}$ given in \eqref{eq:ftapDefinition} we get

\begin{equation}
\begin{aligned}
    f_{\TAP}\big(\beta(1-|m|^2),q^{*,m}\lM,q^{*,m}\big)
    = \ &
    \beta(1-|m|^2) q^{*,m} \lM
    -
    \frac{1}{2}\log\left(2\beta(1-|m|^2)\right)
    \\
    &-
    L^*_{\mu}\left(\beta(1-|m|^2)(1-q^{*,m})\right)
    -
    \frac{1}{2},
\end{aligned}
\end{equation}
where $q^{*,m}$ is the critical $q$ at temperature $\beta (1-|m|^2)$ and so $\beta(1-|m|^2)(1-q^{*,m})=\beta_c$. If we let $q=|m|^2$ and denote by

\begin{equation}\label{eq:def_energy_2spin}
    E_N(q)=\sup_{m\in B_N: \ |m|^2=q} N^{-1} H_N^f(m)
\end{equation}
we can bound

\begin{equation}\label{eq:supremum_non_plefka}
\begin{aligned}
    \sup_{m\in B_N: \ |m|^2\les q^*} & \left\{ \frac{1}{2}\log (1-|m|^2)
    +
    \beta N^{-1} H_N^f(m)  
    +
    f_{\TAP}\big(\beta(1-|m|^2),q^{*,m}\lM,q^{*,m}\big) \right\}\\
    \les &
    \sup_{ 0\les q \les q^* }
    \left\{
    \beta E_N(q)-\frac{1}{2}-\frac{1}{2}\log(2\beta)-L_{\mu}^*\left( \beta_c \right) +\beta(1-q)\lM -\beta_c\lM
    \right\}.
\end{aligned}
\end{equation}
The only dependence on $q$ comes from $E_N(q)-q\lM$. The following lemma proves that the supremum of this function for $q\in[0,q^*]$ is close to its value at $q^*$.

\begin{lemma}\label{lem:energy_at_norm_q_increases}
Suppose that Assumption \ref{main_assumption} holds. Let $U_N \subseteq \mathbb R^N$
be a subspace of dimension $d_N=O(1)$, and let
$f_N:\overline{B_N}\to\mathbb R$ be continuous and satisfy
\begin{equation}
    f_N(m)=f_N(P^{U_N}m)
    \qquad\text{for all }m\in\overline{B_N}.
\end{equation}
Define the normalized maximal energy at squared norm $q$ by
\begin{equation}
    E_N(q)
    :=\frac1N\max_{|m|^2=q}H_N^f(m)
    =\max_{|m|^2=q}
      \bigl\{m^\top M_Nm+f_N(m)\bigr\},
    \qquad q\in[0,1).
\end{equation}
For every $N\ge d_N+3$ and $0\le q\le q'<1$,
\begin{equation}
    E_N(q')
    \ge E_N(q)+(q'-q)\lambda_{N-d_N-2}.
\end{equation}
Consequently, for every $q^*\in[0,1)$,

\begin{equation}
    0
    \le
    \sup_{q\in[0,q^*]}
       \bigl\{E_N(q)-q\lM\bigr\}
       -\bigl\{E_N(q^*)-q^*\lM\bigr\}
    \le
    q^*\bigl|\lambda_{N-d_N-2}-\lM\bigr|
    =o(1).
\end{equation}
\end{lemma}

\begin{proof}
Fix $N\ge d_N+3$ and $0\le q<q'<1$. By continuity and
compactness, there exists $m\in\mathbb R^N$ such that

\begin{equation}
    |m|^2=q,
    \qquad
    E_N(q)=m^\top M_Nm+f_N(m).
\end{equation}
Let $V_N$ be the span of $d_N+3$ orthonormal eigenvectors
corresponding to the largest eigenvalues of $M_N$. Since

\begin{equation}
    \dim\bigl(U_N+\operatorname{span}\{m,M_Nm\}\bigr)
    \le d_N+2,
\end{equation}
there exists a unit vector

\begin{equation}
    v\in V_N\cap
    \bigl(U_N+\operatorname{span}\{m,M_Nm\}\bigr)^\perp.
\end{equation}
In particular,

\begin{equation}
    \langle v,m\rangle=\langle v,M_Nm\rangle=0,
    \qquad P^{U_N}v=0,
    \qquad v^\top M_Nv\ge\lambda_{N-d_N-2}.
\end{equation}
Set $m'=m+\sqrt{q'-q}\,v$. Then $|m'|^2=q'$ and $P^{U_N}m'=P^{U_N}m$, so
$f_N(m')=f_N(m)$. Moreover, the cross term vanishes:

\begin{equation}
    (m')^\top M_Nm'
    =m^\top M_Nm+(q'-q)v^\top M_Nv.
\end{equation}
It follows that

\begin{equation}
\begin{aligned}
    E_N(q')
    &\ge (m')^\top M_Nm'+f_N(m')\\
    &\ge E_N(q)+(q'-q)\lambda_{N-d_N-2}.
\end{aligned}
\end{equation}
The same inequality is immediate when $q=q'$.

Taking $q'=q^*$ and rearranging gives, for every $q\in[0,q^*]$,

\begin{equation}
\begin{aligned}
    E_N(q)-q\lM
    &\le E_N(q^*)-q^*\lM
       +(q^*-q)
         \bigl(\lM-\lambda_{N-d_N-2}\bigr)\\
    &\le E_N(q^*)-q^*\lM
       +q^*\bigl|\lM-\lambda_{N-d_N-2}\bigr|.
\end{aligned}
\end{equation}
Taking the supremum proves the claimed upper bound; the lower
bound follows by choosing $q=q^*$.
Finally, Assumption~1 implies
$\lambda_{N-d_N-2}\to\lM$, which completes the proof of the lemma.
\end{proof}

\begin{proof}[Proof of Theorem \ref{thm:main_with_ext_field}]

A direct consequence of Lemma \ref{lem:energy_at_norm_q_increases} is that the supremum in \eqref{eq:supremum_non_plefka}, up to a vanishing error, is reached at $|m|^2=q^*$. As it is equivalent to $\beta(1-|m|^2)=\beta_c$, this fact alongside \eqref{eq:Good_part_sup_external_field} and \eqref{eq:external_field_main_variational_UB_bis} yields


\begin{equation}\label{eq:external_field_final_UB}
    F_N^f(\beta)
    \les
    \sup_{m\in B_N: \ \beta(1-|m|^2)\les \beta_c}  F_{\TAP}(m)
    +o(1).
\end{equation}
To conclude, it suffices to prove the corresponding lower bound. We start by noting that $F_{\TAP}(m)\les \beta(2\Lambda+\gamma)+\frac{1}{2}\log(1-|m|^2)$, so there exists $\rho>0$ small enough that

\begin{equation}\label{eq:sup_m_less-1-rho}
    \sup_{m\in B_N: \ \beta(1-|m|^2)\les \beta_c}  F_{\TAP}(m)
    =
    \sup_{m\in (1-\rho) B_N: \ \beta(1-|m|^2)\les \beta_c}  F_{\TAP}(m).
\end{equation}
Next, remark that for any affine subspace $\mathcal{E}$, $m$ its point of smallest norm and $\overline{\mathcal{E}}$ its thickened version, Lemma \ref{lem:TAP_on_subspace} yields

\begin{equation}
    F_N^f(\beta)
    \ges
    F_N^f(\beta,\overline{\mathcal{E}})
    \ges 
    \frac1N\log\mathbb E_{\mathcal{E}}e^{\beta H_N^f(\sigma)}
    +
    \frac{1}{N}\log E\left[ \indf_{\overline{\mathcal{E}}} \right]    
    -
    \frac{2\beta(\Lambda+\gamma)\epsilon}{\sqrt{1-(|m|+\epsilon)^2}}
\end{equation}
We choose $\overline{\mathcal{E}}$ as follows. Fix $m$ satisfying $\beta(1-|m|^2)\le\beta_c$ and $|m|\le1-\rho$. Set $V=\operatorname{span}\{u,m,M_Nm\}^{\perp}$ and $n=\dim V$. On the slice $\mathcal{E}=m+V$, the external field is constant and the effective external field vanishes. Consequently, if $\nu$ denotes the empirical spectral distribution of the projection of $M_N$ to $V$, then

\begin{equation}
\begin{aligned}
    \frac1N\log E_{\mathcal{E}}e^{\beta H_N^f(\sigma)}
    =&
    \frac{\beta}{N}H_N^f(m)
    +
    \frac{n}{N}\frac1n\log E_{\mathcal{E}-m}\exp\left(\frac{N}{n}\beta\big(1-|m|\big)^2 H_n(\sigma)\right)\\
    \ges&
    \frac{\beta}{N}H_N^f(m)
    +\Ffin{\nu}\left(\beta (1-|m|^2) \right)
    -
    \frac{6}{N}\beta\Lambda,
\end{aligned}
\end{equation}
where we used boundedness and Lipschitzness to get rid of the $N/n$ terms. Since $N-n\le3$, the Cauchy interlacing Theorem and Theorem \ref{thm:sec2_main} imply, uniformly in $m$,

\begin{equation}
    \Ffin{\nu}\left(\beta (1-|m|^2)\right)
    =
    \Flim{\mu}\big(\beta (1-|m|^2)\big)+o(1).
\end{equation}
Moreover, Lemma \ref{lem:measure_sphere_LB} gives

\begin{equation}
    \frac{1}{N}\log E\left[ \indf_{\overline{\mathcal{E}}} \right]
    \geq
    \frac{1}{2}\log\big(1-|m|^2\big)
    -
    \frac{3\epsilon}{\rho}
    +
    \frac{3}{N}\log \epsilon
\end{equation}
Hence, choosing $\epsilon=1/N$ gives, uniformly in $m\in (1-\rho)B_N$,

\begin{equation}
\begin{aligned}
    F_N^f(\beta)
    \ges &
    \frac{\beta}{N}H_N^f(m)
    +\frac12\log\big(1-|m|^2\big)+\Flim{\mu}\big(\beta (1-|m|^2)\big)
    -o(1)\\
    \ges&
    F_{\mathrm{TAP}}(m)
    -o(1).
\end{aligned}
\end{equation}
Taking the supremum over $m$ with $|m|\leq 1-\rho$ and recalling \eqref{eq:sup_m_less-1-rho} proves

\begin{equation}\label{eq:external_field_final_LB}
    F_N^f(\beta)
    \ges
    \sup_{m\in B_N: \ \beta(1-|m|^2)\les \beta_c}  F_{\TAP}(m)
    -o(1),
\end{equation}
which together with \eqref{eq:external_field_final_UB} concludes the proof of Theorem \ref{thm:main_with_ext_field} for linear external field.

For the general external field, we may subtract $f_N(0)$,
since this shifts both sides of the claimed TAP formula by
$\beta f_N(0)$. Thus, assume $f_N(0)=0$, so that
$|f_N|\le c$ on $B_N$.

In the covering construction, include $U_N$ among the
constrained directions. The resulting slice directions
satisfy $V\subset U_N^\perp$, and hence
$f_N(m+v)=f_N(m)$ on every slice.
On its $\epsilon$-thickening,
\begin{equation}
    |f_N(\sigma)-f_N(m)|\le c\epsilon.
\end{equation}
The preceding estimates therefore apply with the Lipschitz
constant $c$ in place of $|\gamma|$.
\end{proof}

\section{Notation}

To help the reader, in this section we compile most of the quantities defined in the paper, and point to the equation where they are defined.

\begin{center}
\begin{tabular}{|c|c|| c|c|}
    \hline
     $H_N(\sigma)$ & \eqref{eq:def_HN} & $H_N^f(\sigma)$ & \eqref{eq:def_HN_ext_field}  \\
     \hline
     $\lambda_i^{(N)}:=\lambda_i$ & \eqref{eq:def_eigenvalues_MN} & $\mu_N$ & Assumption \ref{main_assumption} \\
     \hline
     $C_{\epsilon}(\tau_L,\tau_R)$ & \eqref{eq:simplex_tauL_tau_R} & $\Lambda$ & \eqref{eq:def_Lambda} \\
     \hline
     $F_N(\beta)$ & \eqref{eq:def_FE_no_ext_field} & $F_N^f(\beta)$ & \eqref{eq:def_FE_with_ext_field} \\
     \hline
     $Z_N(\beta)$ & \eqref{eq:def_FE_no_ext_field} & $Z_N^f(\beta)$ & \eqref{eq:def_FE_with_ext_field} \\
     \hline
     $F_N(\beta,A)$ & \eqref{eq:def_FN_S} & $Z_N(\beta,A)$ & \eqref{eq:def_ZN_A} \\
     \hline
     $S_{\mu}(\lambda)$ & \eqref{eq:def_stieltjes_transform} & $L_{\mu}(x)$ & \eqref{eq:def_log_potential} \\
     \hline
     $L_{\mu}^*(x)$ & \eqref{eq:LegendreLstar} & $\vf_{\nu}(z;\beta)$ & \eqref{eq:vf_mu_def} \\
     \hline
     $s_{\mu}$ & Assumption \ref{main_assumption} & $\beta_c$ & \eqref{eq:def_beta_c} \\
     \hline
     $q^*$ & \eqref{eq:def_q_star} & $I(1-\alpha)$ & \eqref{eq:eqdef_Ialpha} \\
     \hline
     $F_{\mu}(\beta)$ & \eqref{eq:F_def} & $\On(\beta,q)$ & \eqref{eq:OnsagerDefinition} \\
     \hline
     $f_{\TAP}(\beta,E,q)$ & \eqref{eq:ftapDefinition} & $F_{\TAP}(m)$ & \eqref{eq:FTAP_first_def} \\
     \hline
     $\hat{F}_\nu$ & \eqref{eq:F_nu_def} & $W_1(\mu_N,\mu)$ & Wasserstein-$1$ distance \\
     \hline
     $\mathbb{Q}_{\eta}$ & \eqref{eq:Q_def} & $g_\nu(\eta)$ & \eqref{eq:def_g_mu} \\
     \hline
     $\tau_{L,K}(\sigma)$ & \eqref{eq:def_q_tauLR} & $q_K(\sigma)$ & \eqref{eq:def_q_tauLR} \\
     \hline
     $A_K$ & \eqref{eq:def_AK} & $A_{K,\epsilon}^{\tau_L,\tau_R}$ & \eqref{eq:defAkqlqrdelta} \\
     \hline
     $\pi_L$ & \eqref{eq:def_set_piL_piR} & $\pi_R$ & \eqref{eq:def_set_piL_piR}\\
     \hline
     $D(\alpha||\tau_L,\tau_R)$ & \eqref{eq:def_entropy} & $\mathfrak{S}_{\beta}$ & \eqref{eq:def_sup_set_sec5} \\
     \hline
     $M_{N,K}(\tau_L,\tau_R)$ & \eqref{eq:def_MNK} & $\rho_{N,K}$ & \eqref{eq:qL_qR_density} \\
     \hline
     $C_{N,K}$ & \eqref{eq:C_NK_def} & $\fTAPalpha(\beta,\tau_L,\tau_R,\alpha)$ & \eqref{eq:def_ftap} \\
      \hline
     $\mu_{L,\alpha}$ & \eqref{eq:eqdef_muLalpha_muRalpha} & $\mu_{R,\alpha}$ & \eqref{eq:eqdef_muLalpha_muRalpha} \\
     \hline
     $\mathrm{Err}_N^{(*)} \big(\mu_N,\mu\big)$ & \eqref{eq:def_master_error} & $\mathrm{Err}_N^{(**)}\big(\mu_N\big)$ & \eqref{eq:def_master_error_bis} \\
     \hline
     $\lambda_i^{\mathcal{E}}$ & below \eqref{eq:hamiltonian_decompo} & $\overline{\mathcal{E}_{a_1}}$ & \eqref{eq:def_calE_alpha1} \\
     \hline
     $\mathcal{E}_{a_1,a_2}$ & \eqref{eq:def_calE_alpha12} & $\mathcal{E}_{a}$ & \eqref{eq:def_calE_alpha} \\
     \hline
     $\mathcal{E}_{a}'$ & \eqref{eq:def_calE_alphaprime} & $m_{a}$ & \eqref{eq:def_m_alpha} \\
     \hline
     $H_N^m(\sigma)$ & \eqref{eq:hamiltonian_decompo} & $A_K^{\mathcal{E}}$ & \eqref{eq:def_AK_calE} \\
     \hline
     $\tau_{L,K}^{\mathcal{E}}(\sigma)$ & \eqref{eq:def_tau_q_calE} & $\tau_{R,K}^{\mathcal{E}}(\sigma)$ & \eqref{eq:def_tau_q_calE} \\
     \hline
     $A^{\mathcal{E}}_{K,\tau_L,\tau_R,\epsilon}$ & \eqref{eq:def_A_tau_q_calE} & $F_{\TAP}(\beta,m,\tau_L,\tau_R,\alpha)$ & \eqref{eq:final_def_Ftap} \\
     \hline
     $\On(\beta,\tau_L,\tau_R,\alpha)$ & \eqref{eq:final_def_Ftap} & $S_{\mathcal{E}}$ & \eqref{eq:def_sup_set} \\
     \hline
     $\widehat{S}_{\mathcal{E}}$ & \eqref{eq:def_sup_set_discrete} & $E_N(q)$ & \eqref{eq:def_energy_2spin} \\
     \hline 
\end{tabular}
\end{center}

\printbibliography

@article{belius2022high,
  title={High temperature TAP upper bound for the free energy of mean field spin glasses},
  author={Belius, D.},
  journal={arXiv preprint arXiv:2204.00681},
  year={2022}
}

@article{BK,
  title = {The {{TAP}}\textendash{{Plefka}} Variational Principle for the Spherical {{SK}} Model},
  author = {Belius, David and Kistler, Nicola},
  year = {2019},
  volume = {367},
  pages = {991--1017},
  publisher = {{Springer}},
  journal = {Communications in Mathematical Physics},
  number = {3}
}

@article{genovese2015,
    title={Legendre duality of spherical and Gaussian spin glasses},
    author={Genovese, G. and Tantari, D.},
    journal={Mathematical Physics, Analysis and Geometry},
    volume={18},
    number={1},
    article={10},
    year={2015},
    pages={1--19}
}

@article{harman2010,
    title={On decompositional algorithms for uniform sampling from $n$-spheres and $n$-balls},
    author={Harman, R. and Lacko, V.},
    journal={Journal of Multivariate Analysis},
    volume={101},
    year={2010},
    pages={2297--2304}
}

@article{ghosh2021exponential,
  title={Exponential tail bounds for chisquared random variables},
  author={Ghosh, Malay},
  journal={Journal of Statistical Theory and Practice},
  volume={15},
  number={2},
  pages={35},
  year={2021},
  publisher={Springer}
}

@article{bogoya2016convergence,
  title={From convergence in distribution to uniform convergence},
  author={Bogoya, Johan Manuel and B{\"o}ttcher, Albrecht and Maximenko, Egor A},
  journal={Bolet{\'\i}n de la Sociedad Matem{\'a}tica Mexicana},
  volume={22},
  number={2},
  pages={695--710},
  year={2016},
  publisher={Springer}
}

@book{villani2021topics,
  title={Topics in optimal transportation},
  author={Villani, C{\'e}dric},
  volume={58},
  year={2021},
  publisher={American Mathematical Soc.}
}

@book{mezard1987spin,
  title={Spin glass theory and beyond: An Introduction to the Replica Method and Its Applications},
  author={M{\'e}zard, Marc and Parisi, Giorgio and Virasoro, Miguel Angel},
  volume={9},
  year={1987},
  publisher={World Scientific Publishing Company}
}

@book{talagrand2010mean,
  title={Mean field models for spin glasses: Volume I: Basic examples},
  author={Talagrand, Michel},
  year={2010},
  publisher={Springer Science \& Business Media}
}

@book{panchenko2013sherrington,
  title={The Sherrington-Kirkpatrick Model},
  author={Panchenko, Dmitry},
  year={2013},
  publisher={Springer Science \& Business Media}
}

@article{thouless1977solution,
  title={Solution of solvable model of a spin glass'},
  author={Thouless, David J and Anderson, Philip W and Palmer, Robert G},
  journal={Philosophical Magazine},
  volume={35},
  number={3},
  pages={593--601},
  year={1977},
  publisher={Taylor \& Francis}
}

@article {MR663708,
    AUTHOR = {Plefka, T.},
     TITLE = {Convergence condition of the {TAP} equation for the
              infinite-ranged {I}sing spin glass model},
   JOURNAL = {J. Phys. A},
  FJOURNAL = {Journal of Physics. A. Mathematical and General},
    VOLUME = {15},
      YEAR = {1982},
    NUMBER = {6},
     PAGES = {1971--1978},
      ISSN = {0305-4470,1751-8121},
   MRCLASS = {82A57},
  MRNUMBER = {663708},
       URL = {http://stacks.iop.org/0305-4470/15/1971},
}

@article {MR4588318,
    AUTHOR = {Chen, Wei-Kuo and Panchenko, Dmitry and Subag, Eliran},
     TITLE = {Generalized {TAP} free energy},
   JOURNAL = {Comm. Pure Appl. Math.},
  FJOURNAL = {Communications on Pure and Applied Mathematics},
    VOLUME = {76},
      YEAR = {2023},
    NUMBER = {7},
     PAGES = {1329--1415},
      ISSN = {0010-3640,1097-0312},
   MRCLASS = {82D30 (82B20)},
  MRNUMBER = {4588318},
MRREVIEWER = {Flora\ Koukiou},
       DOI = {10.1002/cpa.22040},
       URL = {https://doi.org/10.1002/cpa.22040},
}

@article{subag2023free,
  title={The free energy of spherical pure p-spin models: computation from the TAP approach},
  author={Subag, Eliran},
  journal={Probability Theory and Related Fields},
  volume={186},
  number={3},
  pages={715--734},
  year={2023},
  publisher={Springer}
}

@article {MR2132396,
    AUTHOR = {Guionnet, A. and Ma\"ida, M.},
     TITLE = {A {F}ourier view on the {$R$}-transform and related
              asymptotics of spherical integrals},
   JOURNAL = {J. Funct. Anal.},
  FJOURNAL = {Journal of Functional Analysis},
    VOLUME = {222},
      YEAR = {2005},
    NUMBER = {2},
     PAGES = {435--490},
      ISSN = {0022-1236,1096-0783},
   MRCLASS = {60F10 (46L54)},
  MRNUMBER = {2132396},
MRREVIEWER = {Sompong\ Dhompongsa},
       DOI = {10.1016/j.jfa.2004.09.015},
       URL = {https://doi.org/10.1016/j.jfa.2004.09.015},
}

@inproceedings{diaconis1987dozen,
  title={A dozen de Finetti-style results in search of a theory},
  author={Diaconis, Persi and Freedman, David},
  booktitle={Annales de l'IHP Probabilit{\'e}s et statistiques},
  volume={23},
  number={S2},
  pages={397--423},
  year={1987}
}

\end{document}